\documentclass[preprint,10.5pt,3p,authoryear]{elsarticle}
\usepackage{lmodern}
\usepackage[english]{babel}
\usepackage[ansinew]{inputenc}
\usepackage{sansmath}
\usepackage{amsmath}
\usepackage{amsthm}

\newtheorem{corollary}{Corollary}
\newtheorem{lemma}{Lemma}
\newtheorem{definition}{Definition}
\usepackage{enumitem}
\newlist{steps}{enumerate}{1}
\setlist[steps, 1]{label = Step \arabic*:}
\usepackage{xparse}%
\usepackage{xkeyval}%
\usepackage{array}[r]
\usepackage{eqparbox}[r]
\usepackage[noend]{algcompatible}

\usepackage{etoolbox}
\usepackage{cuted}
\usepackage{subfloat}
\usepackage{longtable}
\usepackage{graphicx}
\RequirePackage{subfig}
\usepackage{caption}
\usepackage[dvipsnames, svgnames, x11names]{xcolor}
\usepackage{siunitx}
\usepackage{mhchem}
\usepackage{float}
\usepackage{breqn}
\usepackage{bm}
\usepackage{natbib}
\usepackage{setspace} 
\usepackage{amsfonts,amssymb} 
\usepackage{multirow}
\usepackage{booktabs}
\usepackage[ruled,noend,norelsize,linesnumbered]{algorithm2e}
\usepackage{amssymb}
\usepackage{xspace}

\newtheorem{modl}{Model}
\newtheorem{prop}{Proposition}

\makeatletter
\AtBeginDocument{\def\@citecolor{Cerulean}}
\AtBeginDocument{\def\@linkcolor{Cerulean}}
\AtBeginDocument{\def\@urlcolor{Cerulean}}
\makeatother
\usepackage[colorlinks=true,citecolor=azure,linkcolor=red]{hyperref}
\usepackage[noend]{algcompatible}
\newcommand{\Fig}{\textcolor{Cerulean}{Figure}\space}
\newcommand{\T}{\textcolor{Cerulean}{Table}\space}
\newcommand{\eq}{\textcolor{Cerulean}{Eq.}}
\newcommand{\ineq}{\textcolor{Cerulean}{Eq.}}

\newcommand{\Lemma}{\textcolor{Cerulean}{Lemma}\space}
\newcommand{\Alg}{\textcolor{Cerulean}{Algorithm}\space}

\newcommand{\s}{\textcolor{Cerulean}{Section}\space}

\allowdisplaybreaks

\begin{document}

\begin{frontmatter}
\title{\LARGE \textbf{Incentive-compatible mechanisms for online resource allocation in mobility-as-a-service systems}}

\author[1,2]{Haoning Xi}
\author[1,3]{Wei Liu}
\author[1]{David Rey\corref{cor1}}
\ead{d.rey@unsw.edu.au}
\author[1]{S.Travis Waller}
\author[2]{Philip Kilby}
\cortext[cor1]{Corresponding author at: Research Centre for Integrated Transport Innovation, School of Civil and Environmental Engineering, University of New South Wales, Sydney, NSW
2052, Australia.}
\address[1]{Research Centre for Integrated Transport Innovation, School of Civil and Environmental Engineering, University of New South Wales, Sydney, NSW
2052, Australia}
\address[2]{Data61, CSIRO, Canberra ACT 2601, Australia}
\address[3]{School of Computer Science and Engineering, University of New South Wales, Sydney, NSW 2052, Australia}

\begin{abstract}
In the context of `Everything-as-a-Service', the transportation sector has been evolving towards user-centric business models in which customized services and mode-agnostic mobility resources are priced in a unified framework. Yet, in the vast majority of studies on Mobility as a Service (MaaS) systems, mobility resource pricing is based on segmented travel modes, e.g. private vehicle, public transit and shared mobility services. This study attempts to address this research gap by introducing innovative auction-based online MaaS mechanisms where users can bid for any amount of mode-agnostic mobility resources based on their willingness to pay and preferences. We take the perspective of a MaaS regulator which aims to maximize social welfare by allocating mobility resources to users. We propose two mechanisms which allow users to either pay for the immediate use of mobility service (pay-as-you-go), or to subscribe to mobility service packages (pay-as-a-package). We cast the proposed auction-based mechanisms as online resource allocation problems where users compete for MaaS resources and bid for travel time per trip. We propose (integer-) linear programming formulations to accommodate user bids based on available mobility resources in an online optimization approach. We show that the proposed MaaS mechanisms are incentive-compatible, develop customized online algorithms and derive performance bounds based on competitive analysis. Extensive numerical simulations are conducted on large scale instances generated from realistic mobility data, which highlight the benefits of the proposed MaaS mechanisms and the effectiveness of the proposed online optimization approaches.
\end{abstract}

\begin{keyword}
Auctions/bidding, Incentive-compatibility, Online resource allocation, Mobility-as-a-Service.
\end{keyword}

\end{frontmatter}

\section{Introduction}\label{model}

In recent years, the rapid evolution of the digital and sharing economy has brought significant changes in how various services are provided. The Everything-as-a-Service concept, such as Computing-as-a-Service or Platform-as-a-Service, is revolutionizing traditional models where people value experience over the possession of material commodity. In the transportation sector, the concept of Mobility as a Service (MaaS) is emerging and expected to shift mobility trends. MaaS has been defined as the provision of multimodal, demand-driven mobility services, offering customized travel options to users in real-time via a digital platform \citep{MuConsult}. \cite{hensher2017future} hypothesized that MaaS will shift public transportation contracts from the current output-based form (delivering kilometres on defined modes) to outcome-based models (delivering accessibility using any mode), thereby becoming mode-agnostic. Using multiple travel modes for a trip often requires users to make multiple payments to different transportation service providers (TSP). A central element of MaaS is to provide multimodal mobility services via a single payment \citep{hensher2020special}, thus facilitating user adoption and delivering fully integrated mobility ecosystems. In MaaS systems, two tariff options are typically available: Pay-as-You-Go (PAYG) and Pay-as-a-Package (PAAP). PAYG charges users for the immediate use of mobility services; whereas PAAP allows users to subscribe to mobility packages over a longer time period, such as a week or a month package \citep{ho2018potential}. Hence, PAYG aims to provide on-demand mobility solutions for users while PAAP is intended to reduce the marginal costs of mobility services by introducing longer time frames and offering more flexibility from a supply-side standpoint \citep{kamargianni2016comprehensive,matyas2019potential}. Compared with PAYG, PAAP has more potential to attract users to shared modes and public transit \citep{matyas2019potential}. However, as noted by \cite{ho2018potential}, in the PAAP model, users are charged a larger payment, which may negatively affect the practical attractiveness of this mechanism.

In this study, we explore the potential of MaaS solutions through the lens of auction-based mechanism design, which focuses on the identification of economic incentives to achieve targeted objectives \citep{haeringer2018market}. We take the perspective of a regulator, e.g. the local transportation authority, which aims to maximize social welfare via a MaaS platform. We consider a sequential decision-making framework where, at each time period, users bid for mobility resources and the regulator aims to accommodate users bids by strategically allocating mobility resources subject to resource availability constraints. We assume that when bidding users indicate their preferences including their origin, destination, travel delay budget and  inconvenience tolerance; as well as their willingness to pay (WTP) for mobility services. We model mobility resources as travel time per unit of distance and users have the possibility to place multiple bids for varying levels of mobility resources. This aims to provide a versatile MaaS system where users are able to bid for any amount of mobility resources in a mode-agnostic fashion. We cast these auction-based mechanisms as online resource allocation problems. To promote user adoption, we seek to develop MaaS mechanisms that are strategyproof, i.e. ensure that users truthfully report their WTP when bidding for mobility resources.

We next review the literature on MaaS systems (Section \ref{ecosystems}), auctions and incentive-compatibility in mobility services (Section \ref{auction}) and online resource allocation problems (Section \ref{resource}); before outlining the contributions of this study (Section \ref{contributions}).

\subsection{Overview of MaaS}
\label{ecosystems}

According to the MaaS Alliance \citep{MaaSalliance}, the concept behind MaaS is to ``\textit{put the users at the core of transport services, offering them tailored mobility solutions based on their preference and WTP}''. The goal of a MaaS platform is to provide a diverse menu of mobility options across multiple travel modes, including public transport, ride-, car- or bike-sharing, taxi; hereby referred to as a MaaS bundle. \cite{caiati2020bundling} and \cite{ho2018potential} investigated the problem of how to design MaaS bundles based on a stated preference survey. \cite{reck2020maas} developed concepts of the behavioral design for stated choice experiments for MaaS systems. \cite{ho2021drivers} evaluated the users' interest in various MaaS subscription bundles and identified key drivers of users' choices between PAAP and PAYG using Sydney trial data. \cite{djavadian2017agent} proposed an agent based stochastic day-to-day simulation for modelling MaaS in the two-sided flexible transportation system. \cite{ho2020public} explored the potential demand for MaaS under different business models such as monthly subscription and PAYG based on the stated choice experiments. Overall, the main features of MaaS systems discussed in the literature can be categorized into five categories as summarized in \T\ref{T1}.

\begin{table*}[tb]
\footnotesize
\setlength{\abovecaptionskip}{0pt}
\setlength{\belowcaptionskip}{0pt}
\centering
\caption{Main features of MaaS systems}
\begin{tabular}{lll}
\toprule
Features& Study
\\
\midrule
 Bundle different public and private transportation modes &\cite{caiati2020bundling,reck2020maas}\\
\hline customize a service based on a user's preference and WTP &\cite{ho2021drivers}\\
\hline Users can request or book a service from a digital third platform &\cite{matyas2019potential}\\
\hline Offer on-demand services and tariff& \cite{ djavadian2017agent} \\
\hline Provide two payment options: PAYG and PAAP &\cite{ho2021drivers,ho2020public}\\
\bottomrule
\end{tabular}
\label{T1}
\end{table*}

Although the concept of MaaS aims to be attractive, \cite{karlsson2020development} have shown through a pilot study conducted in Gothenburg that most users are not willing to change their travel habits and  purchase mobility resources via a dedicated MaaS platform. The authors concluded that the perception of users towards the ability of the MaaS platform to `match' their mobility requirements (in terms of service cost, transportation modes, etc.) is a decisive factor in users adoption of MaaS. This highlights that further research is needed to develop and promote practical MaaS solutions.

\subsection{Auctions and incentive-compatibility in mobility services}
\label{auction}

To provide customized mobility services, service providers must elicit private information from users such as their value of time or WTP. Auction theory and mechanism design have thus received an increasing attention in transportation for the purposes of allocating mobility resources to users. In particular, several incentive-compatible mechanisms which can promote truthful user bidding behavior have been developed to elicit private economic and mobility-driven information from users due to the fact that mobility services often involve private information games. We next review recent efforts in different types of mobility services.

In the content of transportation service procurement, \cite{xu2014efficient} proposed efficient auction-based mechanisms for the distributed transportation procurement problem, which can induce truthful bidding from carriers. \cite{zhang2019optimal} proposed an incentive-compatible multi-attribute transportation procurement auction which can minimize the total cost incurred from delivery delay, procurement and low service quality. In the content of public transit regulation, \cite{sun2020regulating} studied the regulation of a monopolistic public transit operator whose marginal cost is unknown to the regulator by presenting an incentive-compatible  regulatory policy which captures interactions between the regulator and the transit operator under asymmetric information provision. In the content of parking reservation problem, \cite{zou2015mechanism} proposed a mechanism design based approach for parking slot assignment and proved the incentive-compatibility in both static and dynamic mechanisms. \cite{shao2020parking} considered an auction-based parking reservation problem where a parking management platform is the auctioneer and drivers are bidders, and then propose an effective multi-stage Vickrey-Clarke-Groves auction mechanism, which is efficient and incentive-compatible. In the content of traffic intersection management, \cite{sayin2018information} studied a new information-driven intersection control method to enhance the quality of transportation by using communication between vehicles and roadside units. The authors proposed a strategy-proof intersection control method via a incentive-compatible payments which maximize social welfare. \cite{rey2021online} presented online mechanisms for auctions in which users bid for priority service, which are shown to be incentive-compatible in the dynamic sense. 

In the proposed MaaS mechanisms, users can bid for mobility resources in a continuous fashion and have the possibility to submit multiple request with different service quanlity request and WTP based on the multi-bids auction setting. Moreover, Due to the time sensitivity of the mobility resources and users, the proposed mechanism should provide a time-varying pricing strategy, instead of the static pricing.  Thus none of the existing approaches to show the incentive-compatibility in the literature can be totally applied into our problem due to  the different auction design and pricing strategies.

\subsection{Online resource allocation problems}
\label{resource}

The implementation of MaaS systems is based on real-time data acquisition and processing. In particular, MaaS systems aims to provide users the possibility to purchase mobility resources on-the-fly to meet their travel needs. This requires that mobility resources be allocated in an dynamic fashion when user requests are submitted to the MaaS platform.

Online resource allocation problems have received considerable attention in the Operations Research, Management Science and Computer Science literature. Online resource allocation problems have been extensively studied in the context of computing. \cite{buyya2002economic} proposed and developed a distributed computational economy-based framework for resource allocation to regulate supply and demand which incentivizes users to trade-off deadline, budget, and quality of service. \cite{Zhang2013}, \cite{shi2015online} and \cite{Zhou2017} proposed a truthful online cloud auction framework  for cloud computing and developed online algorithms which guarantee certain competitive ratios. \cite{xiao2018shared} studied two truthful double-auction mechanisms for a shared parking problem. They consider a parking platform with flexible schedules that aims to promote the typical daily 'go out early and come back at dusk' pattern. \cite{cohen2019overcommitment} introduced a model that quantifies the value of overcommitment in cloud computing and developed competitive online algorithms for solving their problem.

Several online resource allocation problems can be cast as knapsack problems \citep{marchetti1995stochastic}. \cite{zhou2008budget} and \cite{chakrabarty2008online} designed online algorithms for the knapsack problem which can achieve provably optimal competitive ratios. \cite{buchbinder2007online} designed online algorithms for fractional versions of the online knapsack problem with a guaranteed competitive ratio. \cite{wang2018multi} proposed an online algorithm that can minimize the total cost due to waiting, cancellations and overtime capacity usage and proved that the algorithm can reach the best possible competitive ratio  for this class of problems. Recently, \cite{asadpour2020online} studied a class of online resource-allocation problems over the different demand classes with limited flexibility and showed the effectiveness in mitigating supply-demand mismatch under a myopic online allocation policy. \cite{stein2020advance} studied a an online resource allocation problem with heterogeneous customers having specific preferences for each resource, and introduce online algorithms with bounded competitive ratios. 

In this study, we build on the existing literature for online resource allocation problems and develop customized online algorithms for allocating mobility resources in the proposed MaaS system. We next outline the contributions of this study relative to the literature.

\subsection{Our contributions}
\label{contributions}

In this study, we propose auction-based mechanisms for online resource allocation in MaaS systems. The proposed approach is based on the premise that mode-agnostic mobility resources can be regarded as continuous quantities. For instance, distance, travel time, and price can be regarded as continuous features of a mobility service. The main contributions of this paper are summarized as follows:
\vspace{-5pt}
\begin{itemize}[leftmargin=*]
\item We propose an innovative paradigm for MaaS systems in which users compete for mobility resources and bid for travel time per unit of distance in a mode-agnostic fashion. This provides a flexible MaaS framework wherein users have the possibility to purchase any amount of mobility resources to fulfill their travel demand.
\vspace{-5pt}
\item We consider two auction-based MaaS mechanisms, PAYG and PAAP, which capture different time resolutions and cast these mechanisms as online resource allocation problems. We develop mathematical programming formulations to optimize the allocation of mobility resources to users and show that both PAYG and PAAP mechanisms are incentive-compatible. To the best of our knowledge, this study is the first to propose auction-based mechanisms for MaaS systems which are incentive-compatible.
\vspace{-5pt}
\item
We develop customized primal-dual algorithms to solve the proposed online mobility resource allocation problems and derive bounded competitive ratios relative to an optimal offline problem. We also design rolling horizon algorithms that balance solution quality and computational tractability in the proposed auction-based MaaS mechanisms.
\vspace{-5pt}
\item We conduct numerical experiments to test the proposed mechanisms and illustrate their performance on a series of realistic scenarios. The results show that the proposed online algorithms are able to solve large-scale instances involving thousands of users in competitive time. This highlights the potential of the proposed approach to support the deployment of MaaS systems. 
\vspace{-5pt}
\end{itemize}

The rest of this paper is organized as follows: \s \ref{S2} introduces the auction-based mechanisms for MaaS systems; \s \ref{S4}  gives online mobility resources allocation formulations in MaaS system, tailors primal-dual online algorithms, derives the competitive ratio of the online algorithms and proves the incentive compatibility of the proposed online mechanisms. \s \ref{S6} designs a rolling horizon configurations; \s \ref{S7} conducts numerical experiments for illustration; \s \ref{S8} concludes the paper, provides the remarks and discusses future research directions.

\section{Auction-based mechanisms for MaaS systems}
\label{S2}

We propose two mechanisms which allow users to either pay for the immediate use of mobility service referred to as pay-as-you-go (PAYG), or to subscribe to mobility service packages referred to as pay-as-a-package (PAAP). We next give the motivation for designing such auction-based mechanisms for allocating mobility resources in MaaS systems.  

\subsection{Motivation}
\label{S2.1}
The global economic transition from ``commodity'' to ``service'' to ``experience'' has changed the way services are delivered to users. In this context, the transport sector is experiencing a vast revolution brought by MaaS. This motivates the design and the operation of MaaS systems where users can submit mode-agnostic mobility requests with their preferences in terms of user-experience relevant factors (e.g. number of shared riders and extra in-vehicle travel time).

User mobility requests may be expressed in terms of travel distance and travel time. Consider a user seeking to use MaaS for her daily commute trip. The user may request different travel times and have different WTP for each trip request, along with a travel delay budget representing the maximum delay acceptable for this mobility service. In addition, to quantify the ``user experience'' of a multimodal trip, each travel mode can be assigned an inconvenience cost per unit of time, and users set preferences on their maximum acceptable inconvenience cost. Intuitively, the inconvenience cost of travel mode aims to capture discomfort in shared and public transportation modes \citep{bian2019mechanism}. Moreover, each user's maximum acceptable inconvenience cost is defined as inconvenience tolerance .

\begin{figure*}[b!]
\centering
\includegraphics[width=1\textwidth]{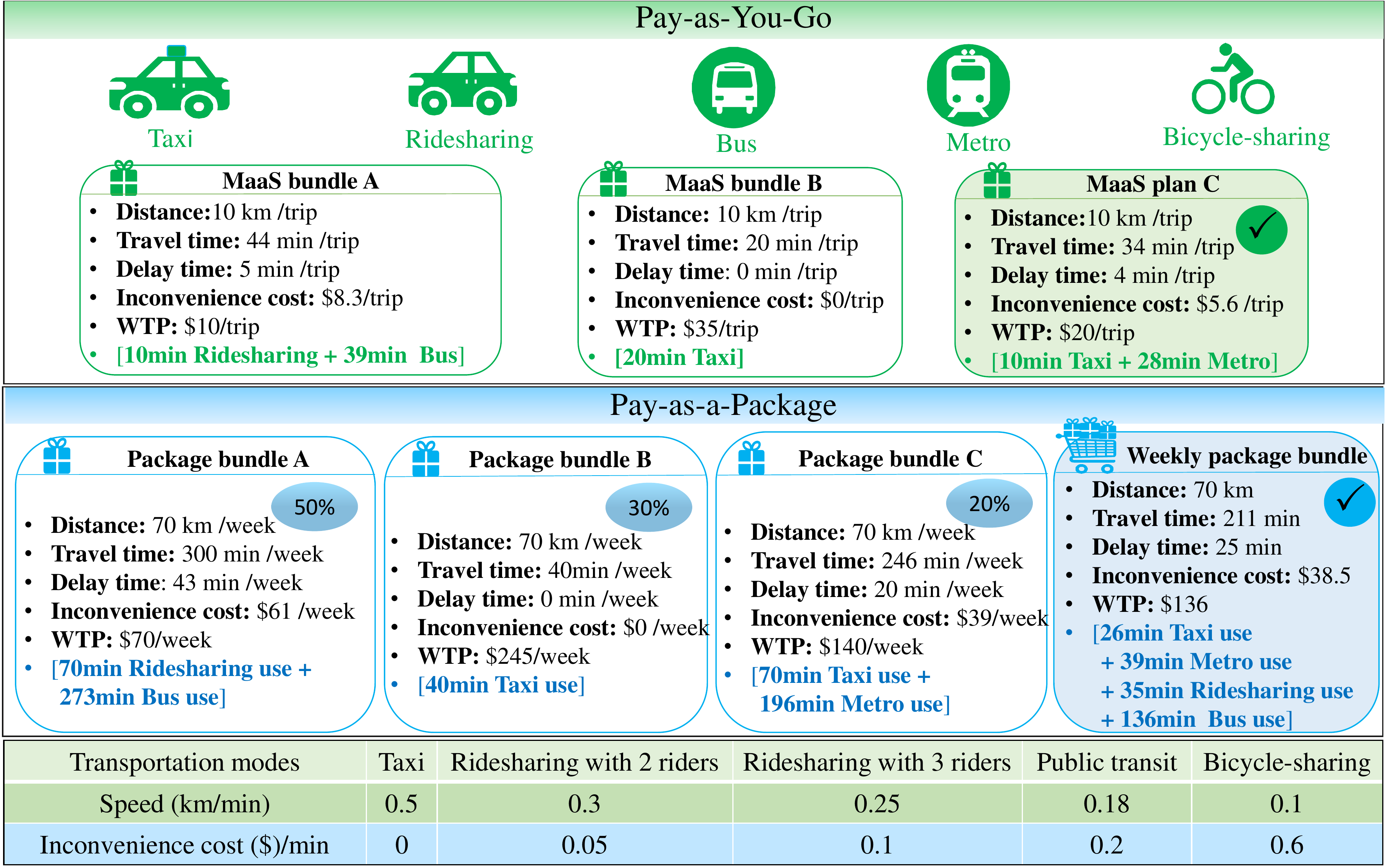}
\caption{Examples of the MaaS bundle in Pay-as-You-Go and Pay-as-a-Package mechanisms}\label{fig2}
\end{figure*}

The role of MaaS platforms is provide customized MaaS bundles to users according to their trip requests, preferences and WTP. We illustrate the proposed PAYG and PAAP mechanisms in \Fig \ref{fig2}. In the PAYG mechanism, a user whose inconvenience tolerance is \$10/trip and travel delay budget is 5 min/trip requests  a 10 km-service to the MaaS regulator. This user submits three bids for this service, each of which with a specific requested travel time and WTP: A(44 min, $\$10)$, B(20 min, $\$50)$ and C(34 min, $\$20)$. The PAYG part of \Fig \ref{fig2} illustrates how MaaS bundles which meet user preferences can be designed for each of these three bids by combining allocating mobility resources across multiple travel modes. In the PAYG mechanism, at most a single bid is accepted for each user, e.g. bid C. In the PAAP mechanism, a user whose inconvenience tolerance is \$70/week and travel delay budget is 45 min/week requests a 70 km-service. This user submits three bids, each of which with a specific requested travel time and WTP: A(300 min, $\$70$), B(40 min, $\$245$) and C(246 min, $\$140$). Unlike in the PAYG mechanism, the MaaS regulator can allocate mobility resources to multiple bids to design a mobility package, e.g. $50\%$ of A, $30\%$ of B and $20\%$ of C which corresponds to a travel time of 211 min and price of $\$136$.

\subsection{Problem statement}

This study takes the perspective of a MaaS regulator under government contracting\footnote{The government-contracted MaaS model has already been discussed in many existing studies, e.g., \citet{wong2020mobility}. One may refer to the literature for more detailed discussion regarding the  the role of government in an emerging MaaS future.} who integrates mobility resources from various TSPs, such as taxi companies, ridesharing companies, metro operators, bus companies and bicycle-sharing companies. Further, we assume that the MaaS regulator and the TSPs are under reselling contracts, in which TSPs are paid by the MaaS regulator to satisfy their reservation utility regardless of whether the provided mobility resources are utilized.\footnote{The reselling model has already been discussed by many in the literature for shared mobility, e.g., \citet{Zhang2020parking}. One may refer to the literature for more detailed discussion regarding the reselling model.} The interaction among the stakeholders of the proposed MaaS system is illustrated in \Fig\ref{1}. We assume that TSPs provide on-demand mobility services without fixed schedules and stops. Therefore, travel is assumed to be flexible with respect to time and space. In the proposed MaaS system, the average speed of different transport modes is assumed known and representative of the ``commercial speed'' of each mode. Specifically, we assume that service delays, e.g. waiting and transfer time, are already incorporated into the average speed of each mode; thus the average speed is the commercial speed rather than the driving speed on the road. 

To compare mobility services across travel modes, we introduce the concept of \textit{mobility resource} defined as speed-weighted travel distance. Using this concept, users' requests in terms of mobility services can be converted into mobility resources which can be themselves matched to mode-specific travel speed. A MaaS bundle is an allocation of mobility resources to a user. Since TSPs are limited in mobility service capacity, we assume that mobility capacity of the network is known and can be expressed in terms of mobility resources. We assume that users have preferences towards mobility services such as limits on mode-specific travel time, travel delay budget and WTP. To capture the preferences of users towards mobility services, we propose an auction model where users have the possibility to bid for one or multiple MaaS bundles. 

The problem addressed in this study is to determine the optimal allocation of mobility resources to users while accounting for users' preferences and resource capacity constraints. We consider that users arrive dynamically in the MaaS system and thus we adopt an online optimization approach to allocate mobility resources. We assume that the aim of the MaaS regulator is to maximize social welfare. Since users' preferences include private information such as users' WTP, the goal is to design incentive-compatible, auction-based mechanisms for MaaS systems.

\begin{figure*}[tb]
\centering
\includegraphics[width=0.8\textwidth]{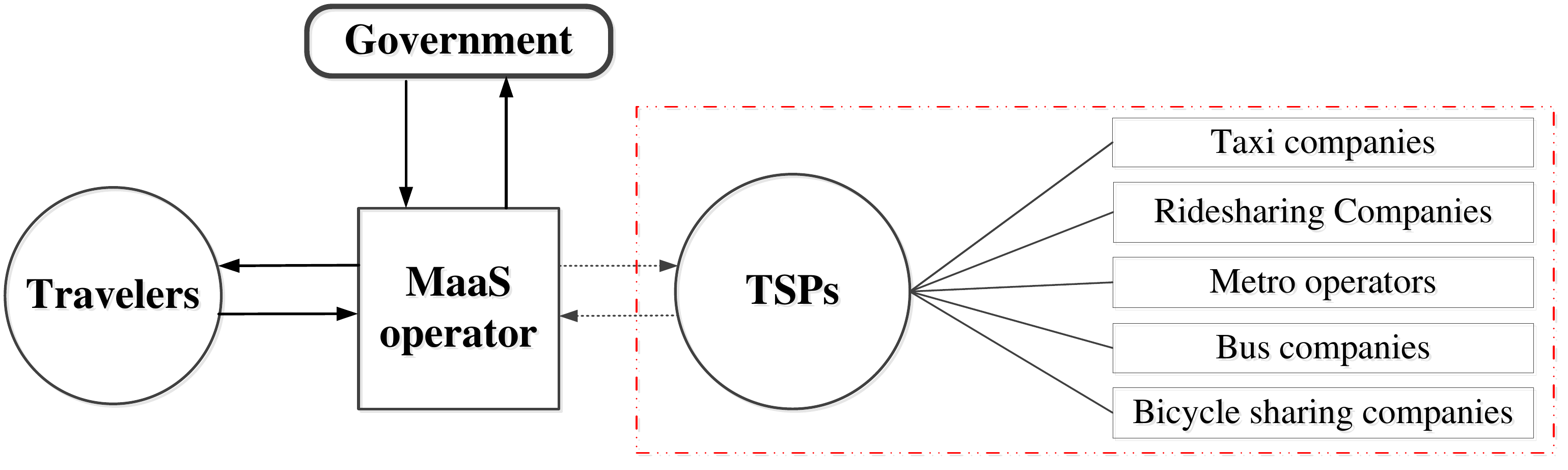}
\caption{Illustration of a MaaS system with a regulator under government contracting}\label{1}
\end{figure*}

We next present the auction model in \s \ref{oas} and resource pricing mechanisms in \s \ref{S3}.

\subsection{Auction model}
\label{oas}

In the proposed MaaS system, the auctioneer is the MaaS regulator and the bidders are users which seek to use mobility services to fulfill some or all of their travel needs. We propose a sealed-bid online auction, where a bidder privately submits one or multiple bids to the auctioneer. We consider a discrete time process and denote $\Omega$ the set of time slots in the auction. At each time slot, users bid for mobility services and the MaaS regulator determines which bids are served by allocating mobility resources subject to resource availability constraints.

We denote $\mathcal{I}(t)$ the set of users bidding at time slot $t\in\Omega$. For each user $i\in\mathcal{I}(t)$, we denote $\mathcal{J}_i$ the set of user $i$'s bids with $|\mathcal{J}_i| \geq 1$ ($\mathcal{J}_i$ is a singleton if user $i$ submits a single bid). We assume that users bid for travel times. For each bid $j \in \mathcal{J}_i$ of user $i \in \mathcal{I}(t)$, we denote $T_{ij}$ the requested travel time corresponding to this bid, and we denote $b_{ij}$ the WTP of user $i$ for service $j$. For each user bid $j \in \mathcal{J}_i$, $i \in \mathcal{I}(t)$, the corresponding mobility resources are defined as $Q_{ij} \triangleq D_i^2/T_{ij}$ which is expressed in speed-weighted travel distance units. The total capacity of the MaaS system is denoted $C$ and at each time slot $t \in \Omega$ the available capacity $A_t \leq C$ is determined based on the allocated mobility resources at previous time slots. 

The interpretation of the proposed auction-based model differs between the PAYG and the PAAP mechanisms. In the PAYG mechanism, users are assumed to bid for the immediate use of mobility services to complete a specific trip. We denote $D_i$ the travel distance corresponding to user $i$'s trip request and we denote $O_i$ the requested departure time of user $i$. We assume that $D_i$ is fixed and representative of the shortest path distance for this trip and that time $O_i$ occurs in the near future, e.g. within hour or day. Hence, in the PAYG mechanism, at most one bid of each user is accepted and allocated mobility resources, while other bids are rejected. In the PAAP mechanism, users are assumed to bid for mobility services over a pre-defined time period of length $L_i$, e.g. a week. In this mechanism, $D_i$ represents the requested travel distance of user $i$ over the time period $[O_i, O_i+L_i]$. Unlike in the PAYG mechanism, in the PAAP mechanism, users' bid may be partially served and we assume that the corresponding mobility resources are uniformly allocated over the time period $[O_i, O_i+L_i]$.


\subsection{Resource pricing and user payment}
\label{S3}

To price mobility services, we assume that at each time slot $t\in\Omega$ the MaaS regulator determines its unit price, denoted $p_t$, as a function of users bids. Let $b_{ij}/Q_{ij}$ be the unit bidding price for bid $j \in \mathcal{J}_i$ of user $i \in \mathcal{I}(t)$. At time slot $t$, the unit price $p_{t}$ of the MaaS system is set based on a pricing function such that $p_{t}\in \left[b_{\min},b_{\min}+b_{\max}\right]$, with $b_{\min}=\min _{i \in I(t), j \in J_{i}}\left\{\frac{b_{ij}}{Q_{ij}}\right\}$ and $b_{\max}=\max _{i \in I(t), j \in J_{i}}\left\{\frac{b_{ij}}{Q_{ij}}\right\}$.

We consider three types of unit price functions: linear, quadratic and exponential. Let $z_{t}=C-A_{t}$ be the allocated mobility resources at time slot $t$. The unit price at time slot $t$ $p_{t}$, can be determined based on the allocation of mobility resources at the previous time slot, using one of the following three functions.

\begin{itemize}[leftmargin=*]
\vspace{-5pt}
\item The linear unit price function is given in \eq\eqref{eq4}:
\begin{equation}
p_{t}^{\mathsf{lin}} (z_{t-1}) = \frac{b_{\max}}{C}z_{t-1}+b_{\min}.
\label{eq4} 
\end{equation}
\item The quadratic unit price function is given in \eq\eqref{eq6}:
\begin{equation}
p_{t}^{\mathsf{quad}} (z_{t-1}) =\frac{1}{C^{2}} (z_{t-1})^{2}+\frac{b_{\max}}{C} z_{t-1}+b_{\min}. \label{eq6} 
\end{equation}

\item The exponential unit price function is given in \eq\eqref{eq8}:
\begin{equation}
p_{t}^{\mathsf{exp}} (z_{t-1})=\frac{b_{\max}}{\alpha_{t}-1}\cdot\left(\alpha_{t}^{\frac{z_{t-1}}{C}}-1\right )+b_{\min},\label{eq8}
\end{equation}

where $\alpha_{t}$ is mechanism-dependent. For PAYG, $\alpha_{t}=(1+\overline{R}_{t})^{\frac{1}{\overline{R}_{t}}}$, with $\overline{R}_{t}=\max_{i\in \mathcal{I}(t)}\left\{\frac{\overline{Q}_{i}}{A_{t}}\right\}$ and  $\overline{Q}_{i}=\max_{j\in\mathcal{J}_{i}} \left\{Q_{ij}\right\}, \forall i\in \mathcal{I}(t)$. For PAAP, $\alpha_{t}=(1+\underline{R}_{t})^{\frac{1}{\underline{R}_{t}}}$, with $\underline{R}_{t}=\min_{i\in \mathcal{I}(t)}\left\{\frac{\underline{Q}_{i}}{A_{t}}\right\}$ and  $\underline{Q}_{i}=\min_{j\in\mathcal{J}_{i}} \left\{Q_{ij}\right\}, \forall i\in \mathcal{I}(t)$.
\end{itemize}

Using one of the pricing functions $p_{t+1} (z_t)$ from \eq\eqref{eq4}-\eq\eqref{eq8}, the payment of user $i$ for bid $j$ is:
\begin{equation}
p_{ij} = Q_{ij} p_{t} (z_{t-1}),\quad \forall i \in \mathcal{I}(t), j\in\mathcal{J}_{i}. 
\label{eq3} 
\end{equation}

\subsection{Auction process}
\label{process}

At each time slot $t\in\Omega$, the proposed auction consists of the following steps: \\

\noindent$\triangleright$ \textbf{Step 1:} \textbf{Auction set up.} Identify the available mobility resources $A_t$ and the set of users $\mathcal{I}(t)$ and their bids, and the unit price $p_t$. Select a unit pricing function (e.g., \eq\eqref{eq4},\eqref{eq6}, \eqref{eq8}) and determine user payments for each user bid based on \eq\eqref{eq3}. 

\noindent$\triangleright$ \textbf{Step 2:} \textbf{Mobility resource allocation}. Solve a resource allocation problem to identify the optimal allocation of mobility resources to users subject to resource availability constraints. This step is discussed in detail in Section \ref{S4}.

\noindent$\triangleright$ \textbf{Step 3:} \textbf{User payment and resource availability update}. For each user bid served by the MaaS system, charge the corresponding user payment based on \eq\eqref{eq3}. Update mobility resource availability and the unit price for the next time slot.

A detailed numerical example is provided in \textcolor{Cerulean}{Supplementary material A} to illustrate the proposed auction-based mechanisms. We next present optimization methods to solve the online mobility resource allocation problems which arise at Step 2 of the auction process.

\section{Online mobility resource allocation}
\label{S4}

In this section, we presents formulations and algorithms for online mobility resource allocation in MaaS systems. We first introduce mathematical programming formulations for the online resource allocation problems considered in \s \ref{PAYGra}. We then develop primal-dual online algorithms in \s \ref{PAYGpd} and conduct a competitive analysis of these online algorithms in \s \ref{PAYGcr}.

\subsection{Mathematical programming formulations}
\label{PAYGra}

We first propose a mathematical programming formulation for the online mobility resource allocation problem corresponding to the PAYG mechanism; We then explain how this formulation can be adapted for the PAAP mechanism.

Recall that $\mathcal{I}(t)$ is the set of users participating in the auction at time slot $t$, and that $\mathcal{J}_{i}$ is the set of bids of user $i$. For each user $i \in \mathcal{I}(t)$, and bid $j \in \mathcal{J}_{i}$, let $x_{ij}$ be a binary variable denoting whether bid $j$ is allocated mobility resources (1) or not (0). Recall that $Q_{ij}$ represents the mobility resources requested for bid $j$ of user $i$ and that $A_t$ is the available mobility resources at time slot $t$. The resource availability constraint ensures that the total amount of allocated mobility resources does not exceed the system capacity at time $t$:
\begin{equation}
\sum_{i \in \mathcal{I}(t)} \sum_{j \in \mathcal{J}_{i}} Q_{ij} x_{ij} \leq A_{t}.
\end{equation}

In the PAYG mechanism, for each user $i \in \mathcal{I}(t)$, at most one bid may be accepted, thus $\sum_{j \in \mathcal{J}_i} x_{ij} \leq 1$. Further, user bids can only be accepted if users' WTP is greater or equal to the market price determined by the pricing function corresponding to the requested mobility resources. Specifically, we require:
\begin{equation}
x_{ij}(b_{ij} - p_{ij}) \geq 0, \quad \forall i \in \mathcal{I}(t), j \in\mathcal{J}_i.
\end{equation}

Thus, if $b_{ij} < p_{ij}$ then $x_{ij} = 0$, otherwise $x_{ij} \geq 0$.

Let $\mathcal{M}$ be the set of travel modes available in the MaaS system. To map allocation decisions to mobility resources, we introduce a real positive variable $l_{ij}^m$ which represents the travel time allocated to mode $m \in \mathcal{M}$ if bid $j$ of user $i$ is accepted. Let $v_m$ be the travel speed of mode $m \in \mathcal{M}$. Allocation decisions $x_{ij}$ are linked to variable $l_{ij}^m$ via the constraint:
\begin{equation}
\sum_{m \in \mathcal{M}} v_{m} l_{ij}^{m} = D_{i} x_{ij}, \quad \forall i \in \mathcal{I}(t), j \in\mathcal{J}_i.
\end{equation}

If $x_{ij}=1$ for some bid $j \in\mathcal{J}_i$, then the requested travel distance $D_i$ must be distributed across speed-weighted, mode-based travel times; otherwise, the vector of variables $\bm{l}_i = [l_{ij}^m]_{j \in \mathcal{J}_i, m\in \mathcal{M}}$ is null. We consider that users have a travel delay budget, denoted $\Phi_i$, which represents the maximum excess travel time relative to the requested travel times $T_{ij}$, for each bid $j \in \mathcal{J}_i$ of user $i \in \mathcal{I}(t)$. Combining this upper bound on travel delay along with the requirement that the allocated mobility resources be at least the requested travel time yields the constraints:
\begin{equation}
0\leq \sum_{m \in \mathcal{M}} l_{ij}^{m}-T_{ij}x_{ij}\leq \Phi_{i}, \quad \forall i \in \mathcal{I}(t), j \in\mathcal{J}_i.
\end{equation}

Further, we assume that users perceive travel time using different travel modes differently, which is a common assumption in the literature \citep{janjevic2020designing}. We use $\sigma_m$ to denote the inconvenience cost of mode $m \in \mathcal{M}$ per unit of time and use $\Gamma_i$ to denote the inconvenience tolerance of user $i \in \mathcal{I}(t)$. We require:
\begin{equation}
\sum_{m \in  \mathcal{M}} \sigma_{m} l_{ij}^{m} \leq \Gamma_{i}, \quad \forall i \in \mathcal{I}(t), j \in\mathcal{J}_i.
\end{equation}

We assume that the goal of the MaaS regulator is to maximize social welfare defined as the sum of consumer surplus and the revenue of TSPs, that is:
\begin{equation}
\max\underbrace{\sum_{i \in \mathcal{I}(t)} \sum_{j \in \mathcal{J}_{i}} b_{ij} x_{ij}}_{\text{social welfare }}  =\underbrace{\sum_{i \in \mathcal{I}(t)} \sum_{j \in \mathcal{J}_{i}}(b_{ij}-p_{ij})x_{ij}}_{\text{consumer surplus}}  
+ \underbrace{\sum_{i \in \mathcal{I}(t)} \sum_{j \in \mathcal{J}_{i}}p_{ij}x_{ij}}_{\text{TSPs' total revenue}}
\end{equation}

The resulting mixed-integer linear programming (MILP) formulation for the PAYG online mobility resource allocation problem is summarized in \textcolor{Cerulean}{Model 1.1}. \\ 

\noindent \textbf{Model 1.1} (PAYG online mobility resource allocation).
\begin{small}
\label{mod1}
\begin{subequations}
\label{mm1}
\allowdisplaybreaks
\begin{align}
&\max \sum_{i \in \mathcal{I}(t)} \sum_{j \in \mathcal{J}_{i}} b_{ij} x_{ij},\label{1a}\\
&\text{subject to:}  && \nonumber \\
&\sum_{m \in \mathcal{M}} v_{m} l_{ij}^{m} =  D_{i} x_{ij},&& \forall i \in \mathcal{I}(t), j \in \mathcal{J}_i,\label{1b}\\
&0\leq \sum_{m \in \mathcal{M}} l_{ij}^{m}-T_{ij}x_{ij}\leq \Phi_{i},&& \forall i \in \mathcal{I}(t),j \in \mathcal{J}_{i},\label{1c}\\
&\sum_{m \in  \mathcal{M}} \sigma_{m} l_{ij}^{m} \leq \Gamma_{i},&& \forall i \in \mathcal{I}(t), j \in \mathcal{J}_{i},\label{1d}\\
&x_{ij}\left(b_{ij}-p_{ij}\right) \geq 0, && \forall i \in \mathcal{I}(t), j \in \mathcal{J}_{i},\label{1e}\\
&\sum_{i \in \mathcal{I}(t)} \sum_{j \in \mathcal{J}_{i}} Q_{ij} x_{i j} \leq A_{t},\label{1f}\\
&\sum_{j \in \mathcal{J}_{i}} x_{i j} \leq 1, &&\forall i \in \mathcal{I}(t),\label{1g}\\
&l_{ij}^{m} \geq 0 && \forall i \in \mathcal{I}(t), j \in \mathcal{J}_{i}, m \in \mathcal{M},\label{1h}\\
&x_{ij}\in\{0,1\},&& \forall i \in \mathcal{I}(t), j \in \mathcal{J}_{i}.\label{1i}
\end{align}
\end{subequations}
\end{small}

 
In the PAAP mechanism, users have the possibility to purchase mobility packages over a pre-defined period of time. Hence, in the PAAP mechanism time slots are intended to represent longer time periods compared to those used in the PAYG mechanism, e.g. one day vs one minute. In this context, the decisions to allocate mobility resources to users' requests, $[x_{ij}]_{i \in \mathcal{I}(t), j \in \mathcal{J}_i}$, can be modeled as continuous variables as opposed to binary variables. Therefore the online mobility resource allocation problem in the PAAP mechanism is formulated as the linear programming (LP) relaxation of \textcolor{Cerulean}{Model 1.1}. This formulation is given in \textcolor{Cerulean}{Model 1.2}.\\

\noindent \textbf{Model 1.2} (PAAP online mobility resource allocation).
\begin{subequations}
\allowdisplaybreaks
\begin{align}
&\max\sum_{i \in \mathcal{I}(t)} \sum_{j \in \mathcal{J}_{i}} b_{ij} x_{ij},\label{39a}\\
&\text{subject to:}  && \nonumber \\
&\text{\eqref{1b}-\eqref{1h}},\nonumber\\
& 0 \leq x_{ij}\leq 1,&& \forall i \in \mathcal{I}(t), j \in \mathcal{J}_{i}.\label{39i}
\end{align}
\end{subequations}\label{mod3}
\vspace{-15pt}

Note that the time resolution of the PAAP mechanism is intended to be significantly larger than that of the PAYG mechanism. Hence, although \textcolor{Cerulean}{Models 1.1} and \textcolor{Cerulean}{Models 1.2} have near identical formulations, from a practical standpoint they are designed to provide solutions which corresponds to different time frames.\\

Solving \textcolor{Cerulean}{Models 1.1} or \textcolor{Cerulean}{Models 1.2} corresponds to step 2 of the auction process outlined in Section \ref{process}. After solving the online resource allocation problem at time slot $t$, the available mobility resources at subsequent time slots $t' \geq t+1$ are updated (step 3). Specifically, for each user $i\in\mathcal{I}(t)$, the number of time slots affected by user $i$'s allocation is $N_i = \lceil\sum_{j \in \mathcal{J}_i}\sum_{m \in \mathcal{M}} l_{ij}^m x_{ij}\rceil$ and for each $t' \in [O_i, O_i+N_i]$, the available mobility resources $A_{t'}$ is decreased by $\sum_{j \in \mathcal{J}_i} Q_{ij} x_{ij}$.\\

To study the proposed online mobility resource allocation formulations, we start by showing that they can be reformulated as multidimensional knapsack problems (MKP). For this, we introduce the concept of user-based feasible MaaS bundles.

\begin{definition}
For any user $i \in \mathcal{I}(t)$ and user bid $j \in \mathcal{J}_{i}$, let $\bm{l}_{ij} = [l_{ij}^{m}]_{m \in \mathcal{M}}$. Let $\mathcal{S}_{ij}$ be the set defined as:
\begin{equation}
\mathcal{S}_{ij} \triangleq \left\{\bm{l}_{ij} \in \mathbb{R}^{|\mathcal{M}|} : \eqref{1b}-\eqref{1d}, b_{ij} \geq p_{ij}, x_{ij} = 1\right\}. 
\end{equation}
We say that $\mathcal{S}_{ij}$ is the set of feasible MaaS bundles corresponding to bid $j$ for user $i$. Further, for each user $i \in \mathcal{I}(t)$, we define $\mathcal{S}_{i} \triangleq \bigcup_{j \in \mathcal{J}_{i}} \mathcal{S}_{ij}$ as the set of feasible MaaS bundles for user $i$. 
\end{definition}


Observe that for any MaaS bundle $s \in \mathcal{S}_{ij}$, the corresponding mobility resources are $Q_{i,s}=Q_{ij}$ and the corresponding user bid is $b_{i,s}=b_{ij}$. 

\begin{lemma}\label{L1}
For each user $i \in \mathcal{I}(t)$ and for each MaaS bundle $s \in \mathcal{S}_{i}$, let $\chi_{i,s}$ be a binary variable representing the allocation of $s$ to $i$. Consider the compact integer program (IP): 

\begin{subequations}
\allowdisplaybreaks
\begin{align}
&\max \sum_{i \in \mathcal{I}(t)}\sum_{s \in \mathcal{S}_{i}} b_{i ,s}\chi_{i,s}, \label{2a}\\
&\text{\emph{subject to:}}  && \nonumber \\
&\sum_{i \in \mathcal{I}(t)}\sum_{s \in \mathcal{S}_{i}}  Q_{i,s}\chi_{i,s} \leq A_{t},\label{2b}\\
&\sum_{s \in \mathcal{S}_{i}} \chi_{i,s} \leq 1, &&\forall i \in \mathcal{I}(t),\label{2c}\\
& \chi_{i,s}\in \left\{0,1\right\}, && \forall i \in \mathcal{I}(t), s \in \mathcal{S}_{i}.\label{2d}
\end{align}\label{CP1}
\end{subequations}

\noindent The compact IP \eqref{CP1} is equivalent to \emph{\textcolor{Cerulean}{Model 1.1}}. Further, the LP-relaxation of \eqref{CP1} is equivalent to \emph{\textcolor{Cerulean}{Model 1.2}}.
\end{lemma}

\begin{proof} 
Observe that all MaaS bundles $s \in \mathcal{S}_{i}$ may be allocated to user $i$ since by definition $b_{ij} \geq p_{ij}$ for any $j \in \mathcal{J}_i$. Constraint \eqref{2c} imposes that at most one MaaS bundle is allocated to $i$, which is equivalent to the requirement that at most one user bid is accepted in {\textcolor{Cerulean}{Model 1.1}}. If $\chi_{i,s} = 0$ for all $s \in \mathcal{S}_{i}$, then no mobility resources are allocated to $i$ which is equivalent to $x_ij = 0$ for all $j \in \mathcal{J}_i$. Else, if $\chi_{i,s} = 1$, then there must exists a bid $j \in \mathcal{J}_i$ such that $s \in \mathcal{S}_{ij}$. Since, by definition, all MaaS bundles $s \in \mathcal{S}_{i}$ verify constraints $\eqref{1b}-\eqref{1d}$, any solution $[\chi_{i,s}]_{i\in \mathcal{I}(t), s \in \mathcal{S}_i}$ of \eqref{CP1} corresponds to a solution $[(x_{ij}, \bm{l}_{ij})]_{i\in \mathcal{I}(t), j \in \mathcal{J}_i}$ of {\textcolor{Cerulean}{Model 1.1}} and vice-versa. Further, since the compact IP \eqref{CP1} is equivalent to {\textcolor{Cerulean}{Model 1.1}}, the LP-relaxation of \eqref{CP1} is equivalent to the LP-relaxation of  {\textcolor{Cerulean}{Model 1.1}} which is {\textcolor{Cerulean}{Model 1.2}}.
\end{proof}

We next show that both of the proposed auction-based MaaS mechanisms are incentive-compatible (IC).

\begin{prop}
\label{theorem1}
The proposed PAYG and PAAP mechanisms are IC, in which bidding truthfully is a weakly dominant strategy.
\end{prop}
\begin{proof}

According to the payment rule of the mechanisms given in \eq(\ref{eq3}), user $i$'s payment for bid $j$ is independent of user $i$'s bidding price, and will be charged with the same payment, no matter how much the bidding price is. User $i$'s utility for bid $j$ is defined in \eq\eqref{ut}. 
\begin{equation}
u_{ij}=v_{ij}-p_{ij},\label{ut}
\end{equation}
where $v_{ij}$ denotes user $i$'s valuation (WTP) on bid $j$, $v_{ij}=b_{ij}$ holds in an IC mechanism, $p_{ij}$ denotes user $i$'s payment for bid $j$ defined in \eq(\ref{eq3}). 

In the PAYG mechanism, if user $i$ bids truthfully, her bidding price on bid $j$ is equal to her true valuation (WTP) on bid $j$, $b_{ij}=v_{ij}$; if bid $j$ is accepted, $x_{ij}=1$, her utility is $u_{ij}=v_{ij}-p_{ij}$; if bid $j$ is rejected, $x_{ij}=0$, her utility is $u_{ij}=0$. If user $i$ lies about her bidding price, $\hat b_{ij}\neq v_{ij}$; if bid $j$ is accepted, user $i$'s utility is still $\hat u_{ij}= v_{ij}-p_{ij}$; if bid $j$ is rejected, her utility is not greater than 0. We enumerate all possible cases to compare user $i$'s utility for bid $j$ with regards to truthful and non-truthful bidding behavior in \T \ref{T3}.
\begin{table*}[ht!]
\setlength{\abovecaptionskip}{0pt}
\setlength{\belowcaptionskip}{0pt}
\small
\centering
\caption{Utility comparison under truthful and non-truthful bidding}
\begin{tabular}{ccccccc}
\toprule
\multirow{2}*{ No. } &\multirow{2}*{ Cases } & \multicolumn{2}{c} {Non-truthful bidding} & \multicolumn{2}{c} {Truthful bidding} &\multirow{2}* { Utility comparison } \\
\cline { 3 - 6 } && $\hat x_{ij}$ & $\hat u_{ij}$ & $ x_{ij}$ & $ u_{ij}$  \\
\midrule
1&$v_{i j}<p_{i j}<\hat b_{i j}$ & $1$ & $v_{i j}-p_{i j}$ & $0$ & $0$ & $\hat u_{i j }< u_{i j}$ \\
 2&$v_{i j}<\hat b_{i j}<p_{i j}$ & $0$ & $0$ & $0$ & $0$ & $\hat u_{i j }= u_{i j }$ \\
3&$p_{i j}<v_{i j}<\hat b_{ij}$ & $1$ & $v_{i j}-p_{i j}$ & $1$ & $v_{i j}-p_{i j}$ & $\hat u_{i j}= u_{i j }$ \\
 4&$\hat b_{i j}<v_{i j}<p_{i j}$ & $0$ & $0$ & $0$ & $0$ & $\hat u_{i j}= u_{i j }$ \\
5&$\hat b_{ij}<p_{i j}<v_{i j}$ & $0$ & $0$ & $1$ & $v_{i j}-p_{i j}$ & $\hat u_{i j}< u_{i j}$ \\
6&$p_{i j}<\hat b_{i j}<v_{i j}$ & $1$ & $v_{i j}-p_{i j}$ & $1$ & $v_{i j}-p_{i j}$ & $\hat u_{i j}= u_{i j}$ \\
\bottomrule
\end{tabular}
\label{T3}
\end{table*}

\T \ref{T3} shows that user $i$'s utility for bid $j$ when bidding truthfully is not less than her utility when bidding non-truthfully, i.e. $u_{ij}\geq \hat u_{ij}$ holds under the proposed PAYG mechanism. Since bidding truthfully is weakly dominant strategy, the proposed PAYG mechanism is IC.

For PAAP mechanism, the dual problem of the compact LP is:
\begin{subequations}
\label{ddd}
\allowdisplaybreaks
\begin{align}
&\min\ A_{t} q(t)+\sum_{i \in \mathcal{I}(t)} u_{i},\label{Duala}\\
&\text{subject to:}  && \nonumber \\
&Q_{i,s} q(t)+u_{i} \geq b_{i,s}, &&\forall i \in \mathcal{I}(t), s \in \mathcal{S}_{i},\label{Dualb}\\
&q(t)\geq 0,\label{DualDualDualc}\\
&u_{i}\geq 0,&& \forall i \in \mathcal{I}(t),\label{Duald}
\end{align}\label{Dual}
\end{subequations}
where $q(t)$ and $u_{i}$ are the dual variables corresponding to constraints (\ref{2b}) and (\ref{2c}), respectively, and (\ref{Dualb}) is the dual constraint corresponding to the primal variable $\chi_{i,s}$ in the compact LP.\\

According to primal-dual complementary slackness, if $u_{i}>0$, constraint (\ref{2c}) is binding, $\sum_{s\in \mathcal{S}_{i}}\chi_{i,s}= 1, \forall i \in \mathcal{I}(t)$, namely, user $i$ will be allocated with the requested quantity of resources; else if $u_{i}=0$, then we have $\sum_{i\in \mathcal{I}(t)}\sum_{s\in \mathcal{S}_{i}}\chi_{i,s}\leq 1$, namely, user $i$ will be fractionally allocated with the requested quantity of resources. Moreover, if $\chi_{i,s}>0$, constraint (\ref{2b}) is binding, $u_{i}=\sum_{i \in \mathcal{I}(t)}\sum_{s \in \mathcal{S}_{i}}b_{i,s}-Q_{i,s} q(t)$; if $\chi_{i,s}=0$, $u_{i}\geq \sum_{i \in \mathcal{I}(t)}\sum_{s \in \mathcal{S}_{i}}b_{i,s}-Q_{i,s} q(t)$. Hence:
\begin{equation}
u_{i}^{*}=\max \left\{0, \left\{b_{i,s}-Q_{i,s} q^{*}(t)\right\}\right\}, \quad \forall i \in \mathcal{I}(t), s\in \mathcal{S}_{i}. \label{utility}
\end{equation}

where $u_{i}^{*}$ and $q^{*}(t)$ denote the optimal solutions of the dual problem \eqref{ddd}. If we interpret $q(t)$ as the unit price at time slot $t$, then  $Q_{i,s} q(t)$ can be interpreted as user $i$'s payment for package bundle $s$ and $u_{i}$ can be interpreted as user $i$'s utility obtained from package bundle $s$. According to \eq\eqref{utility}, if $u_{i}\geq 0$, the proposed PAAP mechanism maximizes user $i$'s utility. \\

Since we have proved that the maximum utility of each user holds at the optimal solution of the compact LP and its dual problem, we conduct a sensitivity analysis on user $i$'s bidding price ($b_{is}$) through the simplex tableau method to show IC. Without any loss of generality, we consider the case in which two users $i$ and $i+1$ bid for MaaS bundle ($s$) with bidding price $b_{i,s}$ and $b_{i+1,s}$. Assume that user $i$ is non-truthful, i.e., bids higher or lower than her valuation for MaaS bundle $s$ ($v_{i,s}$). The bid of user $i$ can be written as $\hat{b}_{i,s}=b_{i,s}+\Delta b$, with $b_{i,s}=v_{i,s}$. In turn, assume that user $i+1$ bids truthfully, i.e. $b_{i+1,s}=v_{i+1,s}$. The standardized formulation of this LP can be formulated as LP (\ref{STD}):
\begin{subequations}
\allowdisplaybreaks
\begin{align}
&\max b_{i,s}\chi_{i,s}+b_{i+1,s}\chi_{i+1,s}, \label{STDa}\\
&\text{subject to:}  && \nonumber \\
&Q_{i,s} \chi_{i,s}+Q_{i+1,s} \chi_{i+1,s}+y_{1} = A_{t},\label{STDb}\\
&\chi_{i,s}+y_{2}=1, \label{STDc}\\
&\chi_{i+1,s}+y_{3}=1, \label{STDd}\\
&\chi_{i,s},\chi_{i+1,s},y_{1},y_{2},y_{3}\geq 0,\label{STDe}
\end{align}\label{STD}
\end{subequations}
where $y_{1}$, $y_{2}$ and $y_{3}$ are the slack variables corresponding to constraints (\ref{STDb}),(\ref{STDc}) and (\ref{STDd}), respectively. The initial simplex tableau (\textcolor{Cerulean}{Table 3.1}) obtained from the standard formulation LP (\ref{STD}) can be converted into the final simplex tableau (\textcolor{Cerulean}{Table 3.2}) via standard methods.

\begin {table*}[ht!]
\centering
\setlength{\abovecaptionskip}{0pt}
\setlength{\belowcaptionskip}{0pt}
\caption{Simplex tableau before sensitivity analysis on $b_{i,s}$}
\footnotesize
\begin{tabular}{|c|c|c|c|c|c|}
\hline
\multicolumn{6}{|c|}{Table 3.1: Initial simplex tableau} \\
\hline $c$& $b_{i,s}$ & $b_{i+1,s}$ & 0 & 0 & 0\\
\hline basic var& $\chi_{i, s}$ & $\chi_{i+1, s}$ & $y_{1}$ & $y_{2}$& $y_{3}$ \\
\hline$y_{1}$ & $Q_{i,s}$ & $Q_{i+1,s}$ & 1 & 0 & 0\\
\hline $y_{2}$ & 1 & 0 & 0 & 1 & 0 \\
\hline $y_{3}$ & 0 & 1 & 0 & 0 & 1 \\
\hline $z$&0&0&0&0&0\\
\hline $\sigma=c-z$ & $b_{i,s}$&  $b_{i+1,s}$ & 0 & 0 & 0\\
\hline
\end{tabular}
$\Rightarrow$
\begin{tabular}{|c|c|c|c|c|c|}
\hline
\multicolumn{6}{|c|}{Table 3.2: Final simplex tableau}\\
\hline $c$ & $b_{i,s}$ & $b_{i+1,s}$ & 0 & 0&0 \\
\hline basic var& $\chi_{i, s}^{*}$ & $\chi_{i+1, s}^{*}$ & $y_{1}^{*}$ & $y_{2}^{*}$& $y_{3}^{*}$\\
\hline $\chi_{i,s}^{*}$&1 & 0 & 0&1&0\\ 
\hline $y_{3}^{*}$ & 0 & 0 & $\frac{-1}{Q_{i+1,s}}$ &$\frac{-Q_{i,s}}{Q_{i+1,s}}$ & 1\\
\hline $\chi_{i+1,s}^{*}$ & 0 & 1 & $\frac{1}{Q_{i+1,s}}$  & $\frac{Q_{i,s}}{Q_{i+1,s}}$& 0 \\
\hline $z$&$b_{i,s}$&$b_{i+1}$&$\frac{b_{i+1,s}}{Q_{i+1,s}}$&$\frac{b_{i,s}Q_{i+1,s}+b_{i+1} Q_{i,s}}{Q_{i+1,s}}$&0\\
\hline $\sigma=c-z$ & 0 & 0 & $-\frac{b_{i+1,s}}{Q_{i+1,s}}$ & $-\frac{b_{i,s}Q_{i+1,s}+b_{i+1,s} Q_{i,s}}{Q_{i+1,s}}$&0\\
\hline
\end{tabular}
\end{table*}
\begin{table*}[ht!]
\setlength{\abovecaptionskip}{0pt}
\setlength{\belowcaptionskip}{0pt}
\caption{Simplex tableau after sensitivity analysis on $b_{i,s}$}
\centering
\footnotesize
$\Rightarrow$
\begin{tabular}{|c|c|c|c|c|c|}
\hline $c$ & $b_{i,s}+\Delta b$ & $b_{i+1}$ & 0 & 0&0 \\
\hline basic variable & $\hat {\chi}_{i,s}$ & $\hat {\chi}_{i+1,s}$ & $\hat {y}_{1}$ & $\hat {y}_{2}$& $\hat {y}_{3}$\\
\hline  $\hat {\chi}_{i,s}$&1 & 0 & 0&1&0\\ 
\hline$\hat {y}_{3}$ & 0 & 0 & $\frac{-1}{Q_{i+1,s}}$ &$\frac{-Q_{i,s}}{Q_{i+1,s}}$ & 1\\
\hline$\hat {\chi}_{i+1,s}$ & 0 & 1 & $\frac{1}{Q_{i+1,s}}$  & $\frac{Q_{i,s}}{Q_{i+1,s}}$& 0 \\
\hline $z$&$b_{i,s}$&$b_{i+1,s}$&$\frac{b_{i+1,s}}{Q_{i+1,s}}$&$\frac{(b_{i,s}+\Delta b)Q_{i+1,s}+b_{i+1,s} Q_{i,s}}{Q_{i+1,s}}$&0\\
\hline $\sigma=c-z$ & 0 & 0 & $-\frac{b_{i+1,s}}{Q_{i+1,s}}$ & $-\frac{(b_{i,s}+\Delta b)Q_{i+1,s}+b_{i+1,s} Q_{i,s}}{Q_{i+1,s}}$& 0\\
\hline
\end{tabular}
\label{TT4}
\end{table*}
 
In \T\ref{TT4}, we conduct sensitivity analysis on $b_{i,s}$ by substituting  $b_{i,s}$ in  \textcolor{Cerulean}{Table 3.2} with $\hat{b}_{i,s}=b_{i,s}$+$\Delta b$. Since the reduced cost of $\hat{y}_{1}$ and $\hat{y}_{2}$ always hold non-positive values, $\sigma\left(\hat{y}_{1}\right)\leq 0$, $\sigma\left(\hat{y}_{2}\right)\leq 0$, the optimal solutions do not change with $\Delta b$. Then we investigate how the dual variables $\hat{u}_{i}$ and $\hat{q}(t)$ change with $\Delta b$. The reduced cost coefficients ($\sigma$) corresponding to the slack variables in the primal final tableau give the opposite value of dual variables. The value of the $i$th dual variable equals to the opposite value of the reduced cost coefficient of the slack variable associated with the $i$th primal constraint \citep{arora2004introduction}. Since the slack variable $\hat y_{1}$ in  \T\ref{TT4} corresponds to the dual variable $\hat q(t)$, we have $\hat{q}(t)=-\sigma(\hat {y}_{1})=\frac{b_{i+1,s}}{Q_{i+1,s}}$. In \textcolor{Cerulean}{Table 3.2} we have $q^{*}(t)=-\sigma({y}_{1}^{*})=\frac{b_{i+1,s}}{Q_{i+1,s}}$, the change value of ${q}(t)$ is: $\Delta q(t)=\hat{q}(t)-q^{*}(t)=0$.

Since $\hat q(t)$ can be interpreted as the unit price at time slot $t$, user $i$'s utility is $\hat u_{i}=v_{i,s}-Q_{i,s} \hat q(t)$, the change in user $i$'s utility ($\Delta u_{i}$) is 0. Moreover, the slack variable $y_{2}$ corresponding to the dual variable $u_{i}$, and thus we have: $\hat{u}_{i}=\sigma(\hat{y}_{2})=\frac{(b_{i}+\Delta b)Q_{i+1,s}+b_{i+1} Q_{i,s}}{Q_{i+1,s}}$, $u_{i}^{*}=\sigma(y_{2}^{*})=\frac{b_{i}Q_{i+1,s}+b_{i+1} Q_{i,s}}{Q_{i+1,s}}$.   
Note that user $i$'s true valuation on package bundle $s$ is $b_{i,s}$, instead of $b_{i,s}+\Delta b$, thus user $i$'s utility is $\hat{u}_{i}-\Delta b$, and the change in user $i$'s utility is : $\Delta u_{i}=\hat{u}_{i}-\Delta b-u_{i}^{*}=0.$

We have shown that if user $i$ bids higher than her true valuation ($\Delta b > 0$), her utility remains unchanged, whereas if she bids lower than this true valuation ($\Delta b < 0$), her utility may remain the same or be reduced. Hence bidding truthfully is a user's weakly dominant strategy under proposed the PAAP mechanism. 
\end{proof}

\subsection{Primal-dual online algorithms}
\label{PAYGpd}
 The compact integer program (IP) for online mobility resource allocation given in \Lemma \ref{L1} falls within the literature on the MKP \citep{bertsimas2002approximate}, which is is NP-hard if the decision variable is binary, and thus motivates us to design a heuristic algorithm. In this section, we propose customized primal-dual algorithms for the online mobility resource allocation problems at hand \citep{borodin2005online, buchbinder2009online}. The compact IP can be converted into a compact LP by relaxing the binary variable $\chi_{i,s}\in \left\{0,1\right\}$ to the continuous variable $\chi_{i,s}\geq 0$. The relaxed compact LP $OPT_{1}(t)$ and its dual problem $OPT_{2}(t)$ are summarized as follows:\\
$$\begin{array}{l|ll}
\multicolumn{1}{c|} {\text { \textbf{Primal online problem $OPT_{1}(t)$} }} & \multicolumn{1}{c} {\text { \textbf{Dual online problem $OPT_{2}(t)$}  }} \\
\hline \max \sum_{i \in \mathcal{I}(t)} \sum_{s \in \mathcal{S}_{i}} b_{i,s}\chi_{i,s}, & \min A_{t} q(t)+\sum_{i \in \mathcal{I}(t)} u_{i},\\
\text { subject to: } & \text { subject to: } \\
\sum_{i \in \mathcal{I}(t)}\sum_{s \in \mathcal{S}_{i}}  Q_{i,s} \chi_{i,s} \leq A_{t}, & Q_{i,s} q(t)+u_{i} \geq b_{i,s},\forall i \in \mathcal{I}(t),s \in \mathcal{S}_{i}, \\
\sum_{s\in \mathcal{S}_{i}} \chi_{i,s} \leq 1, \forall i \in \mathcal{I}(t), & q(t)\geq 0, \\
\chi_{i,s} \geq 0,\forall i \in \mathcal{I}(t), s \in \mathcal{S}_{i}, &u_{i}\geq 0, \forall i \in \mathcal{I}(t),
\end{array}$$

\noindent where the primal variable $\chi_{i,s}$ corresponds to the dual constraint; and where $q(t)$ and $u_{i}$ are the dual variables corresponding to the two constraints of the primal problem $OPT_{1}(t)$ at time slot $t$. Observe that $q(t)$ and $u_{i}$ can be interpreted as the unit price at time slot $t$ and user $i$'s utility. 
\begin{figure}[ht!]
\centering
\includegraphics[width=0.5\textwidth]{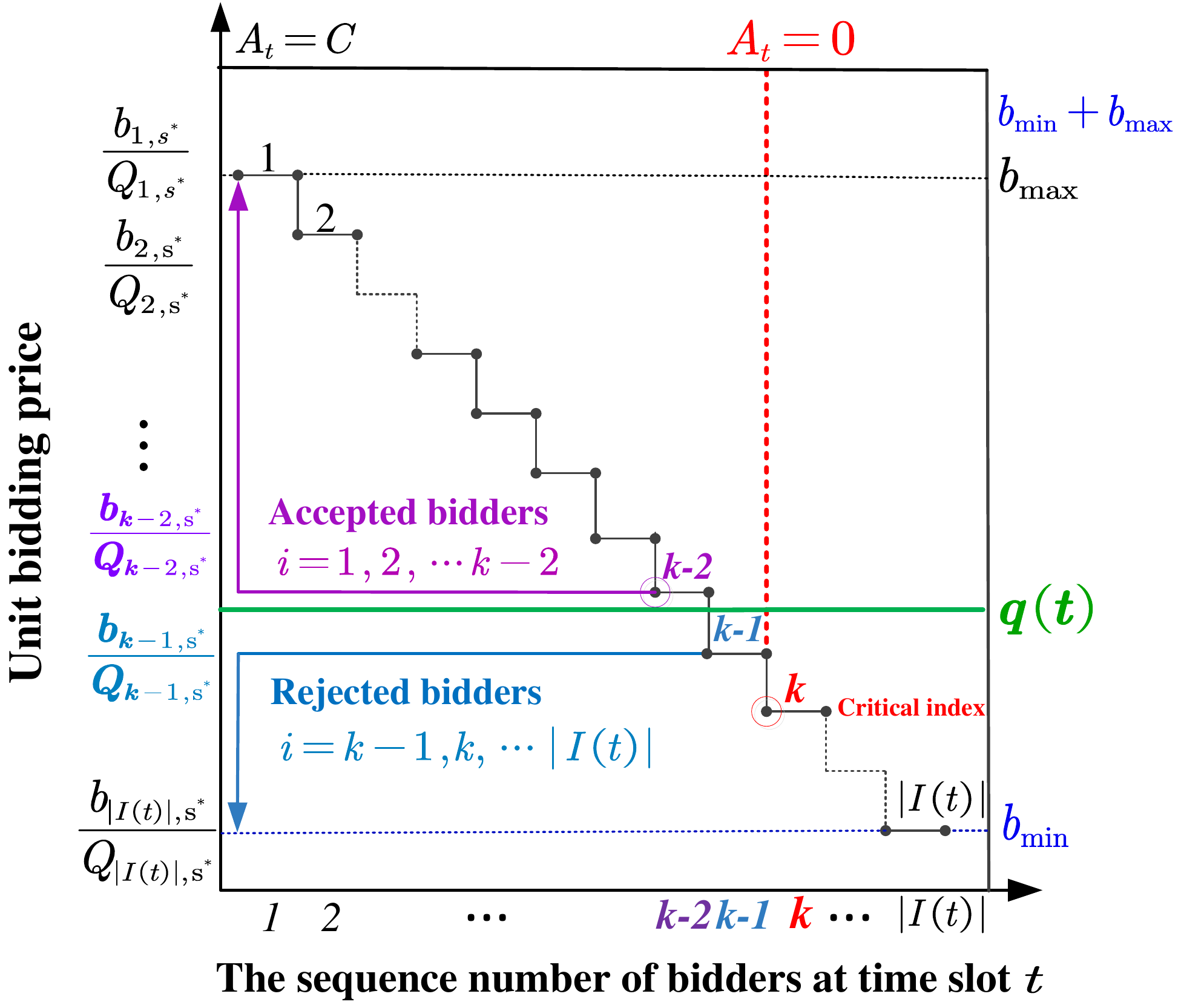}
\caption{Critical index selection and Acceptance/rejection determination}
\label{Fig3}
\end{figure}

\begin{algorithm}[!ht]
\small
\setstretch{1.0}
$\boldsymbol{A} [1,2,\cdots,\omega]\leftarrow C$ \\
\For{$t \in \Omega$}{
	$\overline{Q}_{i}\leftarrow \max_{j\in\mathcal{J}_{i}} \left\{Q_{ij}\right\}, \forall i\in \mathcal{I}(t)$\\
	  $j_{i}^{*}\leftarrow\arg \max_{j \in \mathcal{J}_{i}}\left\{\frac{b_{i j}}{Q_{ij}}\right\},\forall i\in \mathcal{I}(t)$\\
	$\bar{\mathcal{I}}(t) \gets$ sort $\mathcal{I}(t)$ by decreasing unit bidding price $b_{ij_i^*}/Q_{ij_i^*}$\\
	$\bar{\mathcal{J}}_i \gets$ sort $\mathcal{J}_{i}$ by decreasing unit bidding price $b_{ij}/Q_{ij}$\\
	select the critical index $k$ satisfying $\sum_{i=1}^{k-1} \overline{Q}_{i} \leq A_{t}<\sum_{i=1}^{k} \overline{Q}_{i}$\\
	$\bar{\mathcal{I}}_k(t) \gets [1,2,\cdots,k-1]$ \\
	$\overline{R}_{t}\leftarrow\max_{i\in \mathcal{I}(t)}\left\{\frac{\overline{Q}_{i}}{A_{t}}\right\}$\\
	$\overline{\alpha}_{t}\leftarrow(1+\overline{R}_{t})^{\frac{1}{\overline{R}_{t}}}$\\
	$A_{t}\leftarrow \boldsymbol{A}[t]$\\
	$q(t)\leftarrow 0$\\
	$\bm{x} \gets \bm{0}$\\
	\For{$i \in \bar{\mathcal{I}}_k(t)$}{
	    $\mathcal{J}'_i \gets \emptyset$ \\
		\For{$j \in \bar{\mathcal{J}}_{i}$}{
			\If{$q(t)\leq\frac{b_{ij}}{Q_{ij}}$ and $\mathcal{S}_{ij} \neq \emptyset$}{
		        $q(t)\leftarrow q(t)(1+\frac{\overline{Q}_{i}}{A_{t}})+\frac{b_{ij}}{(\overline{\alpha}_{t}-1)A_{t}}$\\
                $\mathcal{J}'_i \gets \mathcal{J}'_i \cup \{j\}$ 
		    }
		}
		\If{$\mathcal{J}'_i \neq \emptyset$}{
		    $j^\star \leftarrow\arg\max_{j \in \mathcal{J}'_i}\left\{b_{i j}-\overline{Q}_{i} q(t)\right\}$\\
		    $x_{ij^\star} \leftarrow 1$\\
		    $[l_{ij^\star}^m]_{m\in\mathcal{M}} \gets$ solve $\mathcal{S}_{ij^\star}$\\
	        $N_{i}\leftarrow\lceil\sum_{m \in \mathcal{M}} l_{ij^\star}^{m} \rceil$\\
    	    $\boldsymbol{A} \left[O_{i}:O_{i}+N_{i}-1\right]\leftarrow \boldsymbol{A} \left[O_{i}:O_{i}+N_{i}-1\right]-Q_{ij^\star}$ 
    	}
	}
}
\Return $\boldsymbol{x}$ and $\boldsymbol{q}$\\
\caption{PAYG primal-dual online algorithm}
\label{alg1}
\end{algorithm}

In \Alg \ref{alg1}, $\omega = |\Omega|$ denotes the last time slot, Line 1 initializes the weighted quantity of available resources at each time slot, Line 3 selects each user's maximum weighted quantity of mobility resources  among her multi-bids. Line 4 $\sim$ Line 8 show critical index selection in the  auction process, which are illustrated in \Fig\ref{Fig3}. Line 4 selects the bid with the maximum unit bidding price of each traveler, Line 5 sorts users by decreasing maximum unit bidding price, Line 6  sorts each user's bids by decreasing the unit bidding price, Line 7 determines the critical index $k$ based on the available mobility resources in time slot $t$ and Line 8 selects users $1, 2, \ldots, k-1$ as participants for the auction at time $t$. Line 9 and Line 10 define the parameters $\overline{R}_{t}$ and $\overline{\alpha}_{t}$. Line 12 $\sim$ Line 22 show the iterations of primal variables, dual variables in a time loop. The iteration rule of $q(t)$ (Line 18) is determined by the exponential unit price function in \eq(\ref{eq8}), Line 17 and Line 22 indicate the acceptance determination in the auction process, which is illustrated in \Fig\ref{Fig3}. Line 21 guarantees that each user is allocated with the bid $j^{*}$ that can maximize her utility. Line 23 solves a feasibility problem $\mathcal{S}_{ij}$ (an LP with 4 constraints, $|\mathcal{M}|$ variables and $|\mathcal{M}|$ bound constraints) to obtain the vector $[l_{ij^\star}^m]$, which denotes a feasible MaaS bundle. Line 24 counts the total number of time slots arranged for each user in a MaaS bundle. Line 25 updates the available mobility resources based on the resources allocated to user $i$, i.e. the available mobility resources $A_t$ are reduced by $Q_{ij^*} x_{ij^*}$ for all time slots $t \in [O_{i}, O_{i}+N_{i}]$ where $N_i$ is the number of time slots allocated to user $i$. 

\begin{lemma}
\label{L2}
Let $\mathcal{S}$ denotes the largest compact set of bundles across all users $i \in \mathcal{I}(t)$. The worst-case time complexity of \Alg \ref{alg1} is $\mathcal{O}(|\Omega| |\mathcal{S}||\mathcal{I}(t)|)$.
\end{lemma}
\begin{proof}
The operations in Lines 3 to 11 have a worst-case time complexity of $\mathcal{O}(|\mathcal{I}(t)| \log |\mathcal{I}(t)|)$ corresponding to the sort operation in Line 5. This is dominated by the \textbf{for} loop starting from Line 12 which requires $\mathcal{O}(|\mathcal{I}(t)||\mathcal{S}|)$ time. 
\end{proof}
\begin{lemma}\label{L3}
In each time iteration of \Alg \ref{alg1} , $\Delta \mathcal{P}(i) = (1-\frac{1}{\overline{\alpha}_{t}}) \Delta\mathcal{D}(i)$  holds, where $\Delta \mathcal{D}(i)$ and $\Delta\mathcal{P}(i)$ denotes the change value in the objective functions of dual and primal problems, respectively. 
\end{lemma}
\begin{proof}
\indent In each iteration of \Alg\ref{alg1}, $\Delta q(t)$ is obtained from Line 18 in \Alg\ref{alg1}, and then let $u_{i}$ be $b_{ij^{*}}-\overline{Q}_{i}q(t),$ $\Delta\mathcal{D}(i)$ can be written as \eq \eqref{PD1}:

\begin{equation}
\label{PD1}
\begin{aligned}
\Delta\mathcal{D}(i)=A_{t}\Delta q(t)+u_{i}&=A_{t}\left[ q(t)\frac{\overline{Q}_{i}}{A_{t}}+\frac{b_{ij^{*}}}{(\overline{\alpha}_{t}-1)A_{t}}\right]+u_{i},\\
&=\overline{Q}_{i}q(t)+\frac{b_{ij^{*}}}{\overline{\alpha}_{t}-1}+b_{ij^{*}}-\overline{Q}_{i}q(t),\\
&=b_{ij^{*}}\left(\frac{1}{\overline{\alpha}_{t}-1}+1\right).
\end{aligned}
\end{equation}

Since $x_{ij^{*}} =1$ (Line 22), the change  value in objective function of primal problem is $\Delta \mathcal{P}(i)=b_{i,s^{*}}$. Moreover, according to \Lemma\ref{L1}, we have $b_{i,s^{*}}=b_{ij^{*}}$, thus the relationship between $\Delta\mathcal{D}(i)$ and $\Delta\mathcal{P}(i)$ is:
\begin{equation}
\Delta\mathcal{D}(i) = \left(\frac{1}{\overline{\alpha}_{t}-1}+1\right)\Delta \mathcal{P}(i).\label{9}
\end{equation}
\end{proof}

\begin{lemma}\label{L4}
\Alg \ref{alg1} constructs feasible solutions for both the primal and dual online problems.
\end{lemma}
 
\begin{proof}
We first prove that \Alg \ref{alg1} yields dual feasible solutions. Let $q(t)$ denote  the value of the dual variable at the end of each iteration in a time loop (Line 2). If $q(t)\geq \frac{b_{ij}}{\overline{Q}_{i}}$, then the dual constraint always holds; else if $q(t)\leq \frac{b_{ij}}{\overline{Q}_{i}}$, the dual variables will be increased until the dual constraints are satisfied. Let the dual variable $u_{i}$ be $b_{ij}-\overline{Q}_{i}q(t)$, the subsequent increase of $q(t)$ is always feasible due to the acceptance determination rules: if user $i$'s unit bidding price on bid $j$ is no less than the unit price at time slot $t$, i.e $q(t)\leq \frac{b_{ij}}{Q_{ij}}$, user $i$ will be allocated with the bid $j^{*}$ maximizing $b_{ij}-\overline{Q}_{i}q(t)$; otherwise, if $q(t)\geq \frac{b_{ij}}{Q_{ij}}$, user $i$ will be rejected ($x_{ij}=0$), and the dual variables will not be updated. 

We next show that \Alg \ref{alg1} yields primal feasible solutions. The iteration rule of $q(t)$ (Line 18) ensures that $q(t)$ is bounded by the sum of a geometric sequence with the common ratio ($1+\frac{\overline{Q}_{i}}{A_{t}}$). Consider a geometric sequence produced by the iterations of $q(t)$ for user $m$: the first item is  $\frac{b_{mj}}{(\overline{\alpha}_{t}-1)A_{t}}$, in which the value of $b_{mj}$ is fixed in each iteration (Lines 16$\sim$18), and the common ratio is $1+\frac{\overline{Q}_{m}}{A_{t}}$. Summing the terms of this geometric sequence gives:
\begin{equation}
\footnotesize
q(t)\geq \frac{b_{mj}}{\overline{Q}_{m}}\cdot\frac{1}{\overline{\alpha}_{t}-1}\cdot \left[\left(1+\frac{\overline{Q}_{m}}{A_{t}}\right)^{\sum_{i \in \mathcal{I}(t)}\sum_{j \in \mathcal{J}_{i}}x_{ij}}-1\right],\label{22}
\end{equation}
where the number of iterations in each time loop is larger than $\sum_{i \in \mathcal{I}(t)}\sum_{j \in \mathcal{J}_{i}}x_{ij}$. \ineq\eqref{22} can be rewritten as:
\begin{footnotesize}
\begin{equation}
q(t)\geq \frac{b_{mj}}{\overline{Q}_{m}}\cdot \frac{1}{\overline{\alpha}_{t}-1}\cdot \left[\left(1+\frac{\overline{Q}_{m}}{A_{t}}\right)^{\frac{A_{t}}{\overline{Q}_{m}}\cdot\frac{\sum_{i \in \mathcal{I}(t)}\sum_{j \in \mathcal{J}_{i}}\overline{Q}_{m}x_{ij}}{A_{t}}}-1\right].\label{22ee}
\end{equation}
\end{footnotesize}
Let $\overline{R}_{t}$ denote the maximum ratio of a user's requested quantity of mobility resources to the available resources at time slot $t$, $\overline{R}_{t}=\max_{i\in \mathcal{I}(t)}\left\{\frac{\overline{Q}_{i}}{A_{t}}\right\}$. Since $0\leq\frac{\overline{Q}_{i}}{A_{t}}\leq \frac{\overline{Q}_{m}}{A_{t}}\leq 1$:
\begin{equation}
\footnotesize
q(t)\geq \frac{b_{mj}}{Q_{m}}\cdot \frac{1}{\overline{\alpha}_{t}-1}\cdot \left[\left(1+\frac{Q_{m}}{A_{t}}\right)^{\frac{A_{t}}{Q_{m}}\cdot\frac{{\sum_{i\in \mathcal{I}(t)}\sum_{j \in \mathcal{J}_{i}}Q_{i}x_{ij}}}{A_{t}}}-1\right].\label{23}
\end{equation}

Since $0\leq \frac{\overline{Q}_{i}}{A_{t}}\leq \overline{R}_{t} \leq 1$, we have $\frac{\ln (1+\frac{\overline{Q}_{i}}{A_{t}})}{\frac{\overline{Q}_{i}}{A_{t}}} \geq \frac{\ln (1+\overline{R}_{t})}{\overline{R}_{t}}$, thus, $(1+\frac{\overline{Q}_{i}}{A_{t}})^{\overline{R}_{t}}\geq (1+\overline{R}_{t})^{\frac{\overline{Q}_{i}}{A_{t}}},$ and:
\begin{equation}
1+\frac{Q_{i}}{A_{t}}\geq (1+\overline{R}_{t})^{\frac{1}{\overline{R}_{t}}\cdot{\frac{\overline{Q}_{i}}{A_{t}}}}. \label{24}
\end{equation}

Since $\alpha_{t}=\left(1+\overline{R}_{t}\right)^{\frac{1}{\overline{R}_{t}}}$ and  $\overline{Q}_{i}=\max_{j\in\mathcal{J}_{i}} \left\{Q_{ij}\right\}$, substituting \ineq(\ref{24}) into \ineq (\ref{23}) yields:
\begin{equation}
\begin{aligned}
q(t) &\geq \frac{b_{mj}}{\overline{Q}_{m}}\cdot\frac{1}{\overline{\alpha}_{t}-1}\cdot\left(\overline{\alpha}_{t}^{\frac{{\sum_{i \in \mathcal{I}(t)}\sum_{j \in \mathcal{J}_{i}}Q_{ij}x_{ij}}}{A_{t}}}-1\right ).\label{25}
\end{aligned} 
\end{equation}

According to \eq\eqref{25}, if $\sum_{i \in \mathcal{I}(t)}\sum_{j \in \mathcal{J}_{i}}Q_{ij}x_{ij}\geq A_{t}$, then $q(t)\geq \frac{b_{mj}}{\overline{Q}_{m}}$. Since \Alg \ref{alg1} does not update the primal solution if $q(t)\geq \frac{b_{mj}}{\overline{Q}_{m}}$, the primal solutions will only be updated once $\sum_{i \in \mathcal{I}(t)}\sum_{j\in \mathcal{J}_{i}}Q_{ij}x_{ij}\leq A_{t}$ or $\sum_{i \in \mathcal{I}(t)}\sum_{s\in \mathcal{S}_{i}}Q_{i,s}\chi_{i,s}\leq A_{t}$.
\end{proof}

\textcolor{Cerulean}{Lemmas} \ref{L2}-\ref{L4} provide a worst-case time complexity bound on \Alg \ref{alg1} and show that this algorithm yields feasible to the online mobility resource allocation for the PAYG mechanism. The analog of \Alg \ref{alg1} for the PAAP mechanism is presented in \Alg \ref{alg2}, which shares a similar structure to the PAYG primal-dual online algorithm but differs in the updates of parameters and variables. We provide theoretical results for \Alg \ref{alg2} in \textcolor{Cerulean}{Lemmas} \ref{L7}-\ref{L9} which are analog to those obtained for \Alg \ref{alg1}. For conciseness, the proofs are given in \textcolor{Cerulean}{Supplementary material B}.

\begin{algorithm}[ht]
\small
\setstretch{1.0}
$\boldsymbol{A} [1,2,\cdots,\omega]\leftarrow C$ \\
\For{$t \in \Omega$}{
    $\underline{Q}_{i}\leftarrow \min_{j\in\mathcal{J}_{i}} \left\{Q_{ij}\right\}, \forall i\in \mathcal{I}(t)$\\
    $j_{i}^{*}\leftarrow\arg \max_{j \in \mathcal{J}_{i}}\left\{\frac{b_{i j}}{Q_{ij}}\right\},\forall i\in \mathcal{I}(t)$\\
	$\bar{\mathcal{I}}(t) \gets$ sort $\mathcal{I}(t)$ by decreasing unit bidding price $b_{ij_i^*}/Q_{ij_i^*}$\\
	$\bar{\mathcal{J}}_i \gets$ sort $\mathcal{J}_{i}$ by decreasing unit bidding price $b_{ij}/Q_{ij}$\\
	$\underline{R}_{t}\leftarrow\min_{i\in \mathcal{I}(t)}\left\{\frac{\underline{Q}_{i}}{A_{t}}\right\}$\\
	$\underline{\alpha}_{t}\leftarrow(1+\underline{R}_{t})^{\frac{1}{\underline{R}_{t}}}$\\
	$A_{t}\leftarrow \boldsymbol{A}[t]$\\
	$q(t)\leftarrow 0$\\
	\For{$i \in \bar{\mathcal{I}(t)}$}{ 
		\For{$j \in \bar{\mathcal{J}}_i$}{        
			\If{$q(t)\leq\frac{b_{ij}}{Q_{ij}}$}{
				$x_{ij}\leftarrow\frac{b_{ij}\sum_{j \in \mathcal{J}_{i}}Q_{ij}}{Q_{ij}\sum_{j \in \mathcal{J}_{i}}b_{ij}}$\\
	            $q(t)\leftarrow q(t)(1+\frac{\underline{Q}_{i}}{A_{t}})+\frac{b_{ij}x_{ij}}{(\underline{\alpha}_t-1)A_{t}}-\frac{(1-x_{ij})b_{ij}}{A_{t}}$\\
		    }
		    \ElseIf{$q(t)>\frac{b_{ij}}{Q_{ij}}$}{
    	        $x_{ij}\leftarrow0$\\
	    	}
		}
	    $\boldsymbol{A} \left[O_{i}:O_{i}+L_{i}-1\right]\leftarrow \boldsymbol{A} \left[O_{i}:O_{i}+L_{i}-1\right]-\sum_{j \in \mathcal{J}_{i}}Q_{ij}x_{ij}$ 
	}
}
\Return $\boldsymbol{x}$ and $\boldsymbol{q}$\\
\caption{PAAP primal-dual online algorithm}
\label{alg2}
\end{algorithm}

\begin{lemma}
\label{L7}
The worst-case time complexity of \Alg \ref{alg2} is $\mathcal{O}(|\Omega| |\mathcal{S}||\mathcal{I}(t)|)$.
\end{lemma}

\begin{lemma}
\label{L8}
In each iteration of \Alg\ref{alg2}, $\Delta \mathcal{P}(i,j)=\left(1-\frac{1}{\underline{\alpha}}\right) \Delta \mathcal{D}(i,j)$  holds, where $\Delta \mathcal{D}(i,j)$ and $\Delta \mathcal{P}(i,j)$ denotes the change value in the objective function of the dual and primal problems, respectively. 
\end{lemma}

\begin{lemma}
\label{L9}
\Alg \ref{alg2} constructs feasible solutions for both the primal and dual online problems.
\end{lemma}

We next analyze the performance of the proposed primal-dual online algorithms in terms of solution quality.

\subsection{Upper bound on the competitive ratios}
\label{PAYGcr}
To evaluate the performance of the proposed online algorithms (\Alg \ref{alg1} and \Alg \ref{alg2}), we formulate offline resource allocation problems which take as input the complete travel demand data over the entire optimization period. The offline mobility resource allocation problems for the PAYG and the PAAP mechanisms are summarized in \textcolor{Cerulean}{Model 2.1} and \textcolor{Cerulean}{Model 2.2}, respectively.\\

\noindent \textbf{Model 2.1} (PAYG offline mobility resource allocation).
\begin{subequations}
\label{mod2}
\allowdisplaybreaks
\begin{align}
&\max\sum_{i \in \mathcal{I}} \sum_{j \in \mathcal{J}_{i}}\sum_{t \in \Omega:t\leq O_{i}} b_{ij} x_{i j}^{t},\label{5a}\\
&\text{subject to:} \nonumber \\
&\sum_{m \in \mathcal{M}}  v_{m} l_{i j}^{mt} =D_{i} x_{i j}^{t}, &&\forall i \in  \mathcal{I},j \in \mathcal{J}_{i},t \in \Omega:t\leq O_{i},\label{5b}\\
&0\leq\sum_{m \in \mathcal{M}} l_{i j}^{mt}-T_{ij}x_{i j}^{t}\leq \Phi_{i}, &&\forall i \in  \mathcal{I},j \in \mathcal{J}_{i},t \in \Omega:t\leq O_{i},\label{5c}\\
&\sum_{m \in \mathcal{M}}\sigma_{m} l_{i j}^{mt} \leq \Gamma_{i}, &&\forall i \in  \mathcal{I},j \in \mathcal{J}_{i},t \in \Omega:t\leq O_{i}, \label{5d}\\
&x_{i j}^{t}\left(b_{ij}-p_{i j}^{t}\right) \geq 0, &&\forall i \in \mathcal{I}, j \in \mathcal{J}_{i},t \in \Omega:t\leq O_{i},\label{5e}\\
&p_{t}=\frac{b_{\max}}{C} \left(\sum_{i \in \mathcal{I}}\sum_{j \in \mathcal{J}_{i}}Q_{ij}x_{ij}^{t}\right)+b_{\min}, &&\forall t \in \Omega,\label{5f}\\
&p_{i j}^{t}=Q_{ij}p_{t},&&\forall i \in  \mathcal{I}, j \in \mathcal{J}_{i},t \in \Omega:t\leq O_{i},\label{5g}\\
&\sum_{i \in \mathcal{I}}\sum_{j \in \mathcal{J}_{i}} Q_{ij}x_{ij}^{t}\leq A_{t},&&\forall t \in \Omega,\label{5h}\\
&\sum_{j \in \mathcal{J}_{i}}\sum_{t \in \Omega:t\leq O_{i}} x_{i j}^{t} \leq 1,&& \forall i \in  \mathcal{I},\label{5i}\\
&l_{ij}^{mt} \geq 0, && \forall i \in  \mathcal{I}, j \in \mathcal{J}_{i}, m \in \mathcal{M},t \in \Omega:t\leq O_{i},\label{5j}\\
&b_{\min}\leq p_{t}\leq b_{\min}+b_{\max},&&\forall t \in \Omega,\label{5k}\\
&p_{i j}^{t}\geq 0,&&\forall i \in  \mathcal{I}, j \in \mathcal{J}_{i},t \in \Omega:t\leq O_{i},\label{5l}\\
&x_{i j}^{t}\in\{0,1\},&&\forall i \in  \mathcal{I}, j \in \mathcal{J}_{i},t \in \Omega:t\leq O_{i}.\label{5m}
\end{align}
\label{OFFPAYG}
\end{subequations}

\textcolor{Cerulean}{Model 2.1} can be viewed as a time-extended version of \textcolor{Cerulean}{Model 1.1}, where variable $x_{ij}^{t}$ indicates whether user $i$'s bid $j$ will be allocated at time slot $t$ or not, variable $l_{ij}^{mt}$ represents the number of time slots served by mode $m$ in the MaaS bundle customized for user $i$'s bid $j$, $p_{t}$ denotes the unit price at time slot $t$, and $p_{ij}^{t}$ denotes user $i$'s actual payment for bid $j$ at time slot $t$. The objective function \eqref{5a} aims to maximize social welfare overall all time slots. Constraints \eqref{5b}, \eqref{5c} and \eqref{5d} guarantee that the MaaS bundle can satisfy the user's requested distance requirement, travel time requirement and inconvenience tolerance, respectively. Constraint (\ref{5e}) illustrates that when the bidding price is smaller than the actual payment, the bid will be rejected. Constraint (\ref{5f}) illustrates that the unit price grows with the increasing quantity of allocated mobility resources at different time slot, and varies within $[b_{\min},b_{\min}+b_{\max}]$. Constraint (\ref{5g}) illustrates that the actual payment at time slot $t$ is determined by the unit price at time slot $t$ and the requested weighted quantity  of mobility resources. Constraint \eqref{5h} restricts the weighted quantity  of requested mobility resources at time slot $t$ can not exceed the weighted quantity  of available mobility resources at time slot $t$. Constraint \eqref{5i} guarantees at most one of user $i$'s multi-bids can be accepted at  time slot $t$. Observe that if $|\Omega|=1$, \textcolor{Cerulean}{Model 2.1} is equivalent to \textcolor{Cerulean}{Model 1.1}.\\

\noindent \textbf{Model 2.2} (PAAP offline mobility resource allocation).
\label{mod4}
\begin{subequations}
\allowdisplaybreaks
\begin{align}
&\max\sum_{i \in \mathcal{I}} \sum_{j \in \mathcal{J}_{i}}\sum_{t \in \Omega:t\leq O_{i}}  b_{i j} x_{ij}^{t},\label{11a}\\
&\text{subject to:}   \nonumber \\
&\text{\eqref{5b}-\eqref{5l}}, \nonumber\\
& 0 \leq x_{i j}^{t}\leq 1,&& \forall i \in \mathcal{I}, j \in \mathcal{J}_{i},t\in \Omega:t\leq O_{i}.\label{11k}
\end{align}
\end{subequations}
\textcolor{Cerulean}{Model 2.2} is analogous to \textcolor{Cerulean}{Model 1.2}, and if $|\Omega|=1$, \textcolor{Cerulean}{Model 2.2} is equivalent to \textcolor{Cerulean}{Model 1.2}. Similarly to the online resource allocation problems, the offline formulations represented by (\textcolor{Cerulean}{Model 2.1} and \textcolor{Cerulean}{Model 2.2}) can be reformulated as compact IP and LP, respectively.

\begin{definition}
For any user $i \in \mathcal{I}(t)$ and user bid $j \in \mathcal{J}_{i}$, let $\bm{l}_{ij}^{t} = [l_{ij}^{mt}]_{m \in \mathcal{M}}$. Let $\mathcal{S}_{ij}^{t}$ be the set defined as:
\begin{equation}
\mathcal{S}_{ij}^{t} \triangleq \left\{\bm{l}_{ij}^{t} \in \mathbb{R}^{|\mathcal{M}|} : \eqref{5b}-\eqref{5g}, b_{ij} \geq p_{ij}^{t}, x_{ij}^t=1\right\}. 
\end{equation}
We say that $\mathcal{S}_{ij}^{t}$ is the set of feasible MaaS bundles corresponding to bid $j$ for user $i$ at time slot $t \in \Omega$. Further, for each user $i \in \mathcal{I}(t)$, we define $\mathcal{S}_{i}^{t} \triangleq \bigcup_{j \in \mathcal{J}_{i}} \mathcal{S}_{ij}$ as the set of feasible MaaS bundles for user $i$ at time slot $t \in \Omega$. 
\end{definition}

\begin{lemma}
\label{L5}
For each user $i \in \mathcal{I}(t)$ and for each MaaS bundle $s \in \mathcal{S}_{i}^{t},\forall t \in \Omega$, let $\chi_{i,s}^{t}$ be a binary variable representing the allocation of $s$ to $i$ at time slot $t \in \Omega$. Consider the compact IP: 
\begin{subequations}
\label{CIP}
\allowdisplaybreaks
\begin{align}
&\max \sum_{i \in \mathcal{I}}\sum_{s \in \mathcal{S}_{i}^{t}} \sum_{t \in \Omega} b_{i ,s}\chi_{i,s}^{t}, \label{7a}\\
&\text{subject to:}  && \nonumber \\
&\sum_{i \in \mathcal{I}}\sum_{s \in \mathcal{S}_{i}^{t}}  Q_{i,s}\chi_{i,s}^{t} \leq A_{t},&& \forall t \in \Omega,\label{7b}\\
&\sum_{s \in \mathcal{S}_{i}^{t}}\sum_{t \in \Omega:t\leq O_{i}} \chi_{i,s}^{t} \leq 1, &&\forall i \in \mathcal{I},\label{7c}\\
& \chi_{i,s}^{t}\in \left\{0,1\right\},\ && \forall i \in \mathcal{I},s \in \mathcal{S}_{i}^{t},t \in \Omega:t\leq O_{i}. \label{7d}
\end{align}
\end{subequations}
\noindent The compact IP \eqref{CIP} is equivalent to \emph{\textcolor{Cerulean}{Model 2.1}}. Further, the LP-relaxation of IP \eqref{CIP} is equivalent to \emph{\textcolor{Cerulean}{Model 2.2}}.
\end{lemma}
\begin{proof} 
The proof follows from that of \Lemma \ref{L1}.
\end{proof}

We summarize the primal and dual problems of the offline compact LP as follows:

$$\begin{array}{l|ll}
\multicolumn{1}{l|} {\text { \textbf{Primal offline problem $OPT_{3}$ } }} & \multicolumn{1}{l} {\text { \textbf{Dual offline problem $OPT_{4}$}}} \\
\hline \max \sum_{i \in \mathcal{I}} \sum_{s \in \mathcal{S}_{i}^{t}}\sum_{t \in \Omega:t\leq O_{i}} b_{i,s}\chi_{i,s}^{t}, & \min \sum_{t \in \Omega:t\leq O_{i}}A_{t} q(t)+\sum_{i \in \mathcal{I}} u_{i},\\
\text { subject to: } & \text { subject to: } \\
\sum_{i \in \mathcal{I}}\sum_{s \in \mathcal{S}_{i}^{t}}  Q_{i,s} \chi_{i,s}^{t}  \leq A_{t}, \forall t \in \Omega,& Q_{i,s}q(t)+u_{i}\geq b_{i,s},\forall i \in \mathcal{I},s \in \mathcal{S}_{i}^{t},t \in \Omega:t\leq O_{i},\\
\sum_{s\in \mathcal{S}_{i}^{t}}\sum_{t \in \Omega:t\leq O_{i}}  \chi_{i,s}^{t}  \leq 1, \forall i \in \mathcal{I}, & q(t)\geq 0, \forall t \in \Omega,\\
\chi_{i,s}^{t}  \geq 0,\forall i \in \mathcal{I}, s \in \mathcal{S}_{i}^{t}, t \in \Omega:t\leq O_{i}. &u_{i}\geq 0, \forall i \in \mathcal{I}.
\end{array}$$

For any input sequence $\boldsymbol{\tau}$, let $\mathcal{Z}^{*}(\boldsymbol{\tau})$ denote the maximum value of the offline problem ($OPT_{3}$), if \Alg\ref{alg1} outputs a solution which is at least $\Theta\cdot\mathcal{Z}^{*}(\boldsymbol{\tau})$, then we say that its competitive ratio is $\Theta$ \citep{borodin2005online}. 

\begin{prop}
\label{theorem2}
The competitive ratio of \Alg \ref{alg1} is $\Theta = \left(1-\mathbb{R}_{max}\right)\left(1-\frac{1}{\overline{\alpha}}\right)$, where $\overline{\alpha}=\min_{t \in \boldsymbol{\tau}}\overline{\alpha}_{t}$, $\mathbb{R}_{max}=\max_{t \in \boldsymbol{\tau}}\overline{R}_{t}$ , $\overline{\alpha}_{t}=(1+\overline{R}_{t})^{\frac{1}{\overline{R}_{t}}}$, $\overline{R}_{t}=\max_{i\in \mathcal{I}(t)}\left\{\frac{\overline{Q}_{i}}{A_{t}}\right\}$, $\forall t\in \Omega$, $\overline{Q}_{i}= \max_{j\in\mathcal{J}_{i}} \left\{Q_{ij}\right\}, \forall i\in \mathcal{I}(t)$.
\end{prop}

\begin{proof}
Let $k$ be the critical index determined at Line 6 of \Alg\ref{alg1}. Let ${q}(t)^{\mathrm{end}}$ and ${q}(t)^{\mathrm{start}}$ denote the value of $q(t)$ before and after each iteration in the loop of $i=k$ in \Alg\ref{alg1}, respectively. Substituting ${q}(t)^{\mathrm{end}}$ and ${q}(t)^{\mathrm{start}}$ into Line 16 of \Alg \ref{alg1}, yields \eq\eqref{28ee}:
\begin{equation}
\label{28ee}
q(t)^{\text{end}}= q(t)^{\text{start}}\left(1+\frac{\overline{Q}_{k}}{A_{t}}\right)+\frac{b_{kj}}{(\overline{\alpha}_{t}-1) A_{t}}.\\ 
\end{equation}

According to \ineq(\ref{25}), before examining user $k$'s bids, the value of ${q}(t)^{\mathrm{start}}$ is bounded as:
\begin{equation}
\label{ee29}
q(t)^{\text{start}}\geq \frac{b_{kj}}{\overline{Q}_{k}}\cdot\frac{1}{\overline{\alpha}_{t}-1}\cdot\left(\overline{\alpha}_{t}^{\frac{\sum_{i \in \mathcal{I}(t) \backslash\{k\}}\sum_{j \in \mathcal{J}_{i}} \overline{Q}_{i}x_{ij}}{A_{t}}}-1\right).
\end{equation}

Substituting \ineq \eqref{ee29} into \ineq \eqref{28ee}, and simplifying yields:
\begin{align}
\label{30ee}
q(t)^{\text{end}} 
&\geq \frac{b_{i,k}}{\overline{Q}_{k}}\cdot\frac{1}{\alpha_{t}-1}\cdot\left(\overline{\alpha}_{t}^{\frac{\sum_{i \in \mathcal{I}(t) \backslash\{k\}}\sum_{j\in \mathcal{J}_{i}} \overline{Q}_{i} x_{ij}}{A_{t}}}-1\right) \cdot\left(1+\frac{\overline{Q}_{k}}{A_{t}}\right)+\frac{b_{kj}}{(\overline{\alpha}_{t}-1) \cdot A_{t}},\\
&\geq\frac{b_{i,k}}{\overline{Q}_{k}}\cdot\frac{1}{\overline{\alpha}_{t}-1}\cdot\left[\overline{\alpha}_{t}^{\frac{\sum_{i \in \mathcal{I}(t) \backslash\{k\}}\sum_{j \in \mathcal{J}_{i}} \overline{Q}_{i} x_{ij}}{A_{t}}}\cdot\left(1+\frac{\overline{Q}_{k}}{A_{t}}\right)-1\right]. \label{29}
\end{align}

According to \ineq (\ref{24}), we have $1+\frac{\overline{Q}_{k}}{A_{t}}\geq (1+\overline{R}_{t})^{\frac{1}{\overline{R}_{t}}\cdot{\frac{\overline{Q}_{k}}{A_{t}}}}$, thus \ineq (\ref{29}) can be rewritten as:
\begin{equation}
\begin{aligned}
q(t)^{\text{end}} \geq \frac{b_{kj}}{\overline{Q}_{k}}\cdot\frac{1}{\overline{\alpha}_{t}-1}\cdot\left[\overline{\alpha}_{t}^{\frac{\sum_{i \in \mathcal{I}(t) \backslash\{k\}}\sum_{j \in \mathcal{J}_{i}} \overline{Q}_{i} x_{ij}}{A_{t}}}\cdot (1+\overline{R}_{t})^{\frac{1}{\overline{R}_{t}}\cdot{\frac{\overline{Q}_{k}}{A_{t}}}} -1\right].\label{30}
\end{aligned}
\end{equation}

Assume that user $k$ is accepted at time slot $t$, then $\sum_{i \in \mathcal{I}(t) \backslash\{k\}}\sum_{j \in \mathcal{J}_{i}} \overline{Q}_{i} x_{ij}+\overline{Q}_{k}x_{kj}=\sum_{i \in \mathcal{I}(t)}\sum_{j \in \mathcal{J}_{i}} \overline{Q}_{i} x_{ij}$, and since $\overline{\alpha}_{t}=(1+\overline{R}_{t})^{\frac{1}{\overline{R}_{t}}}$ \ineq (\ref{30}) can be written as:
\begin{equation}
\begin{aligned}
q(t)^{\text{end}}&\geq\frac{b_{kj}}{\overline{Q}_{k}}\cdot\frac{1}{\overline{\alpha}_{t}-1}\cdot\left(\overline{\alpha}_{t}^{\frac{\sum_{i \in \mathcal{I}(t) \backslash\{k\}}\sum_{j \in \mathcal{J}_{i}} \overline{Q}_{i} x_{ij}}{A_{t}}}\cdot \overline{\alpha}_{t}^{\frac{\overline{Q}_{k}}{A_{t}}} -1\right),\\
\end{aligned}\label{31}
\end{equation}
\begin{equation}
\begin{aligned}
&=\frac{b_{kj}}{\overline{Q}_{k}}\cdot\frac{1}{\alpha_{t}-1}\cdot\left(\overline{\alpha}_{t}^{\frac{\sum_{i \in \mathcal{I}(t)}\sum_{j\in \mathcal{J}_{i}} \overline{Q}_{i} x_{ij}}{A_{t}}} -1\right).\label{32}
\end{aligned}
\end{equation}

\eq \eqref{32} illustrates that if the requested quantity of mobility resources exceeds the available mobility resources at time slot $t$, $\sum_{i \in \mathcal{I}(t)}\sum_{j\in \mathcal{J}_{i}} \overline{Q}_{i} x_{ij} > A_{t}$, then user $k$ will not be allocated with any resources $q^{\text{end}}\geq \frac{b_{kj}}{\overline{Q}_{k}}$, ; thus  $\sum_{i \in \mathcal{I}(t)}\sum_{j \in \mathcal{J}_{i}}\overline{Q}_{i}x_{ij}\geq A_{t}-\max_{i\in\mathcal{I}(t)}\{\overline{Q}_{i}\}$, $\forall t \in \Omega$. Since $\overline{R}_{t}=\max_{i\in \mathcal{I}(t)}\left\{\frac{\overline{Q}_{i}}{A_{t}}\right\},\forall t \in \Omega$, the social welfare obtained by \Alg\ref{alg1} at time slot $t$ is at least:
\begin{equation}
\label{ee35}
\begin{aligned}
\sum_{i \in \mathcal{I}(t)}\sum_{j \in \mathcal{J}_{i}}b_{ij} x_{ij}\cdot \frac{A_{t}-\max_{i \in \mathcal{I}(t)}\{\overline{Q}_{i}\}}{A_{t}} &= \sum_{i \in \mathcal{I}(t)}\sum_{j \in \mathcal{J}_{i}} b_{ij} x_{ij}\cdot\left(1-\overline{R}_{t}\right).
\end{aligned}
\end{equation}

Baed on \Lemma\ref{L1}, we have $
\sum_{i \in \mathcal{I}(t)}\sum_{s \in \mathcal{S}_{i}}b_{i,s}\chi_{i,s}=\sum_{i \in \mathcal{I}(t)}\sum_{j \in \mathcal{J}_{i}}b_{ij}x_{ij}.$ According to \Lemma \ref{L3}, $\Delta \mathcal{D}(i)$ and $\Delta \mathcal{P}(i)$ denote the change value in the objective functions of dual problem $OPT_{2}(t)$ and primal problem $OPT_{1}(t)$, respectively, at each iteration of \Alg \ref{alg1}. Let $\mathcal{P}(t)$ and $\mathcal{D}(t)$ denote the objective value of the dual and primal problems obtained by \Alg\ref{alg1}. In time loop $t$, we have $\mathcal{D}(t)=\sum_{i \in \mathcal{I}(t)}\Delta \mathcal{D}(i)$  and $\mathcal{P}(t)=\sum_{i \in \mathcal{I}(t)}\Delta \mathcal{P}(i)$. Based on \Lemma \ref{L3}, the relationship between $\mathcal{D}(t)$ and $\mathcal{P}(t)$ is:
\begin{equation}
\mathcal{P}(t)=\left(1-\frac{1}{\overline{\alpha}_{t}}\right)\mathcal{D}(t).  \label{34} 
\end{equation}

Given the input time sequence $\boldsymbol{\tau}=[1,2,\cdots,t-1,t]$, let $\mathcal{P}(\boldsymbol{\tau})$ and $\mathcal{D}(\boldsymbol{\tau})$ denote the objective value of the primal $OPT_{3}$ and the dual problem $OPT_{4}$ obtained by \Alg\ref{alg1}. Since $\alpha=\min_{t \in \boldsymbol{\tau}}\alpha_{t}$, \eq \eqref{34} implies:
\begin{equation}
\mathcal{P}(\boldsymbol{\tau})\geq\left(1-\frac{1}{\overline{\alpha}}\right)\mathcal{D}(\boldsymbol{\tau}).
\label{eq32}
\end{equation}

Let $W_{\text{Alg1}}(\boldsymbol{\tau})$ denote the social welfare obtained by \Alg\ref{alg1} corresponding to the sequence $\boldsymbol{\tau}$. Since $\mathbb{R}_{max}=\max_{t \in \boldsymbol{\tau}}\overline{R}_{t}$, according to \eq\eqref{ee35}, we have $W_{\text{Alg1}}(\boldsymbol{\tau})\geq\left(1-\mathbb{R}_{max}\right)\mathcal{P}(\boldsymbol{\tau})$ and:
\begin{equation}
\label{ee37}
W_{\text{Alg1}}(\boldsymbol{\tau})\geq\left(1-\mathbb{R}_{max}\right)\left(1-\frac{1}{\overline{\alpha}_{t}}\right)\mathcal{D}(\boldsymbol{\tau}).
\end{equation}

Given the input sequence $\boldsymbol{\tau}$, let $\mathcal{Z}^{IP}(\boldsymbol{\tau})$ denote the optimal value of the offline compact IP \eqref{7a}-\eqref{7d}, and $\mathcal{Z}^{*}(\boldsymbol{\tau})$ denote the optimal value of the offline compact LP $OPT_{3}$. According to weak duality, we have $\mathcal{D}(\boldsymbol{\tau})\geq \mathcal{Z}^{*}(\boldsymbol{\tau})$, and \eq\eqref{ee37} can be rewritten as follows:
\begin{align}
\label{ee38}
W_{\text{Alg1}}(\boldsymbol{\tau})&\geq \left(1-\mathbb{R}_{max}\right)\left(1-\frac{1}{\overline{\alpha}_{t}}\right)\mathcal{Z}^{*}(\boldsymbol{\tau}),\\
&\geq \left(1-\mathbb{R}_{max}\right)\left(1-\frac{1}{\overline{\alpha}}\right)\mathcal{Z}^{IP}(\boldsymbol{\tau}).
\label{ee39}
\end{align}
Hence the competitive ratio of \Alg \ref{alg1} is $\Theta = \left(1-\mathbb{R}_{max}\right)\left(1-\frac{1}{\overline{\alpha}}\right)$.
\end{proof}


\begin{prop}
\label{theorem4}
The competitive ratio ($\Theta$) of \Alg\ref{alg2} is $1-\frac{1}{\underline{\alpha}}$, in which $\underline{\alpha}=\min_{t \in \boldsymbol{\tau}}\underline{\alpha}_{t}$, $\underline{\alpha}_{t}=(1+\underline{R}_{t})^{\frac{1}{\underline{R}{t}}}$, $\underline{R}_{t}=\min_{i\in \mathcal{I}(t)}\left\{\frac{\underline{Q}_{i}}{A_{t}}\right\}$,$\forall t\in \Omega$, $\underline{Q}_{i}=\min_{j\in\mathcal{J}_{i}} \left\{Q_{ij}\right\}, \forall i\in \mathcal{I}(t)$.
\end{prop}
\begin{proof}
The proof of \textcolor{Cerulean}{Proposition} \ref{theorem4} is provided in \textcolor{Cerulean}{Supplementary material B}.
\end{proof}

\begin{corollary}
In a time slot, when the ratio of the quantity of mobility resources requested
by a user to the available  mobility resources approaches zero, the competitive ratio of PAYG online algorithm (\Alg\ref{alg1}) is equivalent to the competitive ratio of the PAAP online algorithm (\Alg\ref{alg2}). Namely, when $\frac{\overline{Q}_{i}}{A_{t}}\rightarrow 0$, we have $\Theta=\left(1-\mathbb{R}_{max }\right)\left(1-\frac{1}{\overline{\alpha}}\right)=\left(1-\frac{1}{\underline{\alpha}}\right)=1-\frac{1}{e}$, in which $e$ denotes Euler's number.
\end{corollary}

\section{Rolling horizon configurations}
\label{S6}
In the proposed MaaS mechanisms, the travel demand data are input into the model as a ``stream'' and the mechanisms are executed periodically based on the current input data, without knowledge of future demand. To implement the proposed MaaS mechanisms, we use a rolling horizon algorithm (RHA) where $\Delta t$ denotes the time frequency at which the optimization problems are solved, and $\mathcal{T}$ denotes the length of the time horizon considered at the current iteration. The time horizon length reflects users' booking flexibility: at the iteration corresponding to time slot $t$, only users with a requested departure time in the time window $[t, t+\mathcal{T}]$ are considered in the online resource allocation problems.
\begin{figure*}[b!]
\centering
\includegraphics[width=1\textwidth]{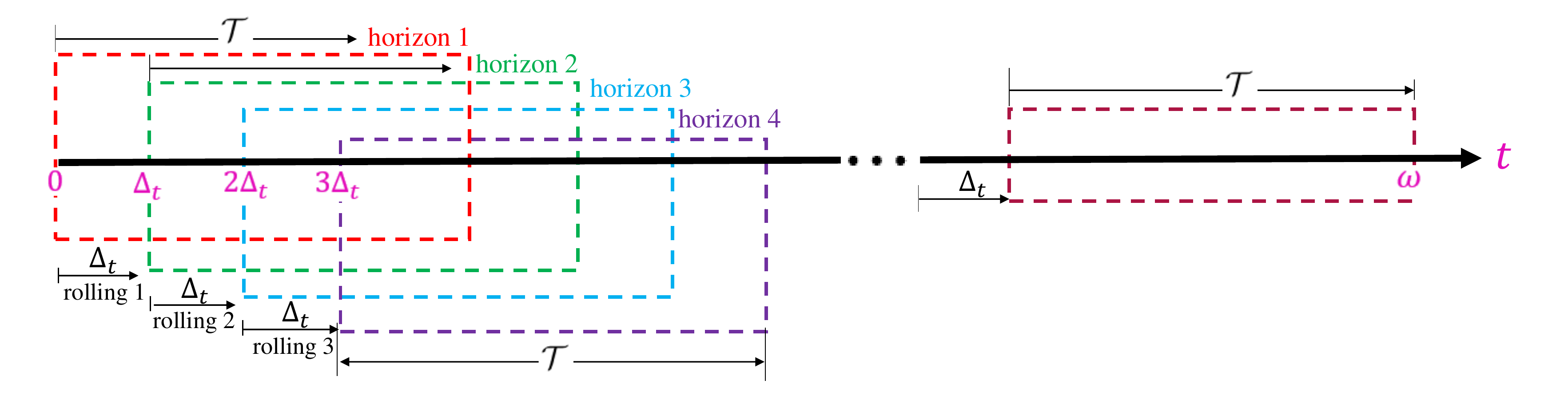}
\caption{Rolling horizon algorithm (RHA) configuration}\label{Fig5}
\end{figure*}

\Fig\ref{Fig5} illustrates that the proposed rolling horizon configuration rolls forward per $\Delta t$ time slots and solve the corresponding mobility resource allocation problem with a time horizon $\mathcal{T}$. Let $\omega$ denote the last time slot, the number of rolling processes ($n$) is written as $\omega/\Delta t$. The detailed procedure of the rolling horizon configuration is introduced in \Alg\ref{alg3}.

\begin{algorithm}[ht!]
\small
	\setstretch{1.0}
    initialize time horizon lengths $\mathcal{T}$ and time step $\Delta t$ \\
	\For{$n \in \left\{0,\Delta t,2\Delta t, \cdots, \omega/\Delta t   \right\}$ }{
	\For{$i \in \mathcal{I}(n)$}{
		{{ \If{$\Delta t=1$}{
			   \If{\emph{MaaS mechanism=PAYG}}{$ \left[x_{ij}, q(t) \right]\leftarrow$\textbf{Algorithm 1} ($\boldsymbol{B}_{i}$)   \COMMENT{execute  PAYG online algorithm}\\}
   			\ElseIf{\text{\emph{MaaS mechanism=PAAP}}} {$\left[x_{ij},q(t) \right]\leftarrow$ \textbf{Algorithm 2} ($\boldsymbol{B}_{i}$)\COMMENT{execute PAAP online algorithm}\\}
   			   	$p_{ij}\leftarrow q(t)Q_{ij}x_{ij}$\\} 
              	\ElseIf{$\Delta t>1$}{
              	\If{\emph{MaaS mechanism=\text{PAYG}}}{	
              	$\left[x_{ij}^{t}, p_{t},p_{ij}^{t}\right]\leftarrow$\textbf{Model 2.1} ($\boldsymbol{B}_{i}$)\COMMENT{execute  PAYG offline  formulation}}
              		\ElseIf{\text{\emph{MaaS mechanism=PAAP}}}{	
              		$ \left[x_{ij}^{t},p_{t},p_{ij}^{t} \right]\leftarrow$\textbf{Model 2.2} ($\boldsymbol{B}_{i}$)\COMMENT{execute  PAAP offline  formulation} }
              	}
		}	}
		}
	}
	\Return $\boldsymbol{x}$, $\boldsymbol{p}$ \\
	\caption{Rolling horizon configurations}
	\label{alg3}
\end{algorithm}

We consider four RHA configurations with different time step ($\Delta t$), time horizon lengths ($\mathcal{T}$) and optimization methods. All configurations can be executed with the both mechanisms (PAYG and PAAP). In each case, the corresponding mobility resources allocation problem can be solved either exactly (using \textcolor{Cerulean}{Model 1.2}/\textcolor{Cerulean}{Model 2.2}) or heuristically (using \Alg\ref{alg1}/\Alg\ref{alg2}).\\

\noindent\textbf{1}) \textbf{RHA} ($\Delta t=1,\mathcal{T}=1 $): rolls forward per  unit time slot, and users placing an order at time slot $t$ can request a service departing at time slot $t$ ($O_{i}=t$).\\
\textbf{2}) \textbf{RHA} ($\Delta t=1,\mathcal{T}=240$): rolls forward per unit time slot, and users placing an order at time slot $t$ can book a service departing within the time window $[t,t+\mathcal{T}]$. \\
\textbf{3}) \textbf{RHA} ($\Delta t>1,\mathcal{T}=240$): rolls forward every $\Delta t$ time slots to solve a small scale offline problem, and the users placing an order at time $t$ can book a service departing within the time window $[t,t+\mathcal{T}]$. \\
\textbf{4}) \textbf{SHA} ($\Delta t=\mathcal{T},\mathcal{T}=240$): this is a single horizon configuration aiming to provide a benchmark to evaluate alternative configurations. In this configuration, the offline problem is solved with full knowledge of the future travel demand overall time slots.

Since the first three configurations run in an online manner, they are referred to as online configurations. Instead, SHA is referred to as the offline configuration and is used as an oracle to benchmark the performance of the online RHA configurations.

\section{Numerical experiments}
\label{S7}

We conduct a series of numerical experiments to evaluate the performance of the proposed MaaS mechanisms. Specifically, we discuss the impacts of several parameters, i.e. the number of bids, bid range ratio, travel demand in a time slot, speed factor, and capacity. We also examine the impact of different types of unit price functions on social welfare, and numerically validate the derived competitive ratios. All numerical experiments are conducted using Python 3.7.4 and CPLEX Python API on a Windows 10 machine with, Intel(R) Core i7-8700 CPU $@$ 3.20 GHz, 3192 Mhz, 6 Core(s) and with 64 GB of RAM.

\subsection{Input data and parameter settings}
Real-time MaaS requests information is an input to the proposed mechanisms. To obtain users' request data in the proposed MaaS system, we conduct stochastic simulations under different auction settings, in which the parameters are set as follows. We consider five types of transport modes with different  commercial speed (km/min) and inconvenience cost per unit of time (\$/min) given in \T\ref{T4}. 
\begin{table}[ht]
\setlength{\abovecaptionskip}{0pt}
\setlength{\belowcaptionskip}{0pt}
\centering
\footnotesize
\caption{The commercial speed and inconvenience cost per unit of time in terms of different modes}
\begin{tabular}{cccccc}
\toprule
\multirow{3}{*}{Modes} & $\mathrm{m}=1$ & $\mathrm{m}=2$ & $\mathrm{m}=3$ & $\mathrm{m}=4$ & $\mathrm{m}=5$ \\
&\multirow{2}{*}{Taxi}&Ride sharing with 2 riders&Ride sharing with 3 riders&Public transit&Bicycle-sharing\\
\midrule
Average speed & $v_{1}$ & $v_{2}$ & $v_{3}$ & $v_{4}$ & $v_{5}$ \\
$(\mathrm{km} / \mathrm{min})$ & 0.5 & 0.3 & 0.25 & 0.18&0.1 \\
\hline Inconvenience & $\delta_{1}$ & $\delta_{2}$ & $\delta_{3}$ & $\delta_{4}$ & $\delta_{5}$ \\
cost (\$)/min& 0 & 0.5 & 1&2&6 \\
\bottomrule
\end{tabular}
\label{T4}
\end{table}

In the PAYG simulations, Traveller $i$'s MaaS request information is $\left\{D_{i}, O_{i},\Phi_{i}, \Gamma_{i}, \left\{T_{ij}, b_{ij} : j \in \mathcal{J}_i\right\}\right\}$.  User $i$'s requested distance ($D_{i}$) for each trip is randomly generated within [1km,18km]. User $i$'s requested departure time is  generated in different ways under different rolling horizon configurations as discussed in \s\ref{S6}. Both travel delay budget ($\Phi_{i}$) and inconvenience tolerance ($\Gamma_{i}$) are user $i$'s own characteristic and have a reverse relationship with her bidding price, $\Phi_{i}$ is randomly generated within $\left[0,\frac{100}{b_{ij}}\right]$, and $\Gamma_{i}$ is randomly generated within $\left[0,\frac{100 D_{i}}{b_{ij}}\right]$. To generate a realistic MaaS system, the requested travel time for each trip is set within the range of travel time taken by the fastest mode (taxi) and the slowest mode (bicycle-sharing), namely, the requested travel time of user $i$'s  bid $j$ ($T_{ij}$) is randomly generated within  $\left[\frac{D_{i}}{v_{5}},\frac{D_{i}}{v_{1}}\right]$. Given the travelling distance, the bidding price ($b_{ij}$) is set based on the tariff of current transportation system in Sydney including: Uber, metro, bus, Tram and Lime. At each time slot, the minimum unit bidding price ($b_{min}$) is set based on the price of public transit in Sydney\footnote{\cite{opal} Trip Planner can be used to estimate the fare of different public transport modes in NSW, Australia.}, and the maximum unit bidding price ($b_{max}$) is set based on the price of UberX in Sydney, which varies over the time during one day\footnote{ \cite{uber}'s Real-time Estimator provides real-time fare estimates on each trip.}. Accordingly, user $i$'s bidding price ($b_{ij}$) is randomly generated within $[b_{\min}Q_{ij},b_{\max}Q_{ij}]$. 

The operation time of the MaaS system is set to 20 hours every day (6:00am-01:00am). We consider time slot of 1 min, hence there are 1200 time slots per day, and the capacity of mobility resources in each time slot is set to 500. The number of travellers placing an order at time slot $t$ ($\lambda_{t}$) satisfies the normal distribution, where the mean value and standard deviation are set as different values between peak-hours and non-peak hours as indicated in \T \ref{T5}.
\begin{table*}[t!]
\footnotesize
\setlength{\abovecaptionskip}{0pt}
\setlength{\belowcaptionskip}{0pt}
\centering
\caption{The travel demand of each time slot ($\lambda_{t}$) in the PAYG simulation}
\begin{tabular}{cccccc}
\toprule
Time & $6: 00-7: 00$ & $8: 00-9: 00$ & $10: 00-17: 00$ & $18: 00-19: 00$ & $20: 00-01: 00$ \\
\midrule
Time slot & [$1,2 \cdots,120$] & [$121,122, \cdots 240$] & [$241,182\cdots,720$] & [$721,722 \cdots,840$] & [$841,842, \cdots,1200$] \\
$\lambda_{t}$ & $\lambda_{t} \sim \mathcal{N}\left(2, 1^{2}\right)$ & $\lambda_{t} \sim \mathcal{N}\left(8, 2^{2}\right)$ & $\lambda_{t} \sim \mathcal{N}\left(2, 1^{2}\right)$ & $\lambda_{t} \sim \mathcal{N}\left(8, 2^{2}\right)$ & $\lambda_{t} \sim \mathcal{N}\left(2, 1^{2}\right)$  \\
\bottomrule
\end{tabular}
\label{T5}
\end{table*}

In the PAAP simulations, user $i$'s MaaS request information is $\left\{D_{i}, O_{i},\Phi_{i}, \Gamma_{i}, L_{i}, \left\{T_{ij}, b_{ij} : j \in \mathcal{J}_i\right\}\right\}$. user $i$'s requested distance ($D_{i}$) for a mobility package is randomly generated within [1km,300km], user $i$'s requested time period of mobility package ($L_{i}$) is randomly generated within [5,14]. The other parameters are generated in the same way as in the PAYG simulations. In the PAAP simulations, each time slot represents one day and we consider 100 time slots (days) for a time cycle, and the capacity of mobility resources at each time slot is set to 10000. The number of travellers placing an order at time slot $t$ ($\kappa_{t}$) satisfies the normal distribution, $\kappa_{t} \sim \mathcal{N}\left(50,10^{2}\right)$. 

\subsection{Numerical Results}
We conduct sensitivity analysis on a series of parameters in \s\ref{721}, compare the PAYG and PAAP mechanisms in \s\ref{722}, validate the derived competitive ratios in \s\ref{723} and discuss the booking flexibility under different rolling horizon lengths in \s\ref{724}; Finally compare the social welfare and computation time under different online and offline configurations in \s\ref{725}.

\subsubsection{Sensitivity analysis on the online MaaS mechanisms}
\label{721}
In this subsection, we execute \Alg \ref{alg1} to evaluate the performance of the PAYG online mechanism by conducting a sensitivity analysis on the bid range ratio $b_{\max}/b_{\min}$, unit price functions, the speed factor and capacity in terms of acceptance ratio and social welfare.

We define the acceptance ratio as the ratio of the number of accepted users to the number of users participating the online auction at time slot $t$, which is an important index for evaluating user satisfaction. For clarity, we only show the hourly average acceptance ratio, i.e. the average acceptance ratio for each bin of 60 time slots. \Fig \ref{fig7a}, \Fig \ref{fig7b} and  \Fig\ref{fig7c} show the acceptance ratio under linear unit price function (\eq\eqref{eq4}), quadratic unit price function (\eq\eqref{eq6}) and exponential unit price function (\eq\eqref{eq8}) for varying bid range ratios, respectively. 

\begin{figure}[ht!]
\centering
	\subfloat[Linear unit price function]{\includegraphics[width=0.9\textwidth]{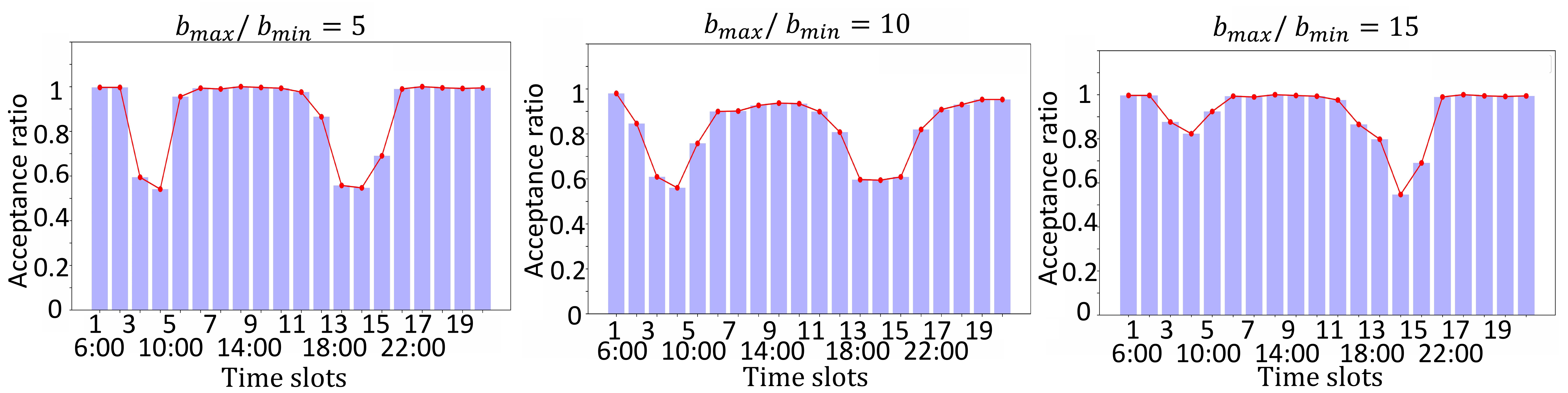}\label{fig7a}}\\
\subfloat[Quadratic unit price function ]{\includegraphics[width=0.9\textwidth]{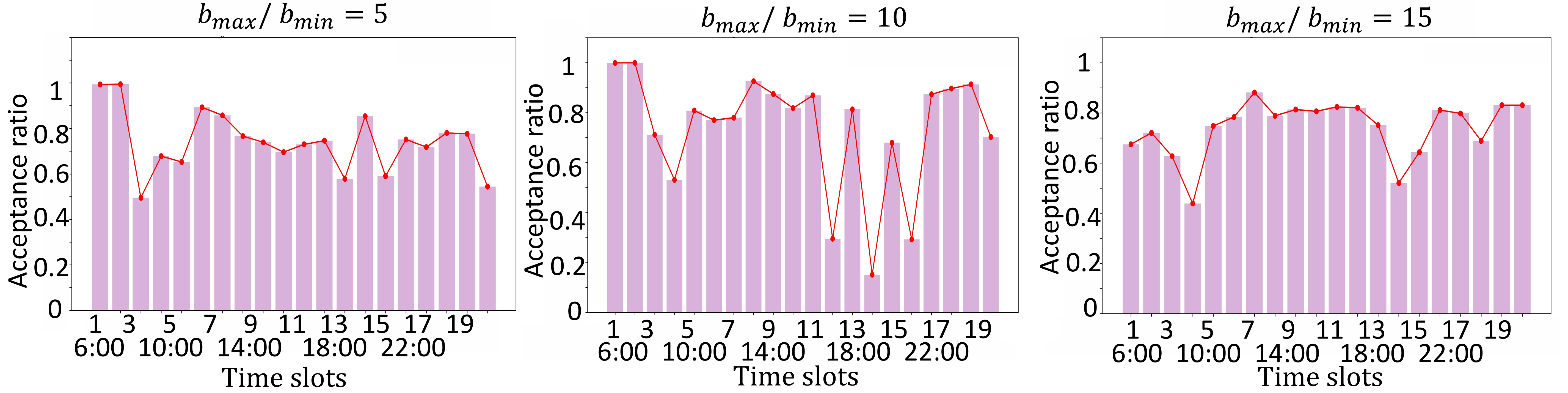}\label{fig7b}}\\
\subfloat[Exponential unit price function ]{\includegraphics[width=0.9\textwidth]{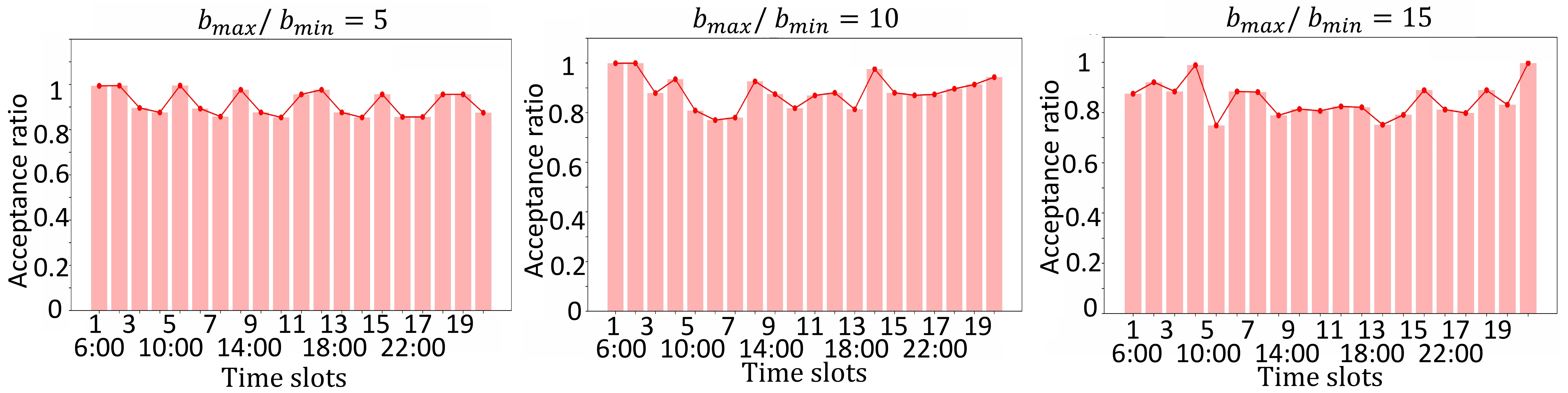}\label{fig7c}}
\caption{Acceptance ratio under different types of unit price functions with different  $b_{\max}/b_{\min}$}
\label{Fig77}
\end{figure}

\Fig \ref{Fig77} shows that the acceptance ratio under exponential unit price function is higher than that under other types of functions, this is the reason why the  exponential unit price function (\eq\eqref{eq8}) is used as the iteration rule in Line 16 of \Alg\ref{alg1}.
Then we report the variation of the hourly average social welfare in \Fig \ref{fig8}. We find that if $b_{\max}/b_{\min}=5$, the value of social welfare is higher than its counterparts under all three types of unit price functions. Moreover, the social welfare in terms of time slot exhibits a similar pattern under three types of unit price functions.

To observe the influence of capacity $C$ and  travel mode speeds onto social welfare, we set the value of $b_{max}/b_{min}$ to 5 and conduct sensitivity analysis on these parameters reported in \Fig\ref{figCS}. \Fig\ref{figcap} shows that social welfare grows with the increase of capacity over all time slots, and that the social welfare remains unchanged beyond a capacity of 500.  then apply a speed factor ($-75\%$, $-50\%$,$-25\%$,1,$+25\%$,$+50\%$) on the commercial speed of each travel mode $v_{m}$ and report the results of social welfare. \Fig\ref{figSpeed} shows that the speed factor has very limited influence on social welfare over all time slots, and thus illustrate the reliability of the proposed MaaS system.
\begin{figure*}[ht!]
	\subfloat[Linear unit price function]{\includegraphics[width=0.33\textwidth]{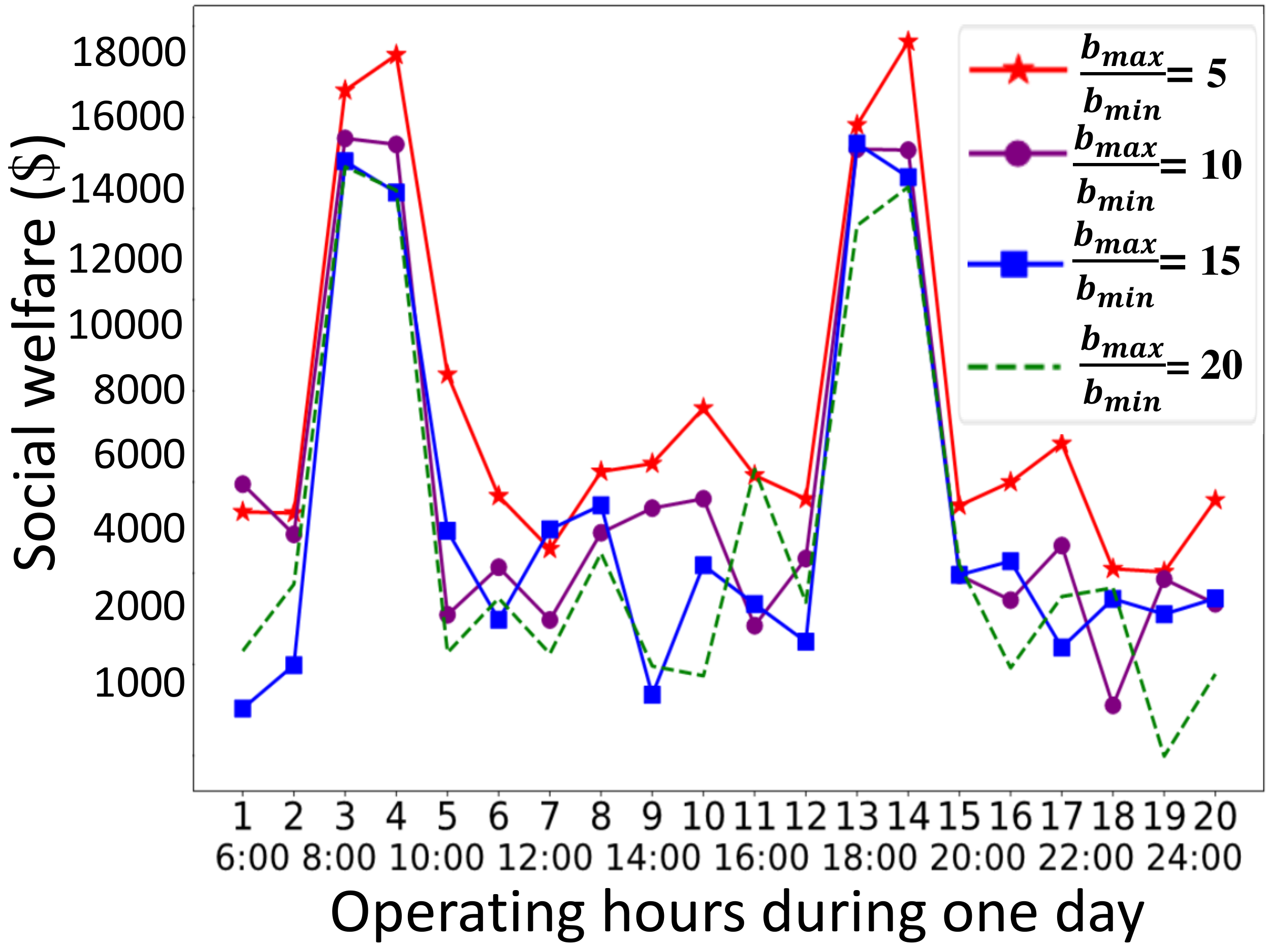}\label{fig8a}}
    \subfloat[Quadratic unit price function]{\includegraphics[width=0.33\textwidth]{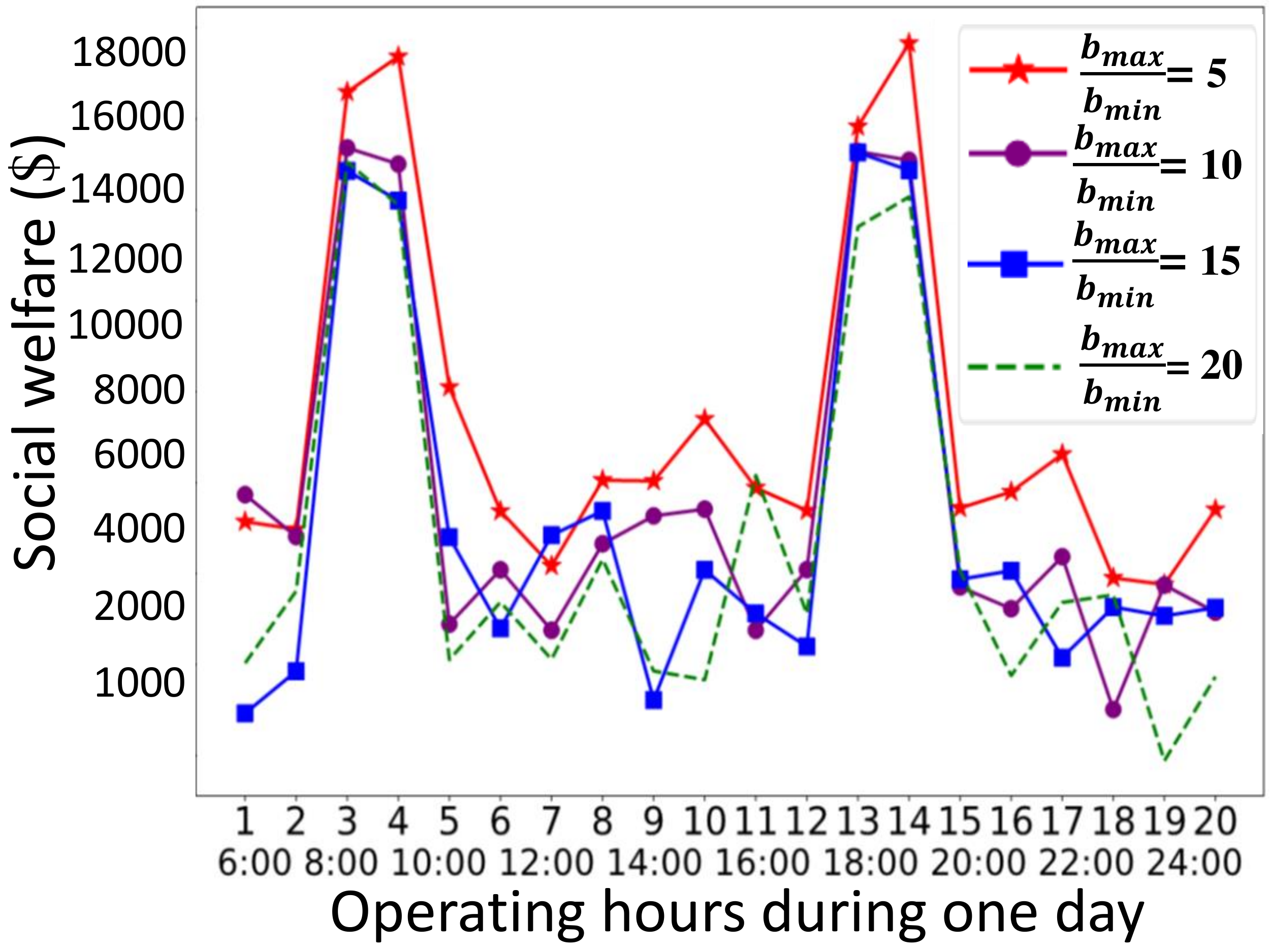}\label{fig8b}}
    \subfloat[Exponential unit price function]{\includegraphics[width=0.33\textwidth]{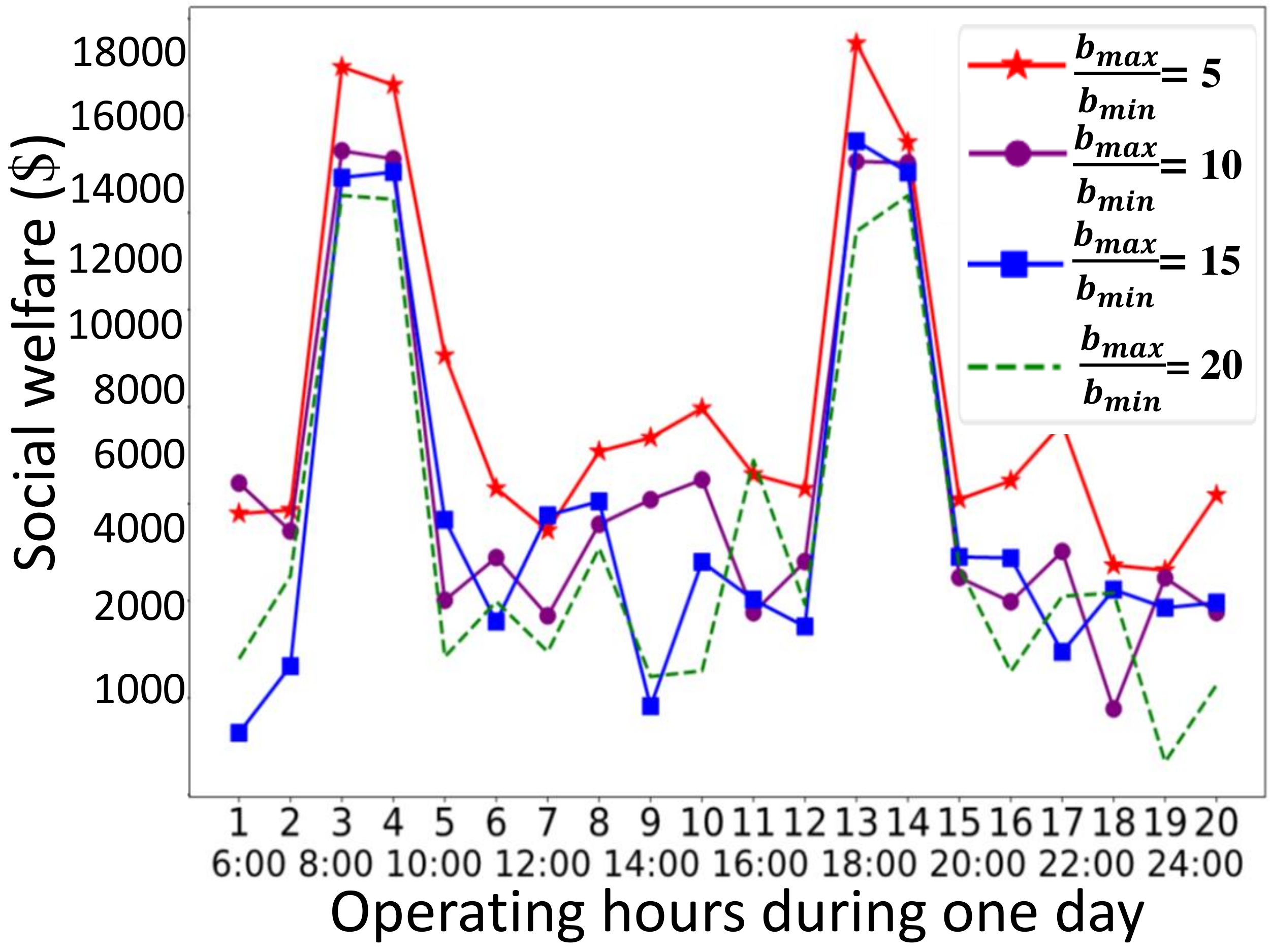}\label{fig8c}}
	\caption{Sensitivity analysis on different types of pricing functions under different $b_{max}/b_{min}$}\label{fig8}
\end{figure*}
\begin{figure}[ht!]
	\subfloat[Sensitivity analysis on capacity ($C$)]{\includegraphics[width=0.45\textwidth]{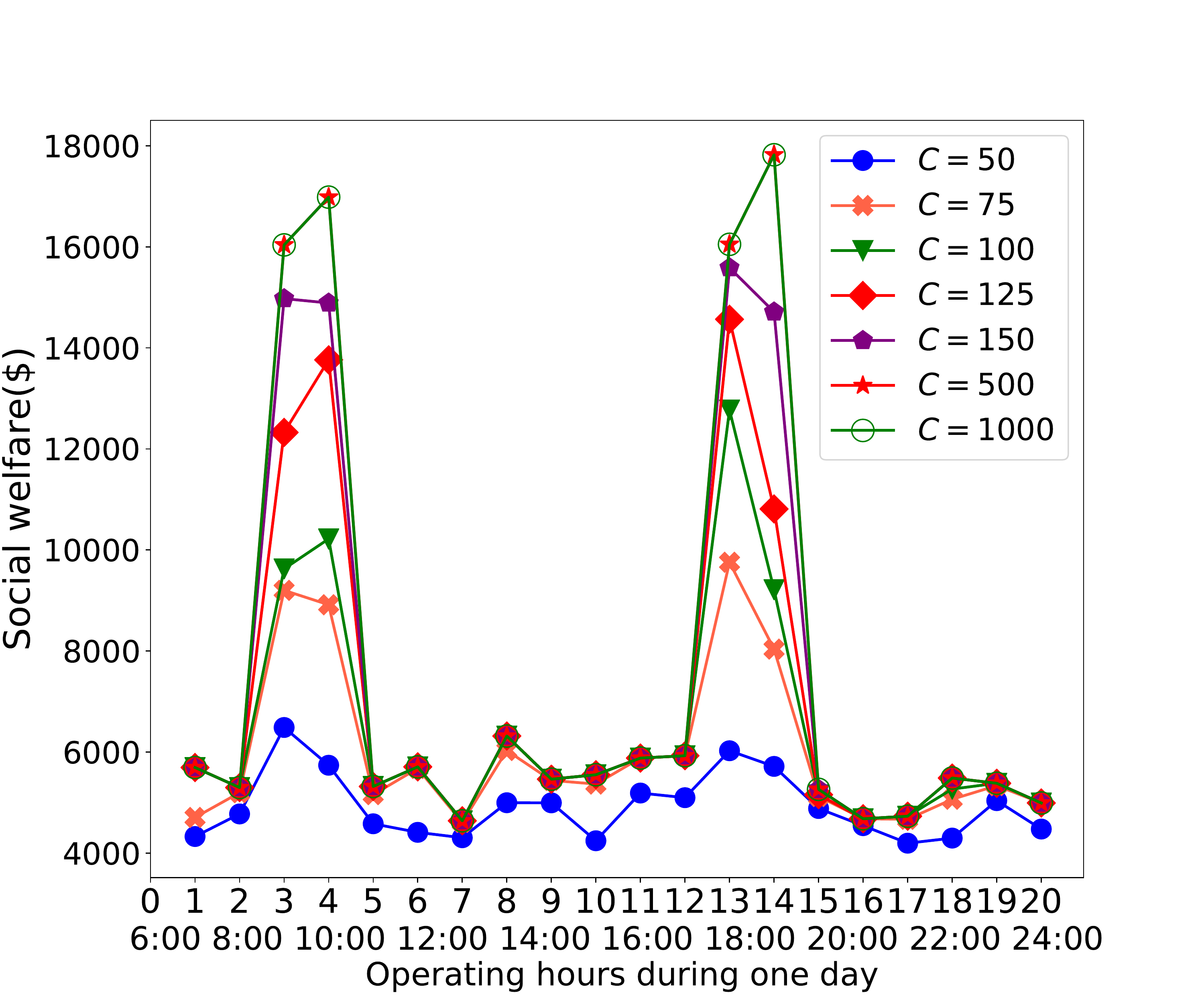}\label{figcap}}
    \subfloat[Sensitivity analysis on speed factor ($\Delta v$)]{\includegraphics[width=0.45\textwidth]{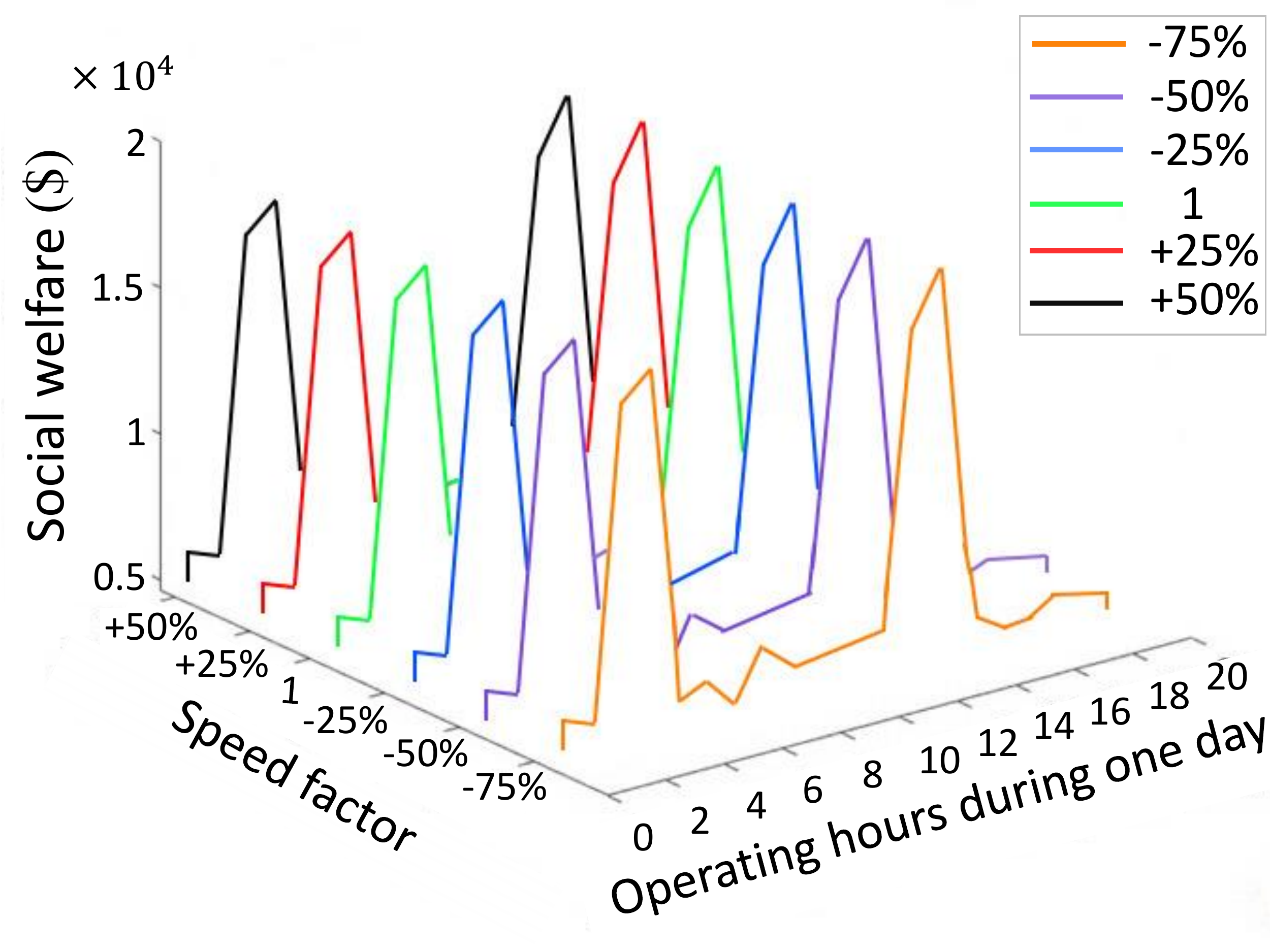}\label{figSpeed}}
	\caption{Sensitivity analysis on capacity ($C$) and speed factor ($\Delta v$)}
	\label{figCS}
\end{figure}

\subsubsection{Comparison of the PAYG and PAAP mechanisms}
\label{722}
In this subsection, we compare the daily social welfare achieved using the PAYG and PAAP mechanisms, denoted by $S_{PAYG}$ and $W_{PAAP}$, respectively. The travel demand at time slot $t$ in the PAYG mechanism ($\lambda_{t}$) and in the PAAP mechanism ($\kappa_{t}$) are given in \T\ref{T5} and \T\ref{T7}, respectively. Since each time slot represents one day in the PAAP mechanism and represents one minute in the PAYG mechanism, travel demand at time slot $t$ in PAAP mechanism ($\kappa_{t}$) is set as the summation of $\lambda_{t}$ over 1200 time slots on workdays,  ($\kappa_{t}=\sum_{t=1}^{1200}\lambda_{t}$), and is set as $\kappa_{t}=h\sum_{t=1}^{1200}\lambda_{t}$ on weekends, where $h$ is randomly set within $[$40$\%,$80$\%]$. The time period ($L_{i}$) is set as 1, 5, 6 and 7 in different weeks, and user data for the PAAP mechanism are obtained based on that in the PAYG  mechanism by multiplying the parameters ($T_{ij}, b_{ij}, \Phi_{i}, \Gamma_{i}, \underline{Q}_{i}$) with $L_{i}$. Let $W_{PAYG}$ denote the social welfare of one time slot (min) in the PAYG mechanism. If $L_{i}=1$, $W_{PAAP}$ corresponds to the summation of $W_{PAYG}$ over 1200 time slots, $S_{PAYG}=\sum_{t=1}^{1200}W_{PAYG}$. 

\begin{table*}[!ht]
\footnotesize
\setlength{\abovecaptionskip}{0pt}
\setlength{\belowcaptionskip}{0pt}
\centering
\caption{Parameters setting in the PAAP mechanism}
\begin{tabular}{ccccccccc}
\toprule
\multirow{2}{*} { Day } & \multicolumn{2}{c} { Week 1 } & \multicolumn{2}{c} { Week 2 } & \multicolumn{2}{c} { Week 4} & \multicolumn{2}{c} { Week 3 (6,7,8) } \\
\cline { 2 -9 } &  weekday & weekend & weekday & weekend & weekday & weekend& weekday & weekend\\
\midrule
$t$ & $[1\cdots,5]$ & [6,7] & [$8,\cdots,12$] & [13,14]& $[22,\cdots,26]$& [27,28]& [$15,\cdots,19$] &[20,21]\\
$L_{i}$ & 1 & 1 &7 & 7 & 6 &6 &5 &5 \\
\bottomrule
\end{tabular}
\label{T7}
\end{table*}
We execute Algorithms \ref{alg1} and \ref{alg2} and report results in Figures \ref{Fig10} and \ref{Fig11}. \Fig \ref{Fig10a} shows that the social welfare during morning peak hours (time slot $120-240$) and evening peak hours (time slot $720-840$) is considerably higher than that in non-peak hours using the PAYG mechanism.
\begin{figure*}[!ht]
\centering
	\subfloat[$t-W_{PAYG}$ ]{\includegraphics[width=0.45\textwidth]{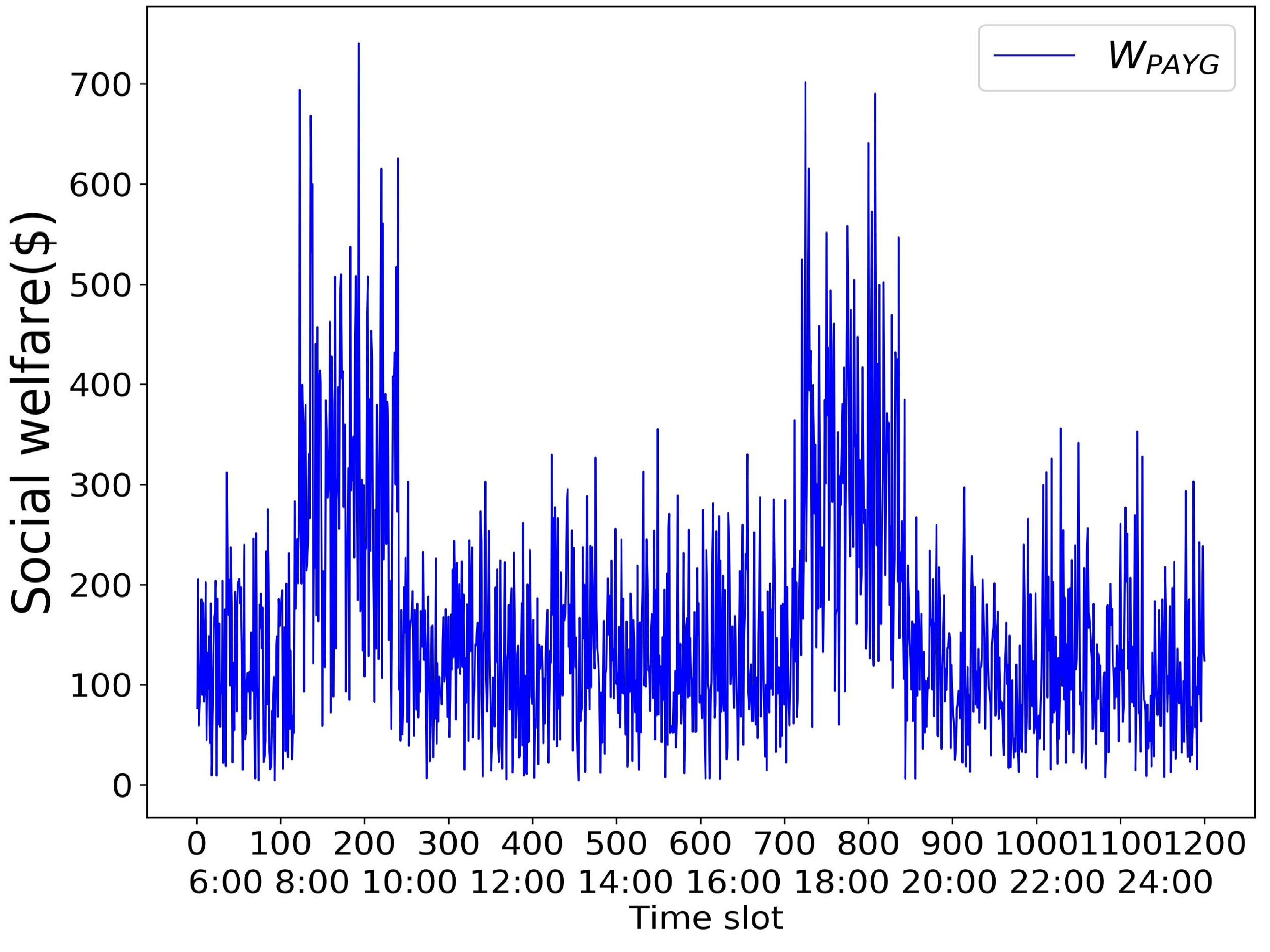}\label{Fig10a}}
    \subfloat[$t-p_{t}$ and $t-A_{t}$]{\includegraphics[width=0.45\textwidth]{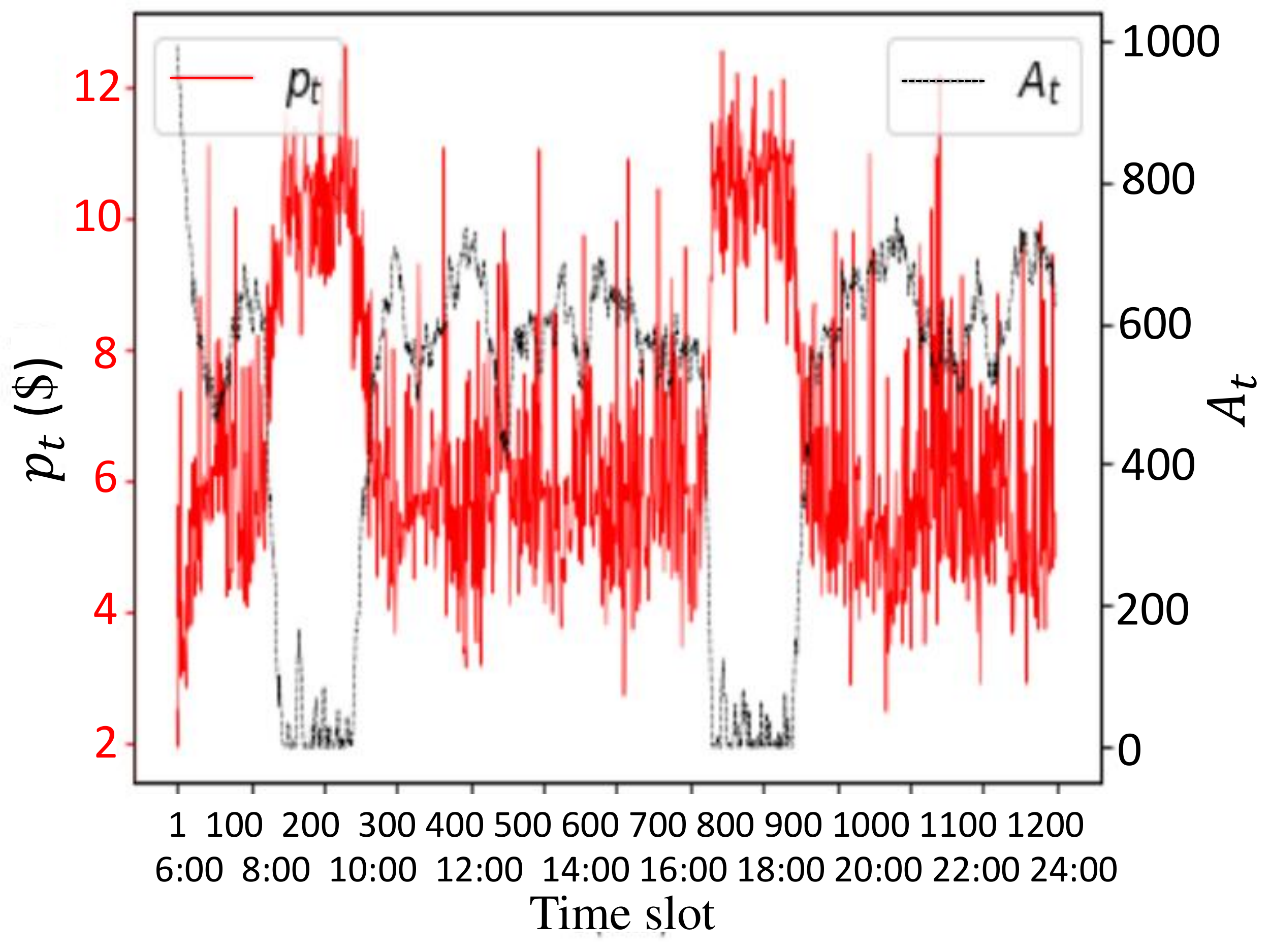}\label{Fig10b}}
	\caption{Social welfare ($W_{PAYG}$), unit price ($p_{t}$) and available mobility resources ($A_{t}$) in the PAYG mechanism}\label{Fig10}
\end {figure*}
 \begin{figure*}[ht!]
 \centering
	\subfloat[$t-W_{PAAP}$]{\includegraphics[width=0.46\textwidth]{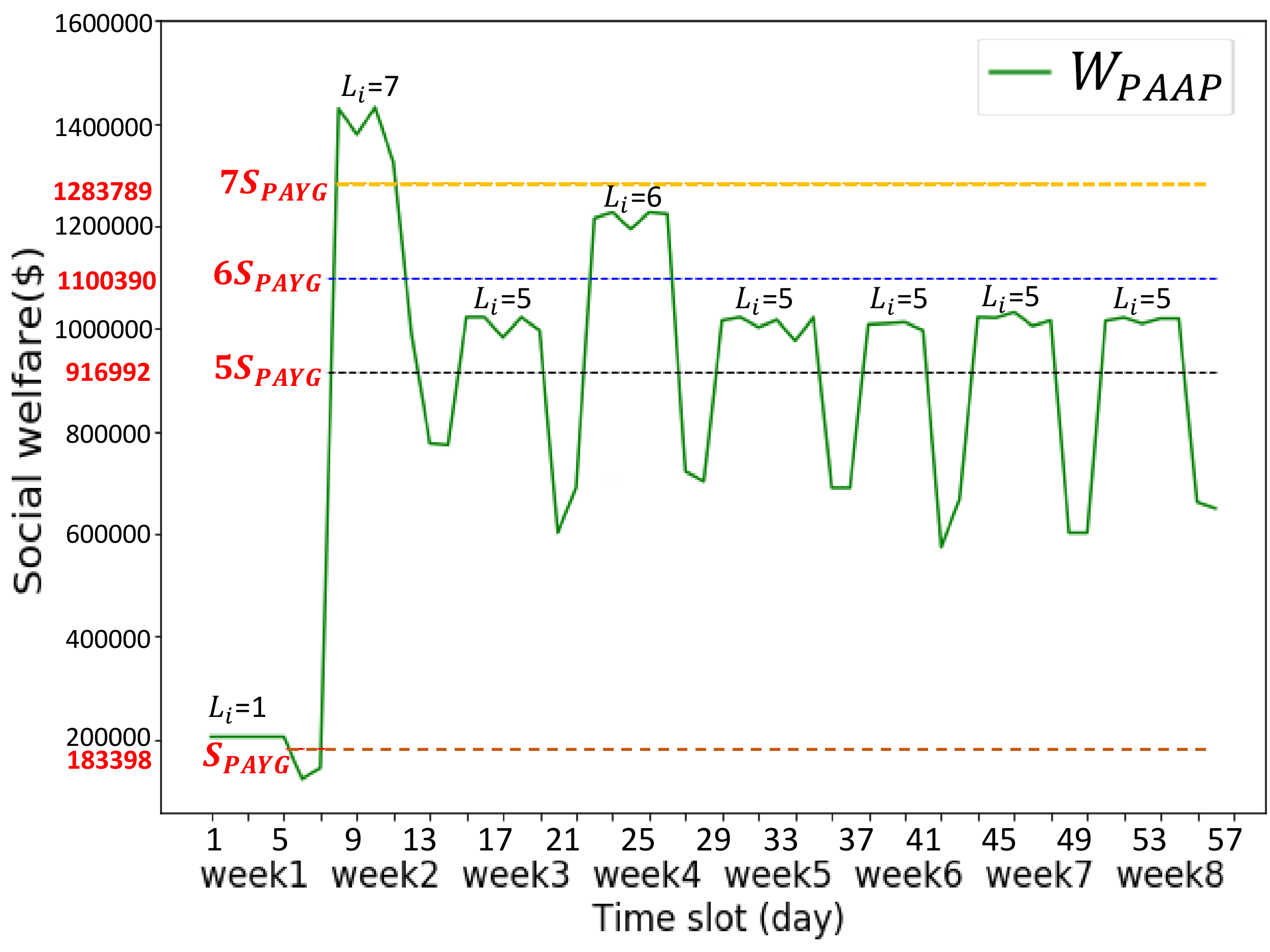}\label{Fig11a}}
    \subfloat[$t-p_{t}$ and $t-A_{t}$]{\includegraphics[width=0.46\textwidth]{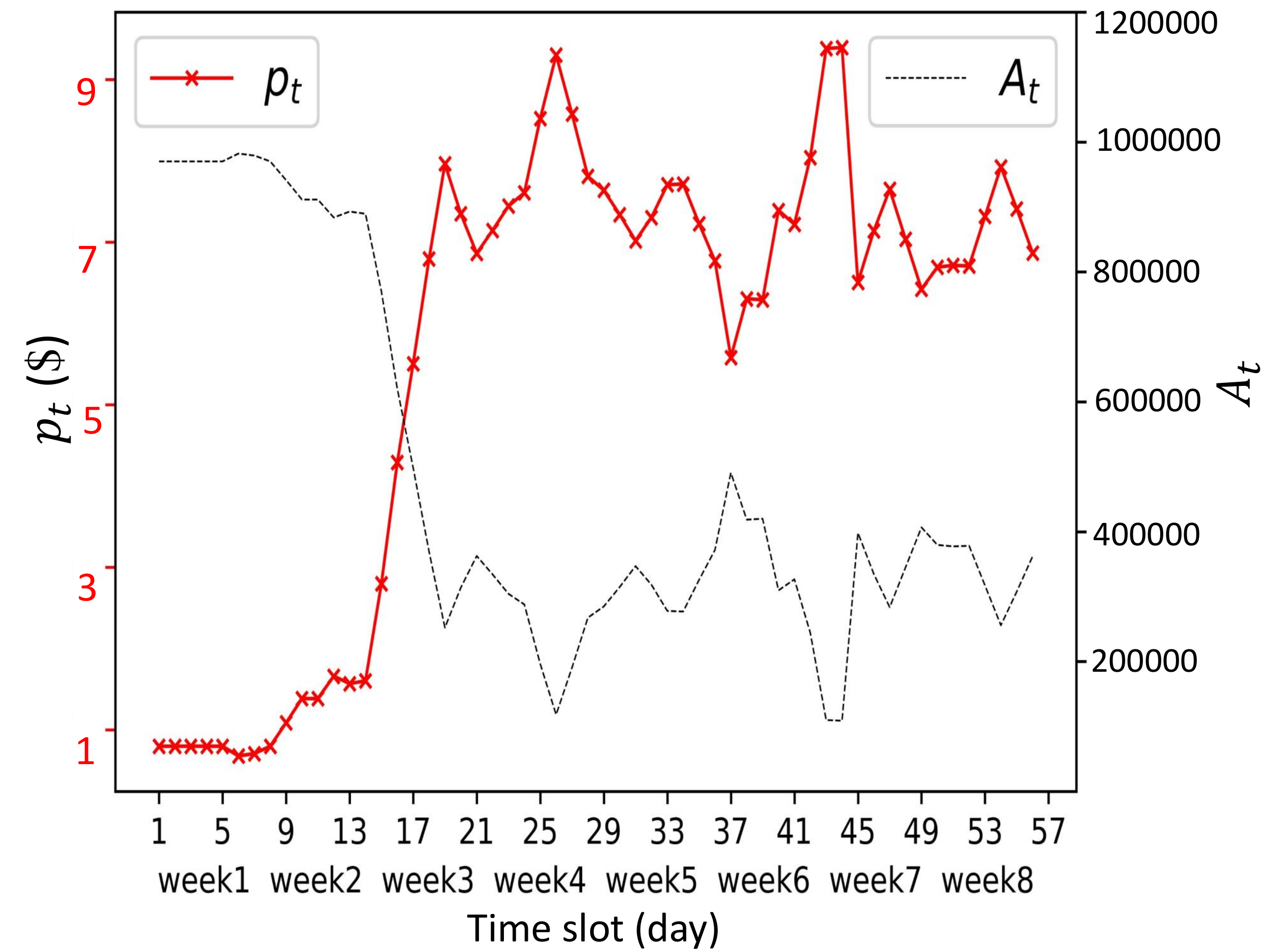}\label{Fig11b}}
	\caption{Social welfare ($W_{PAAP}$), unit price ($p_{t}$) and available mobility resources ($A_{t}$) in the PAAP mechanism}\label{Fig11}
\end {figure*}

Comparing the social welfare obtained over time period $L_{i}$ using the PAYG mechanism, $L_{i}\cdot S_{PAYG}$, with that obtained using the PAAP mechanism, $W_{PAAP}$ (e.g., comparing $S_{PAYG} (\$183398$) with $W_{PAAP}$ in week 1 ($L_{i}=1$) and $7 S_{PAYG}$ (\$1283789) with $W_{PAAP}$ in week 2 ($L_{i}=7$)), we find that the social welfare obtained using the PAAP mechanism is higher than that using the PAYG mechanism (\Fig\ref{Fig11a}). Moreover, \Fig \ref{Fig10b} and \Fig \ref{Fig11b} show that the unit price at time slot $t$ exhibits an opposite pattern against the available resources at time slot $t$ due to the time-varying pricing rule, and the range of unit price ($p_{t}$) in the PAAP mechanism (\$1$\sim$\$9) is lower than $p_{t}$ in the PAYG mechanism (\$2$\sim$\$12). 

\subsubsection{Competitive ratio analysis}
\label{723}
We analyse the competitive ratio ($\Theta$) given in \textcolor{Cerulean}{Proposition}\ref{theorem2} (\textcolor{Cerulean}{Proposition}\ref{theorem4}) by comparing with the social welfare ratio ($\mathcal{R}$) defined as the ratio of social welfare obtained by \Alg \ref{alg1} (\Alg\ref{alg2}) to the corresponding offline problem \textcolor{Cerulean}{Model 2.1} (\textcolor{Cerulean}{Model 2.2}).
\begin{figure*}[ht]
	\subfloat[$|\mathcal{J}|-Nt-\Theta$]{\includegraphics[width=0.33\textwidth]{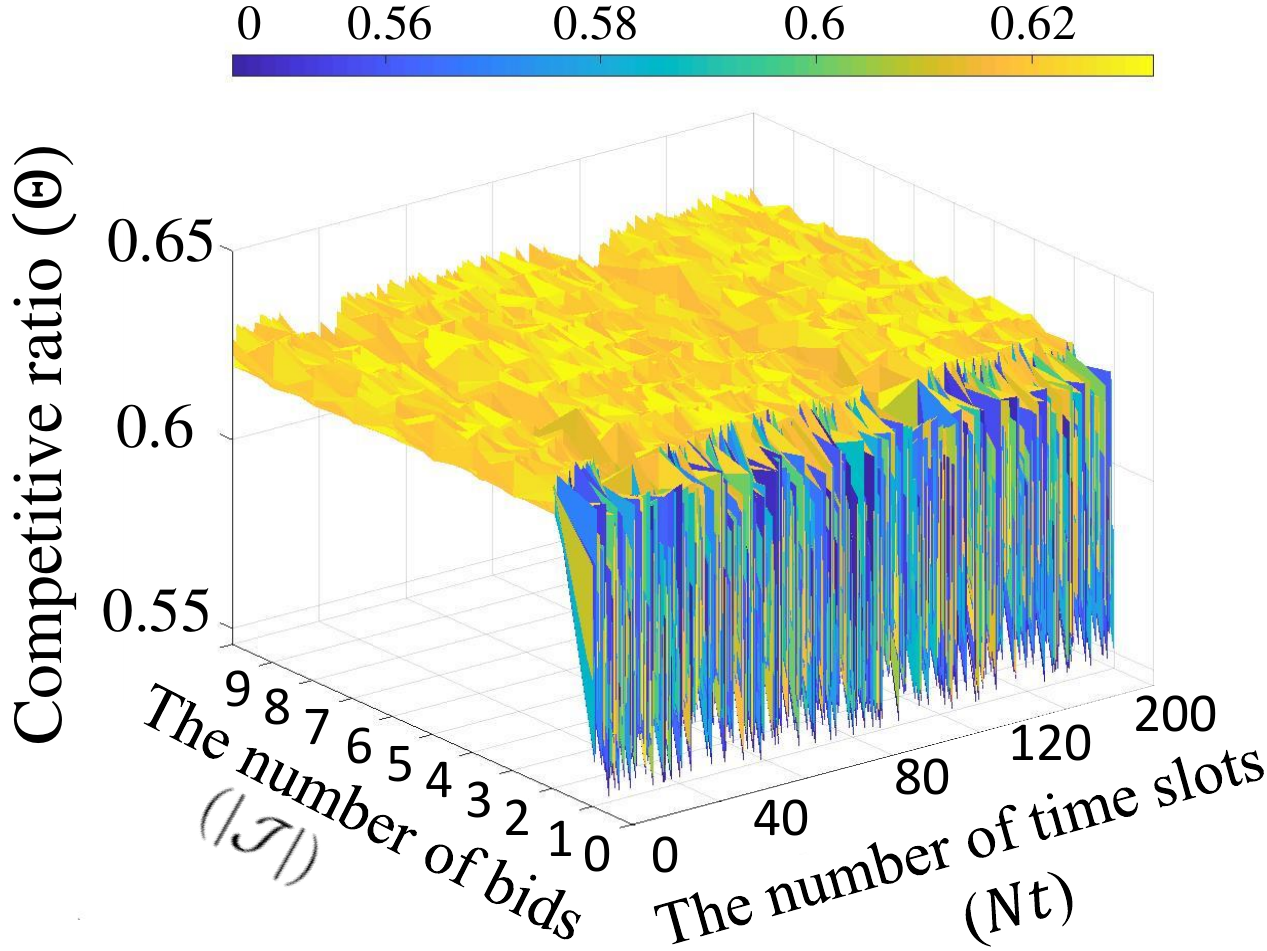}\label{fig12a}}
    \subfloat[$|\mathcal{J}|-\Theta$]{\includegraphics[width=0.33\textwidth]{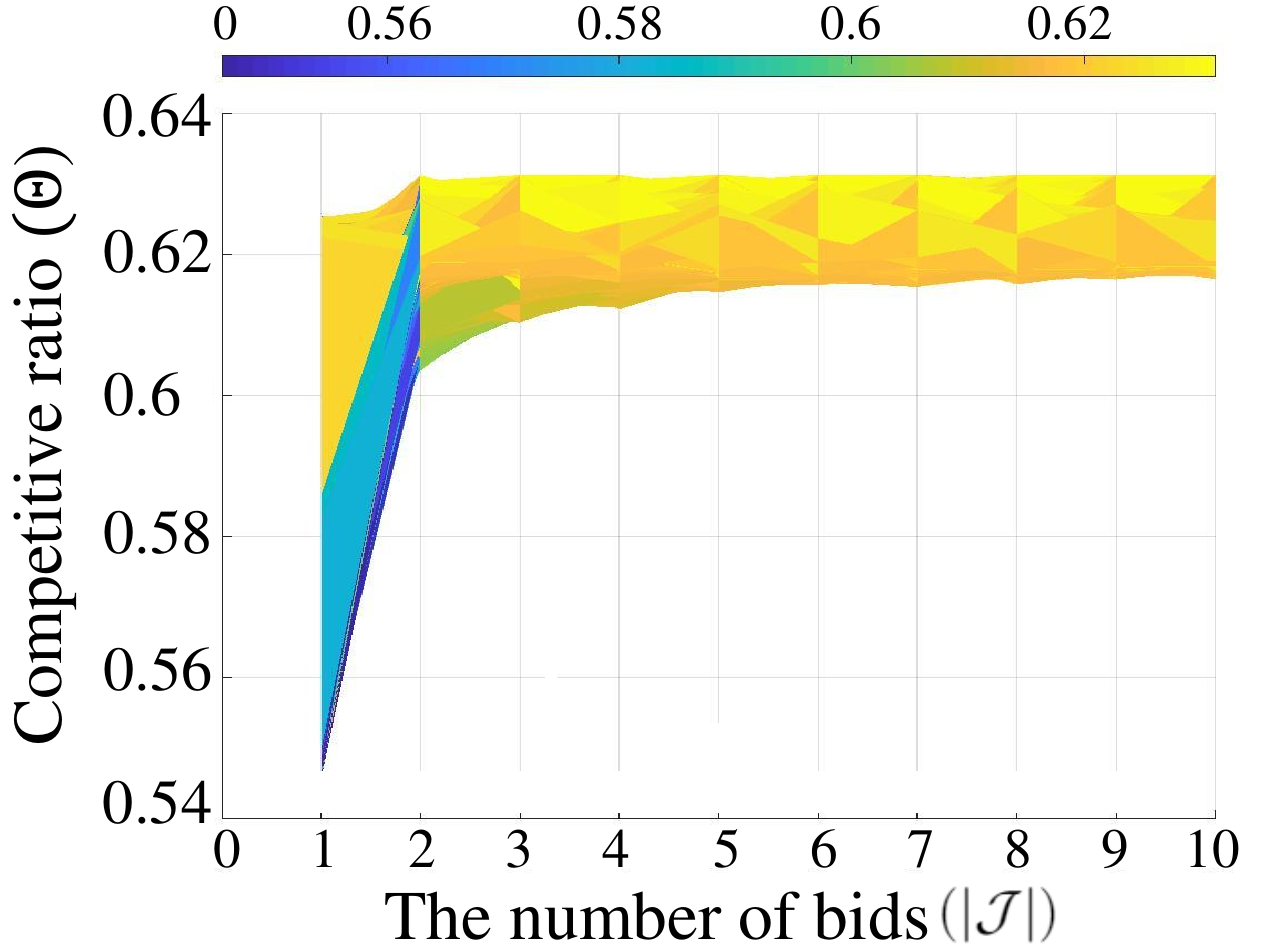}\label{fig12b}}
    \subfloat[$Nt-\Theta$]{\includegraphics[width=0.33\textwidth]{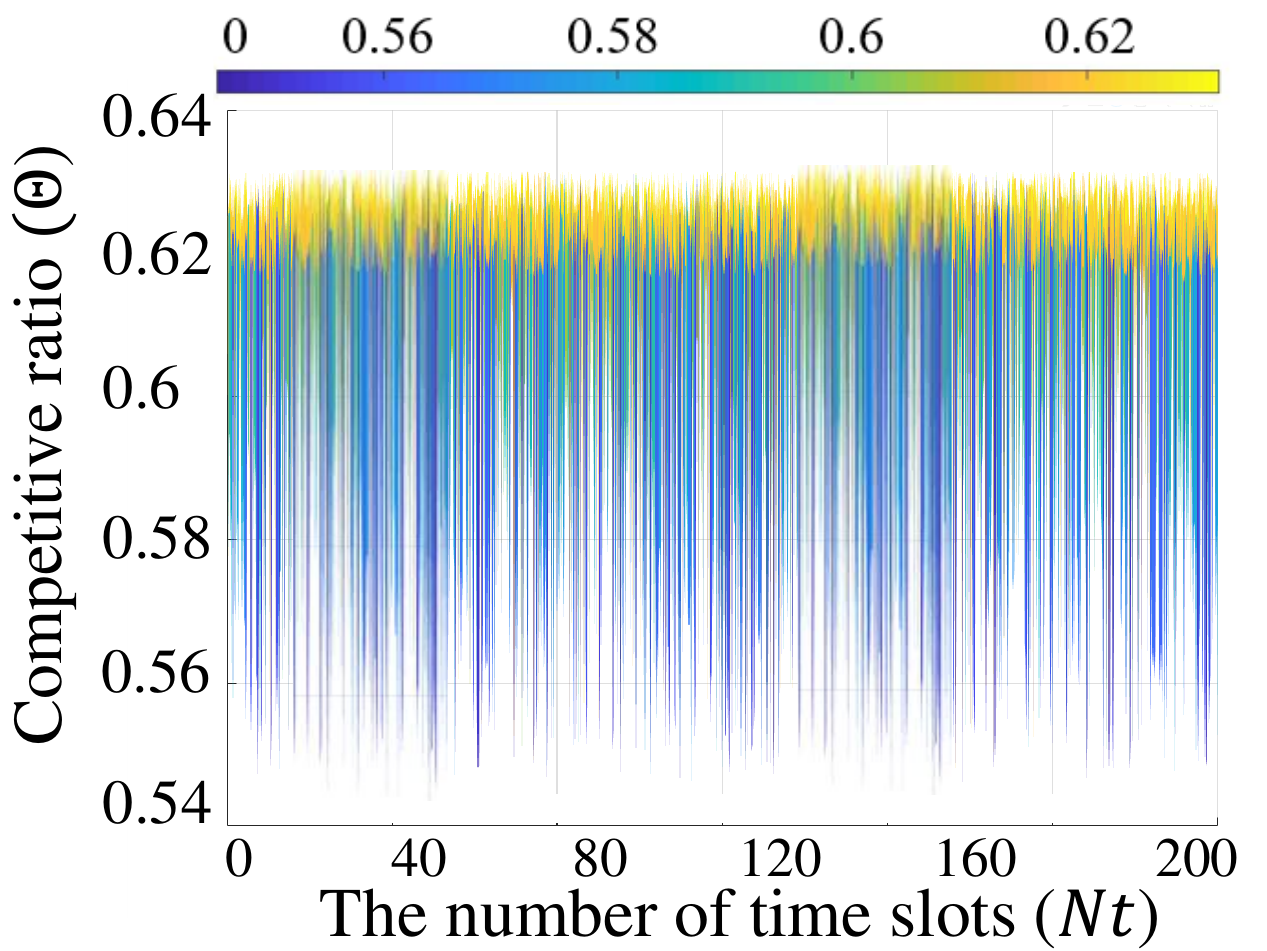}\label{fig12c}}
	\caption{Competitive ratio ($\Theta$), the number of bids ($|\mathcal{J}|$) and the number of time slots ($Nt$) in the PAYG mechanism }\label{fig12}
\end{figure*}
\begin{figure*}[ht!]
	\subfloat[$|\mathcal{J}|-Nt-\mathcal{R}$]{\includegraphics[width=0.33\textwidth]{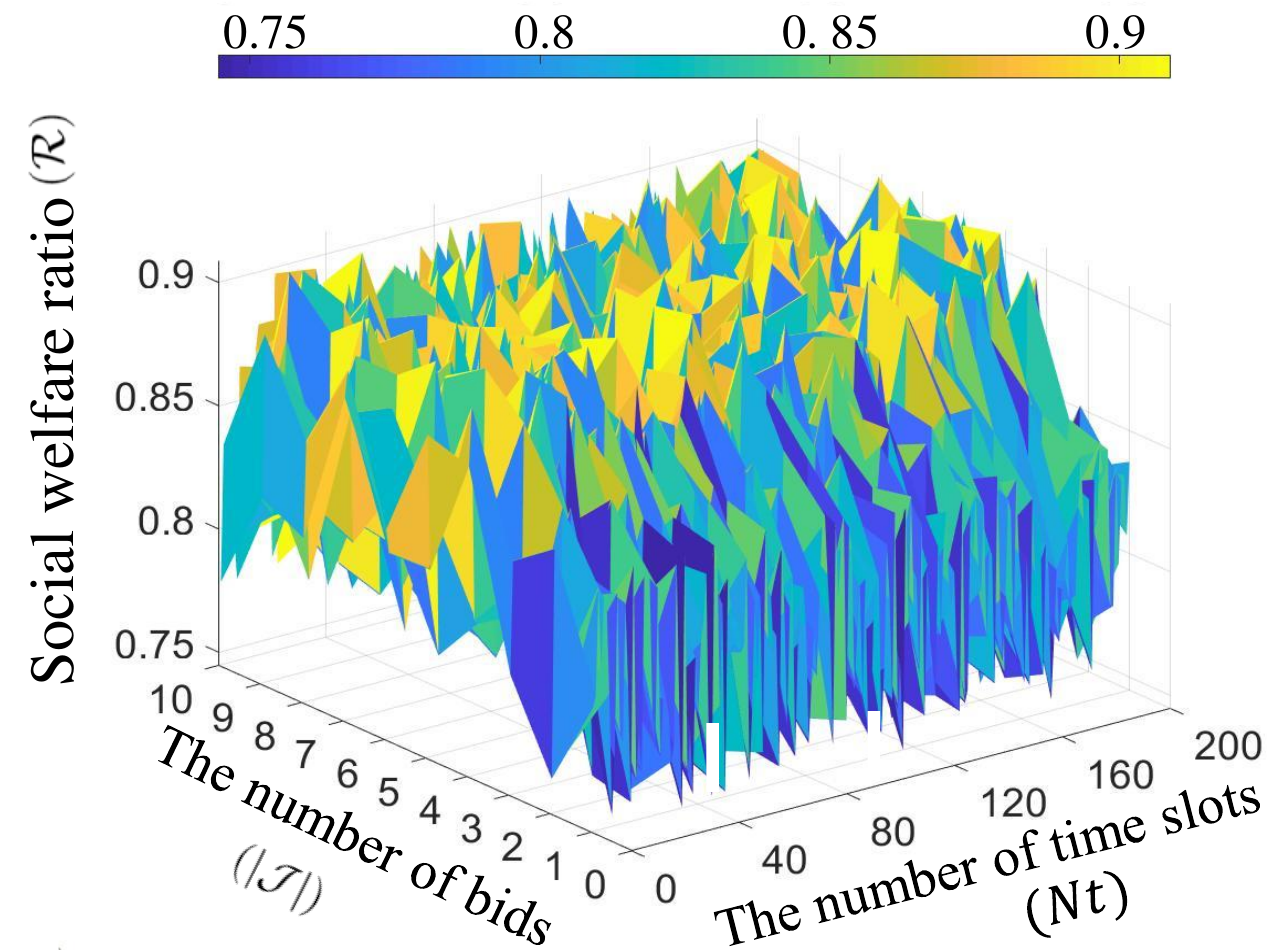}\label{fig13a}}
    \subfloat[$|\mathcal{J}|-\mathcal{R}$]{\includegraphics[width=0.33\textwidth]{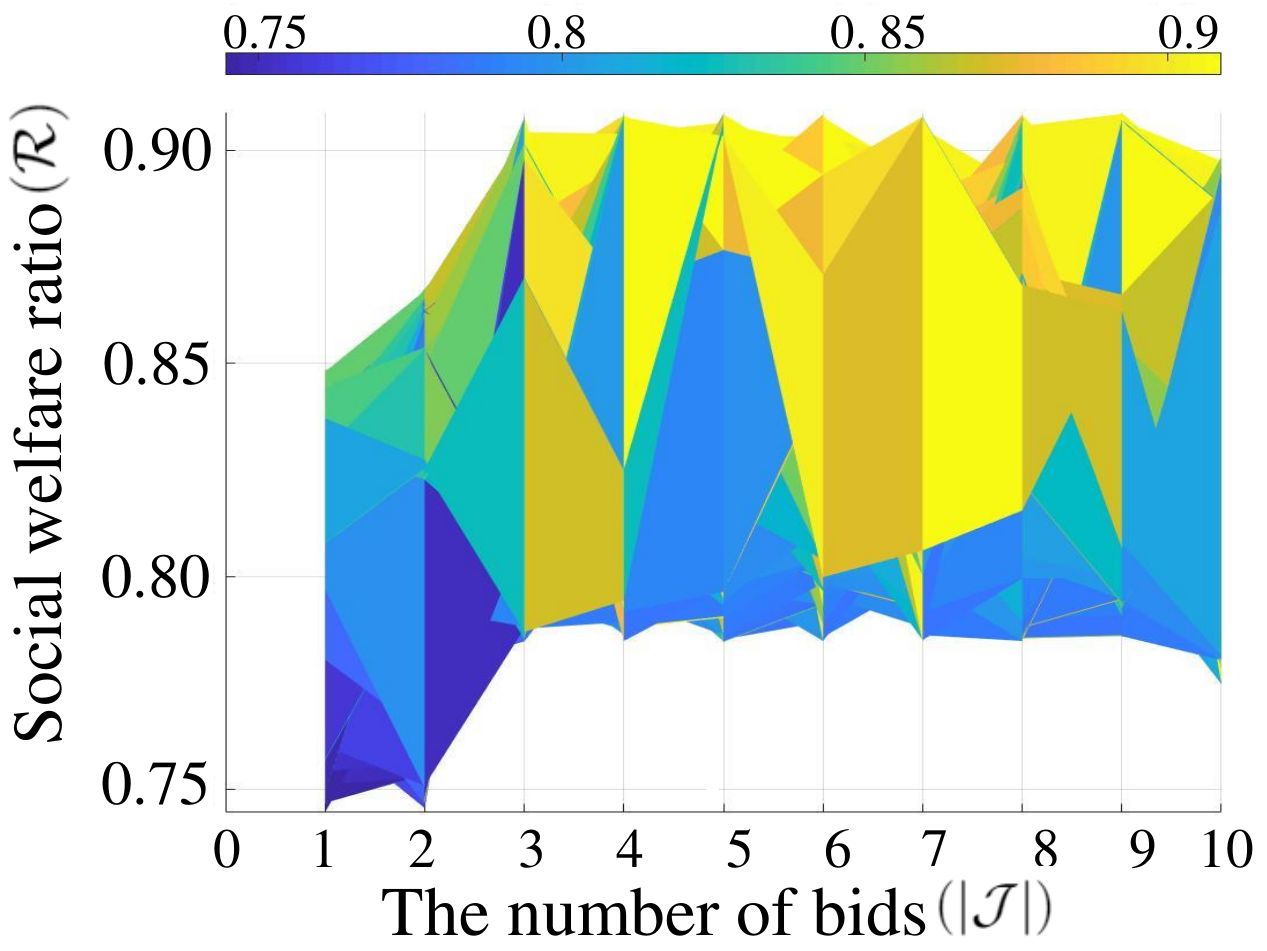}\label{fig13b}}
    \subfloat[$Nt-\mathcal{R}$]{\includegraphics[width=0.33\textwidth]{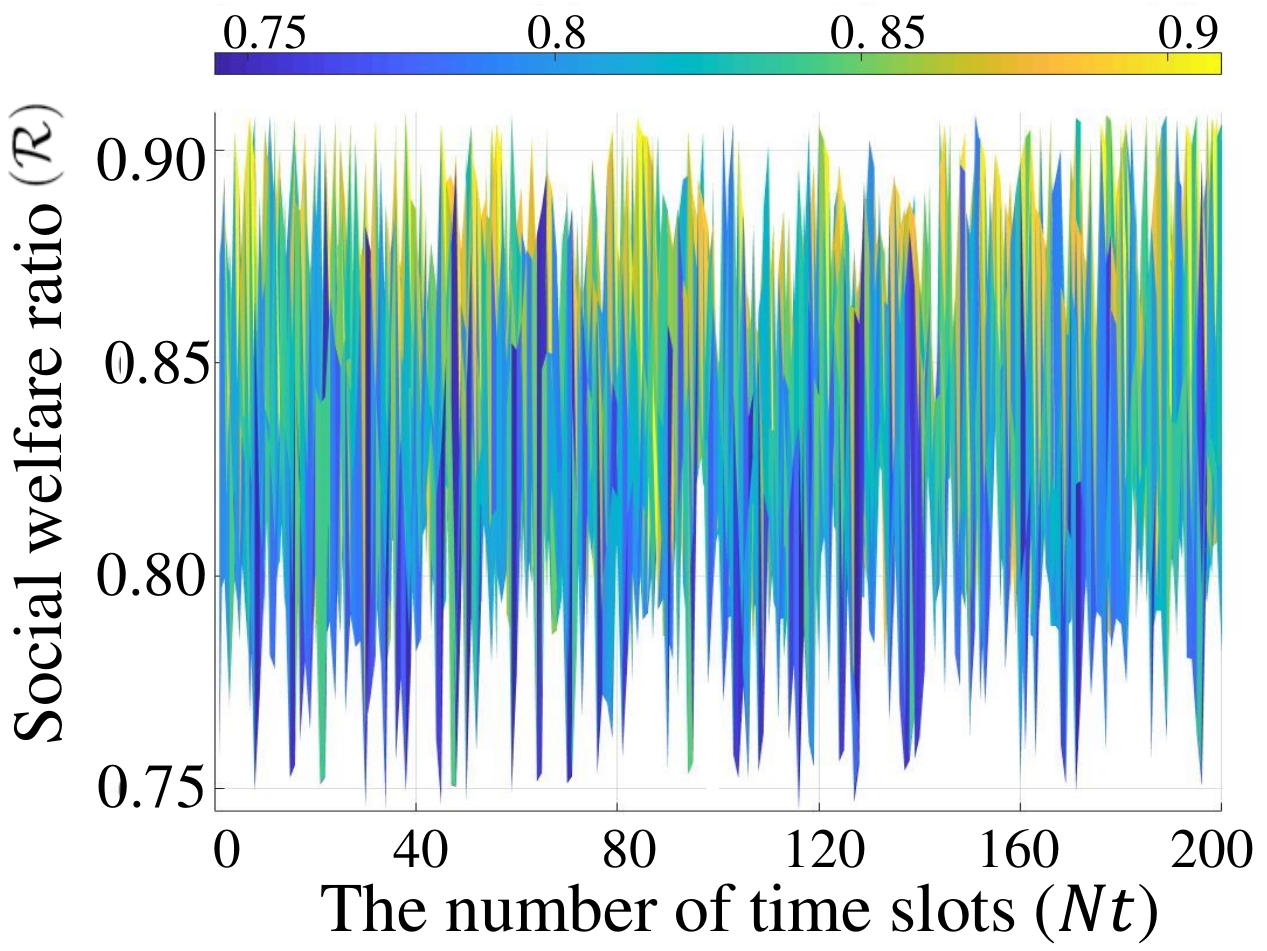}\label{fig13c}}
	\caption{Social welfare ratio ($\mathcal{R}$), the number of bids ($|\mathcal{J}|$) and the number of time slots ($Nt$) in  PAYG mechanism}\label{fig13}
\end{figure*}

In the PAYG mechanism, the number of bids ($|\mathcal{J}|$) is assumed to be comprised between $1\sim 10$, and the travel demand at each time slot is given in \T\ref{T7}. \Fig \ref{fig12a} shows the relationship among the value of the competitive ratio ($\Theta$), the number of bids ($|\mathcal{J}|$) and the number of time slots ($Nt$). \Fig \ref{fig12b} shows that if $|\mathcal{J}|=1$, the value of $\Theta$ is within $0.54511 \sim 0.62145 $; if $|\mathcal{J}|=2$, the value of $\Theta$ is within $0.60109 \sim 0.63128$; whereas if $|\mathcal{J}|=3,4,\cdots,10$, the value of $\Theta$ is within $0.61565 \sim 0.63128$. \Fig \ref{fig12c} shows that competitive ratio ($\Theta$) is independent on the number of time slots ($Nt$). \Fig \ref{fig13a} shows the relationship among social welfare  ratio ($\mathcal{R}$), the number of bids ($|\mathcal{J}|$) and the number of time slots ($Nt$). \Fig \ref{fig13b} shows that if $|\mathcal{J}|=1$, the value of $\mathcal{R} $ is within $0.74451\sim 0.84882$, if $|\mathcal{J}|=2$, the value of $\mathcal{R}$ is within $0.74125\sim 0.86158$, whereas if $|\mathcal{J}|=3,4,\cdots,10$, the value of $\mathcal{R}$ is within $0.78451 \sim 0.90882$. \Fig \ref{fig13c} shows that the value of social welfare ratio ($\mathcal{R}$) is independent on the number of time slots ($Nt$). \Fig \ref{fig12} and \Fig \ref{fig13} show that the competitive ratio ($\Theta$) given in \textcolor{Cerulean}{Proposition} \ref{theorem2} can always provide a lower bound for the social welfare ratio ($\mathcal{R}$) under different size of input time sequence. The gap between $\mathcal{R}$ and $\Theta$ is within $0.19614\sim 0.27754$.
\begin{figure*}[ht]
	\subfloat[$|\mathcal{J}|-Nt-\Theta$]{\includegraphics[width=0.33\textwidth]{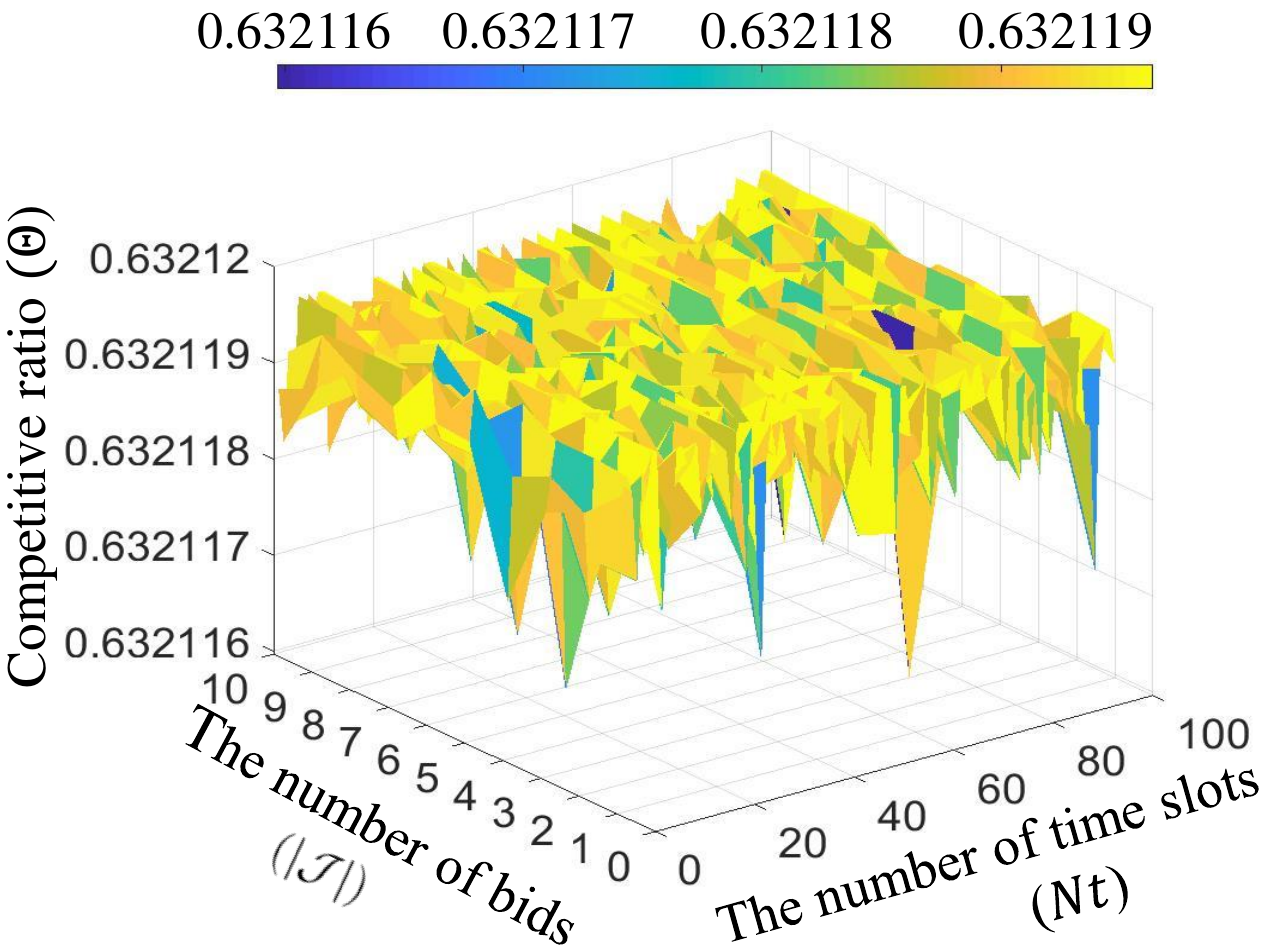}\label{fig14a}}
    \subfloat[$|\mathcal{J}|-\Theta$]{\includegraphics[width=0.33\textwidth]{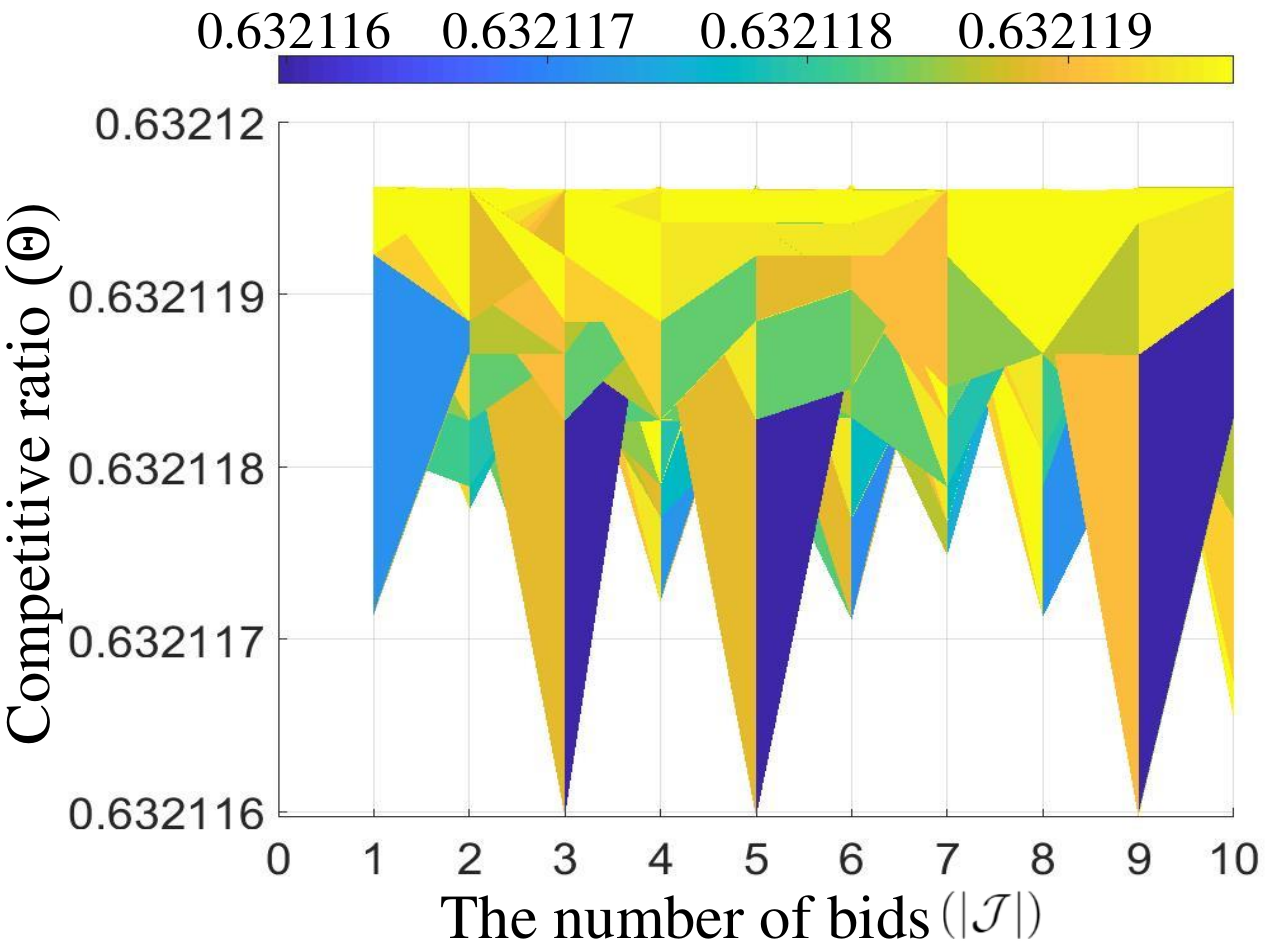}\label{fig14b}}
    \subfloat[$Nt-\Theta$]{\includegraphics[width=0.33\textwidth]{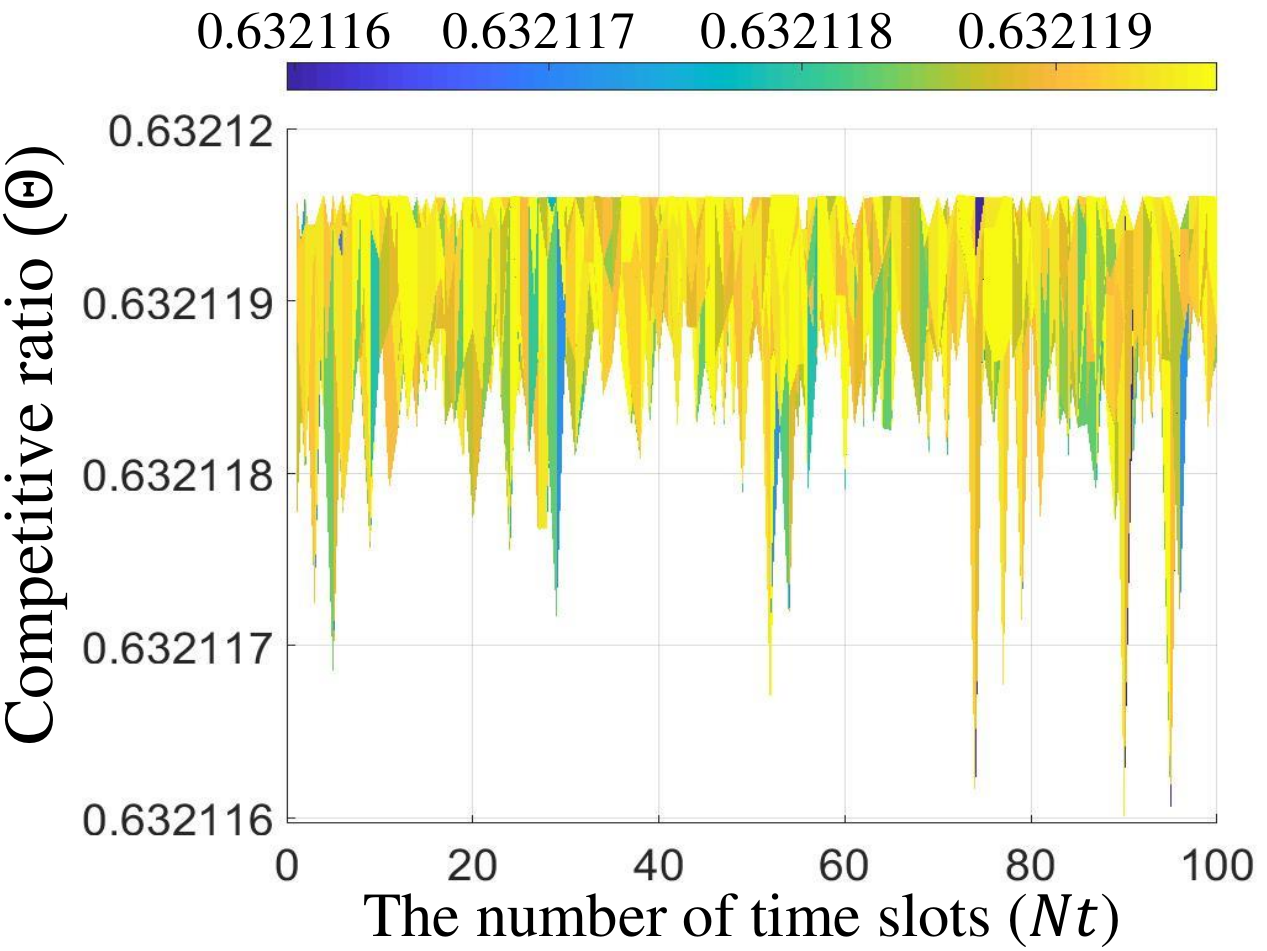}\label{fig14c}}
	\caption{Competitive ratio ($\Theta$), the number of bids ($|\mathcal{J}|$) and the number of time slots ($Nt$) in PAAP mechanism}\label{fig14}
\end{figure*}
\begin{figure*}[ht!]
	\subfloat[$|\mathcal{J}|-Nt-\mathcal{R}$]{\includegraphics[width=0.33\textwidth]{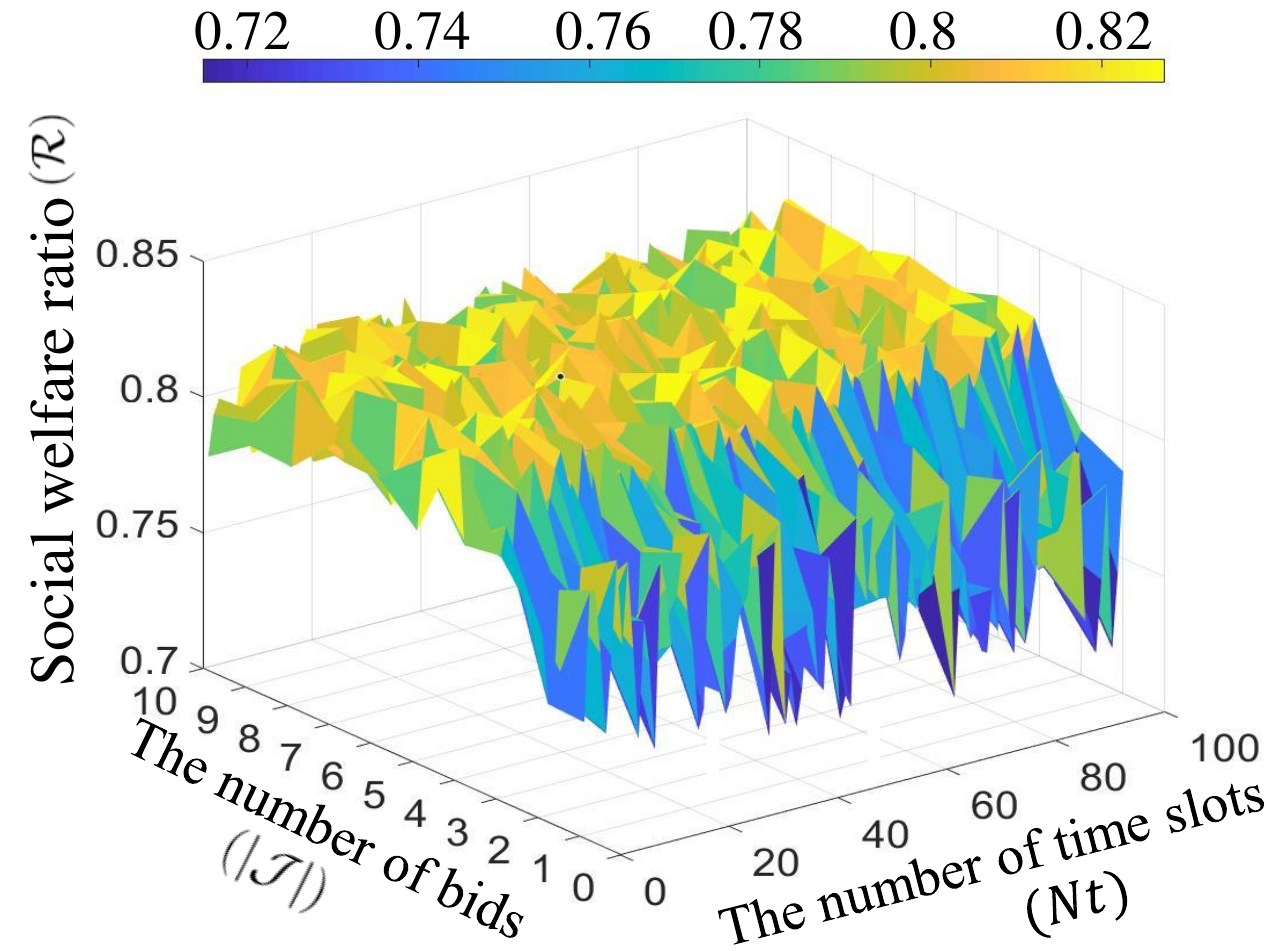}\label{fig15a}}
    \subfloat[$|\mathcal{J}|-\mathcal{R}$]{\includegraphics[width=0.33\textwidth]{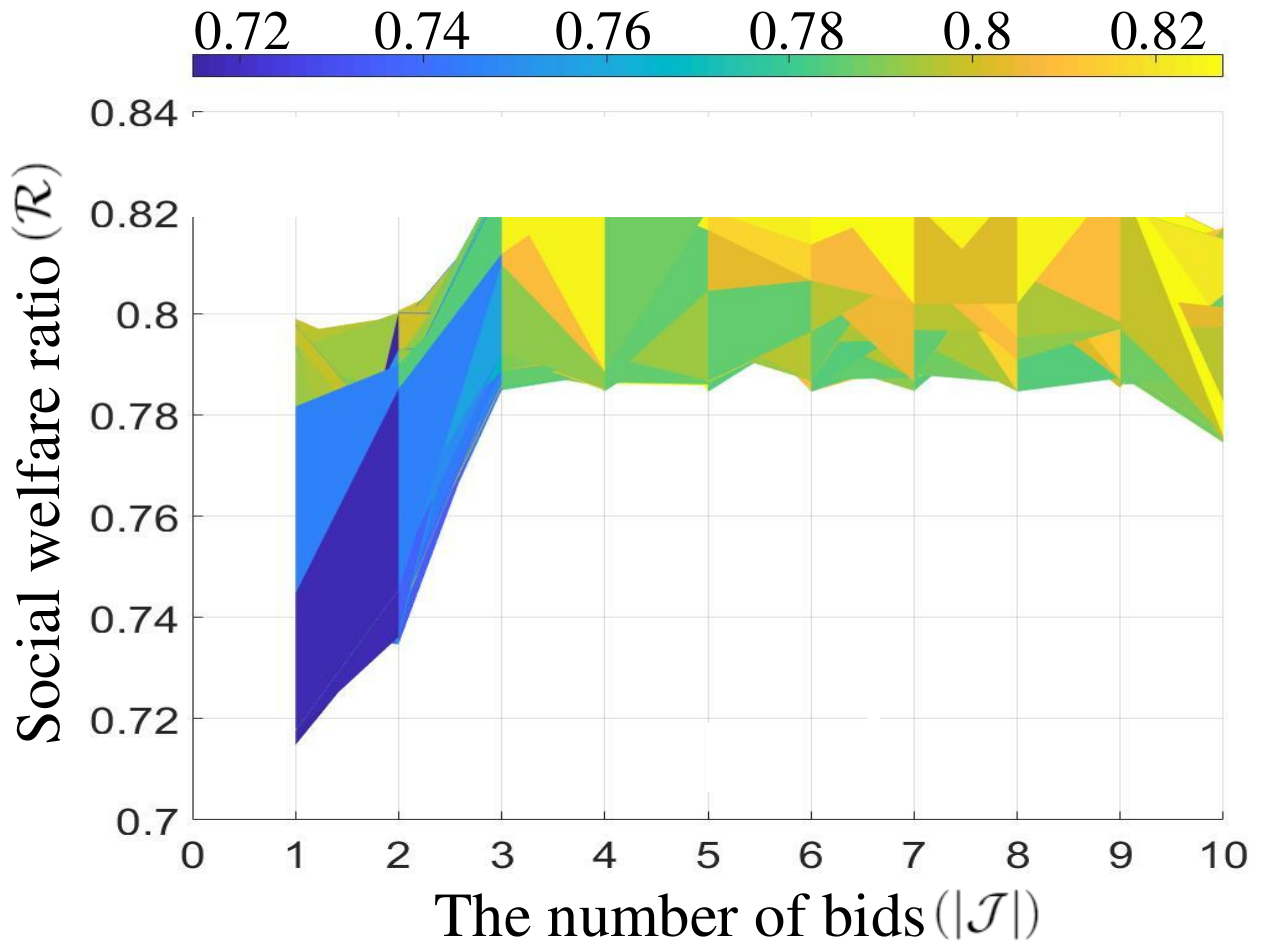}\label{fig15b}}
    \subfloat[$Nt-\mathcal{R}$]{\includegraphics[width=0.33\textwidth]{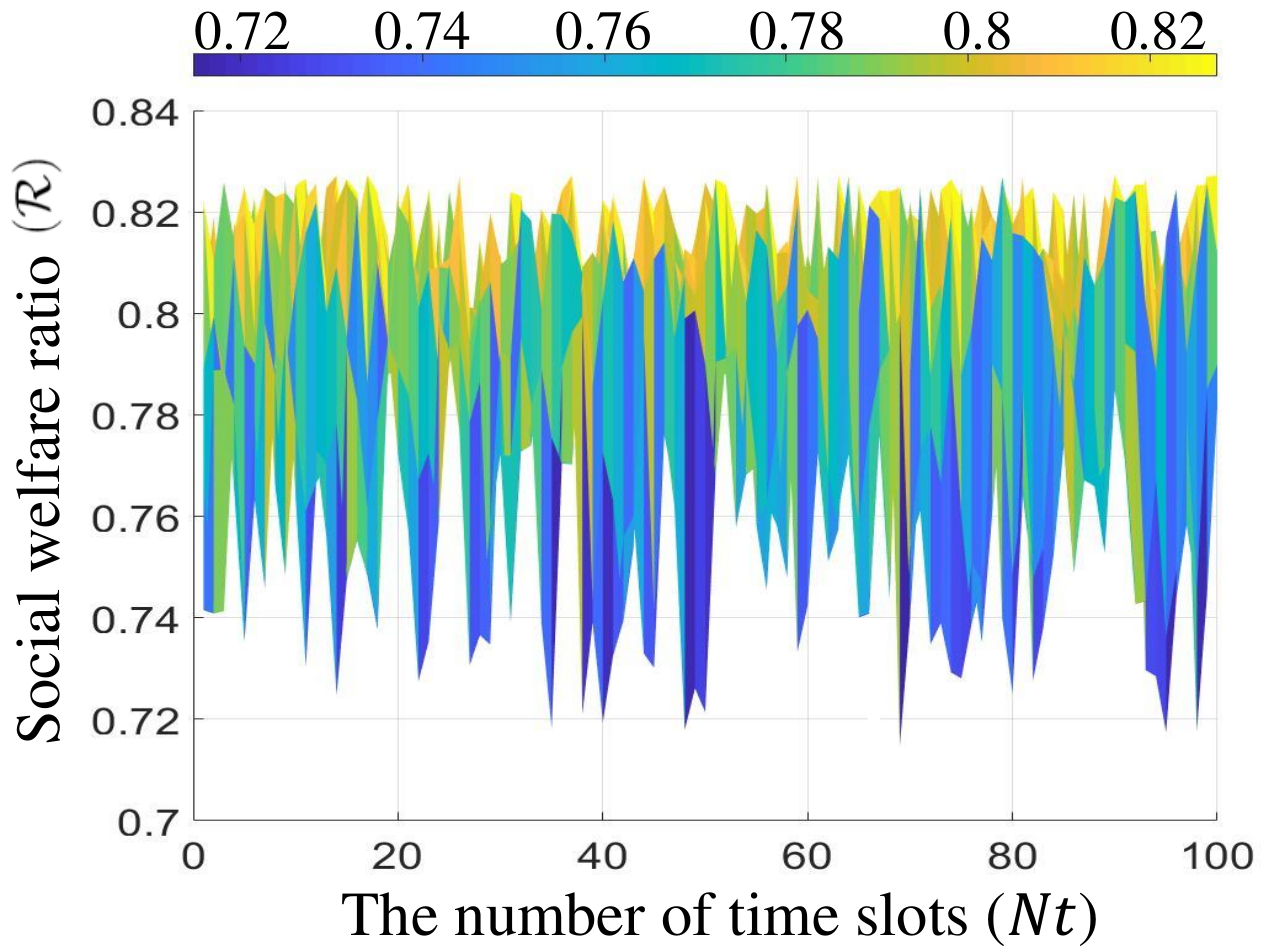}\label{fig15c}}
	\caption{Social welfare  ratio ($\mathcal{R}$), the number of bids ($|\mathcal{J}|$) and the number of time slots ($Nt$) in  PAAP mechanism}\label{fig15}
\end{figure*}

In the PAAP mechanism, \Fig \ref{fig14a} shows the relationship among $\mathcal{R}$, $|\mathcal{J}|$ and $Nt$. \Fig \ref{fig14b} shows that if $|\mathcal{J}|=1$, the value of $\Theta$ is within $0.632117 \sim 0.632120$; if $|\mathcal{J}|=2,3,\cdots,10$, the value of $\Theta$ is within $0.632116 \sim 0.632120$. \Fig\ref{fig14b} and \Fig\ref{fig14c} show that the value of $\Theta$ is independent on $|\mathcal{J}|$ and $Nt$. \Fig \ref{fig15b} shows that if $|\mathcal{J}|=1$, the value of $\mathcal{R}$ is within $ 0.71451\sim 0.80123$; if $|\mathcal{J}|=2$, the value of $\mathcal{R}$ is within $ 0.73412\sim 0.80123$; whereas if $|\mathcal{J}|=3,4,\cdots,10$, the value of $\mathcal{R}$ is within $0.78321 \sim 0.82742$. \Fig\ref{fig15c} shows that the value of $\mathcal{R}$ is independent on $Nt$. \Fig \ref{fig14} and \Fig \ref{fig15} verify the derived competitive ratio ($\Theta$) given in \textcolor{Cerulean}{Proposition} \ref{theorem4}. The gap between $\mathcal{R}$ and $\Theta$  is within $0.082394\sim 0.1953$.

\subsubsection{Impact of booking flexibility}
\label{724}
In this subsection, we investigate the impact of booking flexibility in the MaaS system by comparing two types of rolling horizon configurations with the same time step and different time horizon lengths in the context of the PAYG mechanism: RHA ($\Delta t=1$, $\mathcal{T}=1$) and RHA ($\Delta t=1$, $\mathcal{T}=240$). The input data is the same for both configurations. Considering the booking feasibility in RHA ($\Delta t=1$, $\mathcal{T}=240$), we simulate users' booking behavior and set each user's requested departure time ($O_{i}$)  as a random number satisfying a symmetric triangular distribution within [$t$, $t+240$].  

\Fig \ref{fig16a} shows the relationship among social welfare ($W$), the number of bids ($|\mathcal{J}|$) and time slot ($t$) under RHA ($\Delta t=1$, $\mathcal{T}=1$) and $|\mathcal{J}|$ has little influence on $W$. \Fig \ref{fig16b} and \Fig \ref{fig16c} are projection plots of $t-|\mathcal{J}|-W$ figure under  RHA ($\Delta t=1$, $\mathcal{T}=1$) and RHA ($\Delta t=1$, $\mathcal{T}=240$), respectively. Compared with \Fig \ref{fig16b}, \Fig \ref{fig16c} show that the social welfare during non-peak hours (time slot 0$\sim$240) starts to increase in RHA ($\Delta t=1$, $\mathcal{T}=240$) configuration. Since the total travel demand remains unchanged, the social welfare increases during non-peak hours and decreases during peak hours. Namely, setting a longer time horizon lengths ($\mathcal{T}$) can improve users' booking flexibility and balance the social welfare between peak hours and non-peak hours. 
\begin{figure*}[ht!]
	\subfloat[$t-|\mathcal{J}|-W$ with RHA ($\mathcal{T} =1$)]{\includegraphics[width=0.33\textwidth]{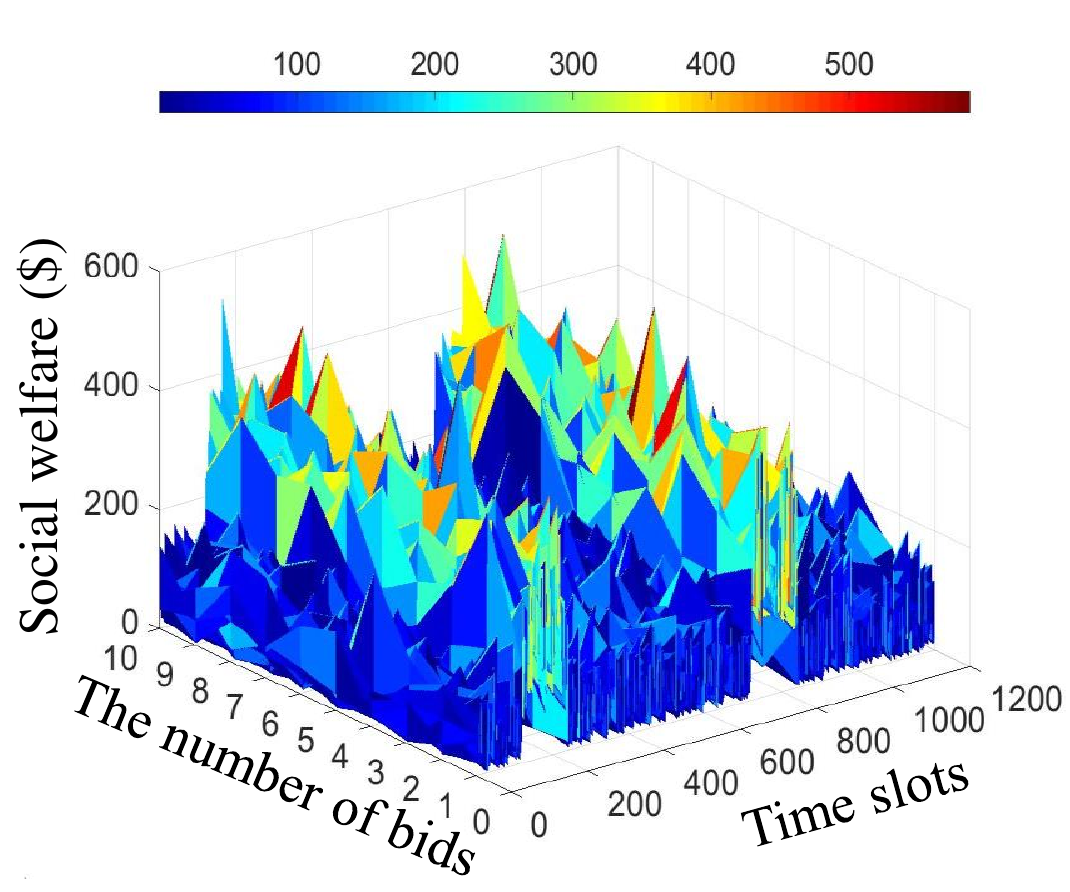}\label{fig16a}}
    \subfloat[$t-W$ with RHA ($\mathcal{T} =1)$]{\includegraphics[width=0.33\textwidth]{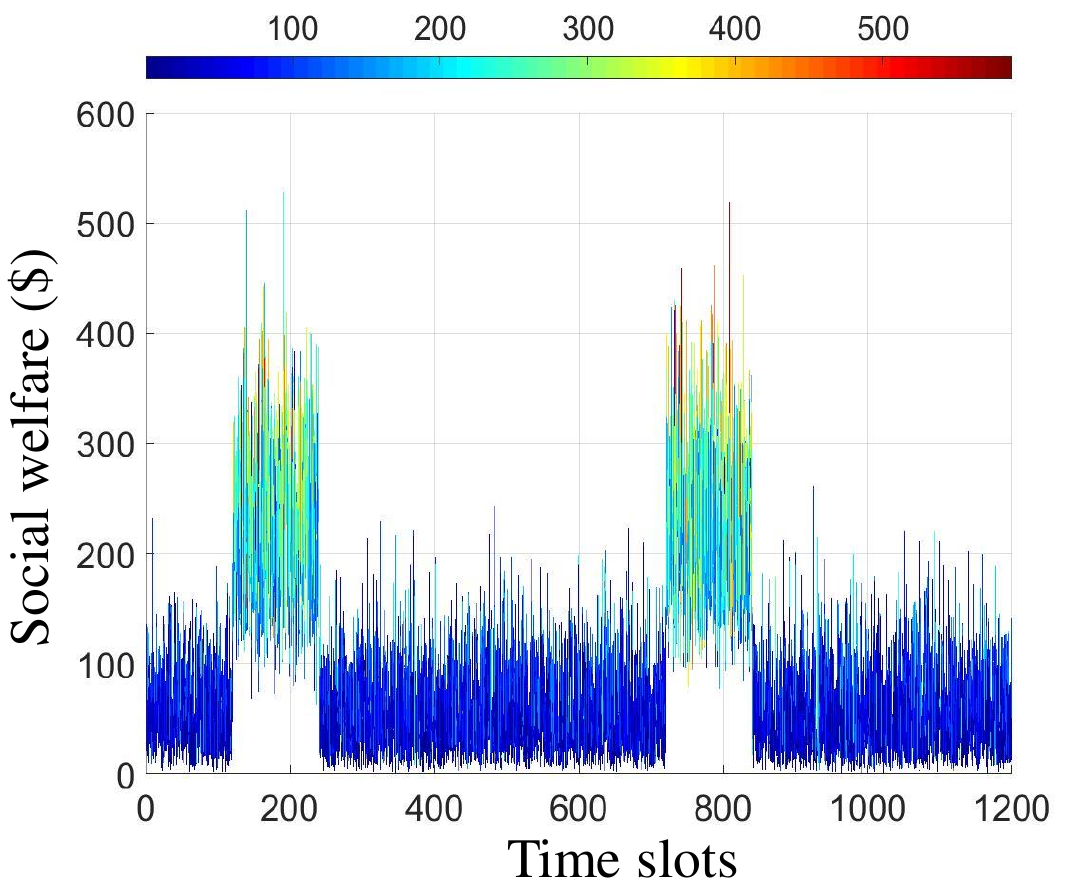}\label{fig16b}}
    \subfloat[$t-W$ with RHA ($\mathcal{T} =240)$]{\includegraphics[width=0.33\textwidth]{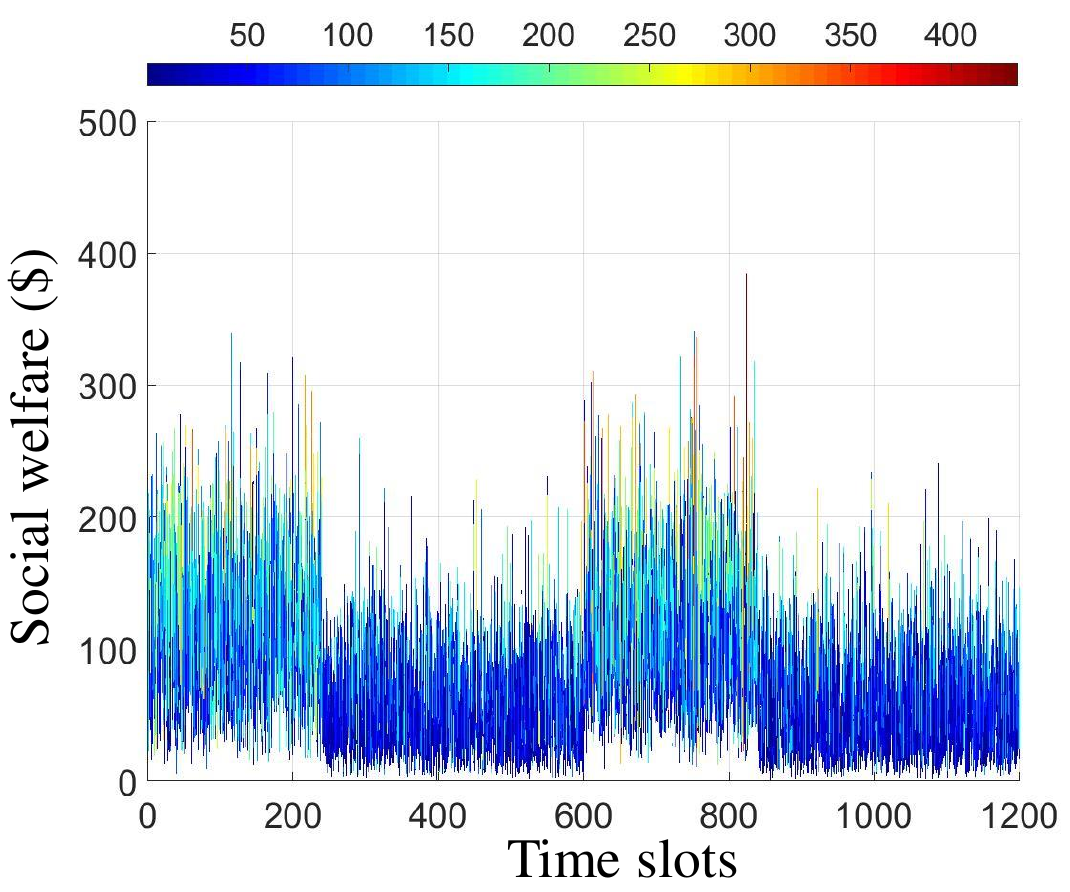}\label{fig16c}}
	\caption{RHA ($\Delta t$=1,$\mathcal{T}$=1 ) and RHA ($\Delta t$=1,$\mathcal{T}$=240)}\label{fig16}
\end{figure*}

\subsubsection{Comparison of rolling horizon algorithm configurations}
\label{725}
In this section, we conduct 100 Monte Carlo simulations to compare the mean value of computation time and social welfare obtained by the proposed online algorithms, online formulations and offline formulations under four types of rolling horizon algorithm (RHA) configurations  introduced in \s\ref{S6}. 

The results for the  PAYG mechanism are reported in \Fig\ref{fig18} based on the following four RHA (SHA) configurations:\\
\noindent \textbf{1})RHA: solve online algorithm (\Alg \ref{alg1}) based on RHA ($\Delta t=1$,$\mathcal{T}=240$);\\
\noindent \textbf{2})RHA: solve online MILP (\textcolor{Cerulean}{Model 1.1}) based on RHA ($\Delta t=1$,$\mathcal{T}=240$);\\
\noindent \textbf{3})RHA: solve offline MILP (\textcolor{Cerulean}{Model 2.1}) based on RHA ($\Delta t=10$, $\mathcal{T}=240$);\\
\noindent \textbf{4})SHA: solve offline MILP (\textcolor{Cerulean}{Model 2.1}) based on SHA ($\Delta t=240$, $\mathcal{T}=240$).\\

\Fig \ref{fig18a} and \Fig \ref{fig18b} show that the computation time of online and offline algorithms increase with the increasing number of time slots ($Nt$). The computational runtime of the offline algorithm increases exponentially with the number of time slots. 
\begin{figure*}[ht!]
\centering
	\subfloat[Computation time ]{\includegraphics[width=0.33\textwidth]{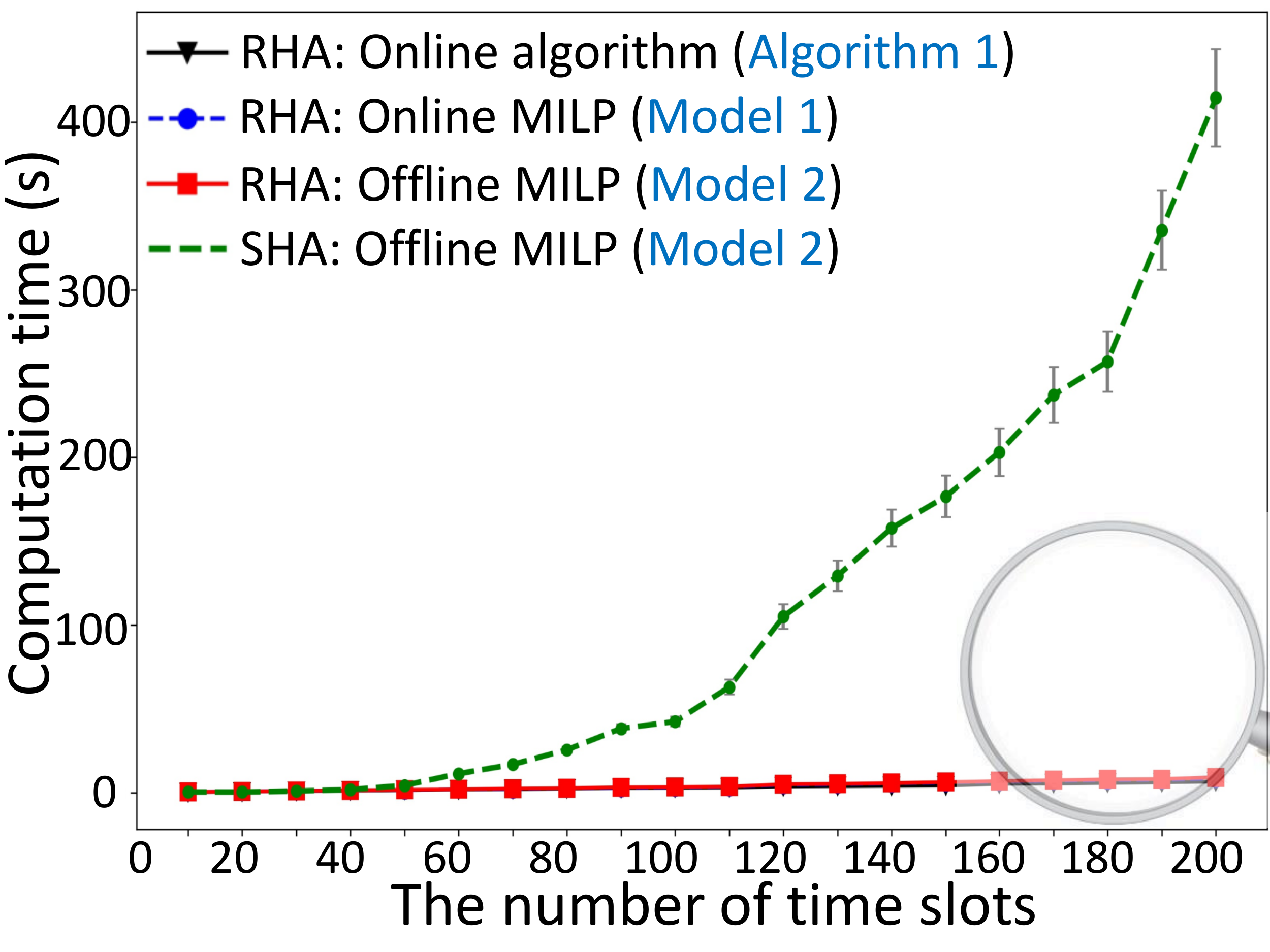}\label{fig18a}} 
    \subfloat[A closer look at online algorithms]{\includegraphics[width=0.33\textwidth]{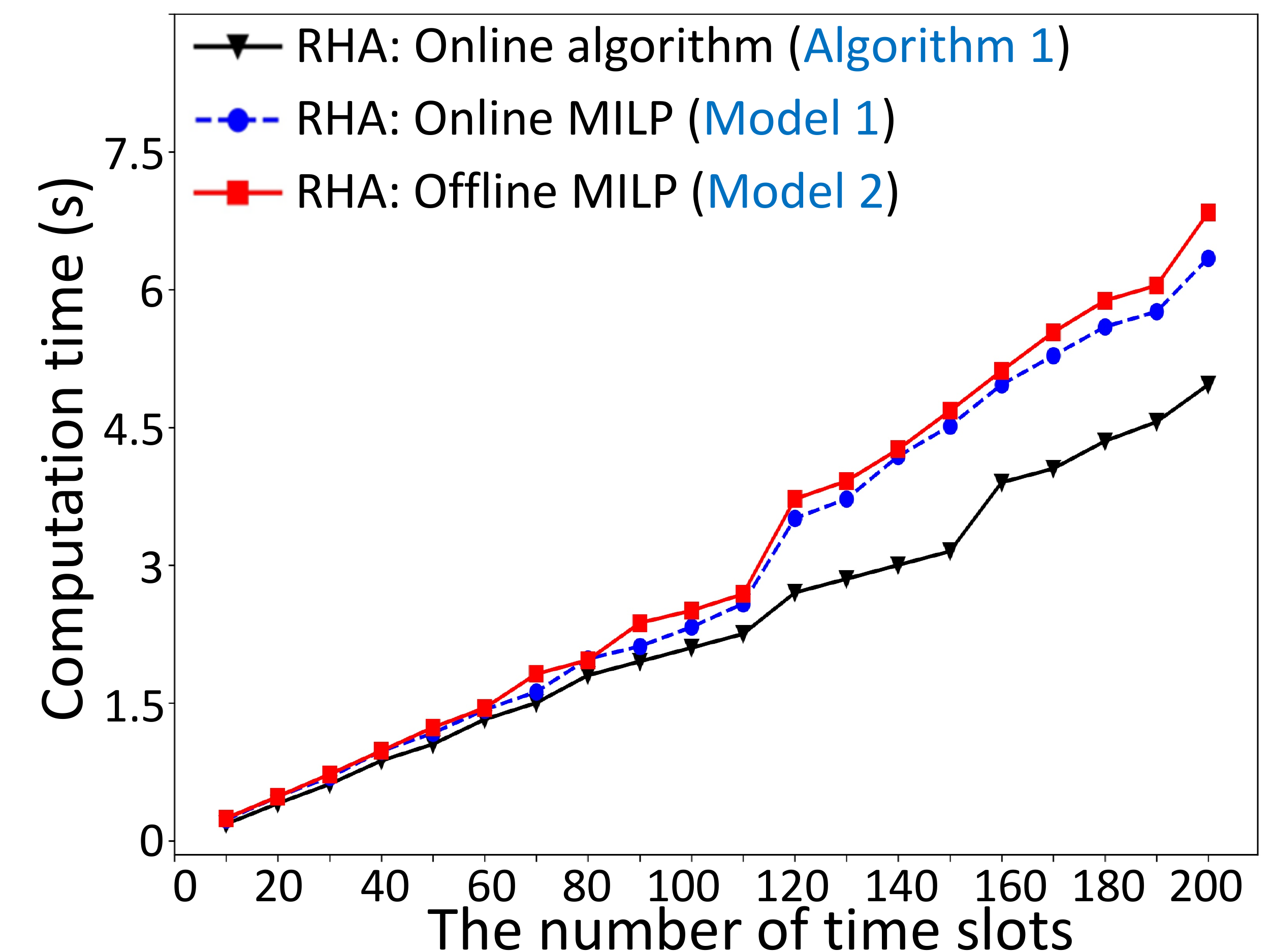}\label{fig18b}}
        \subfloat[Social welfare]{\includegraphics[width=0.33\textwidth]{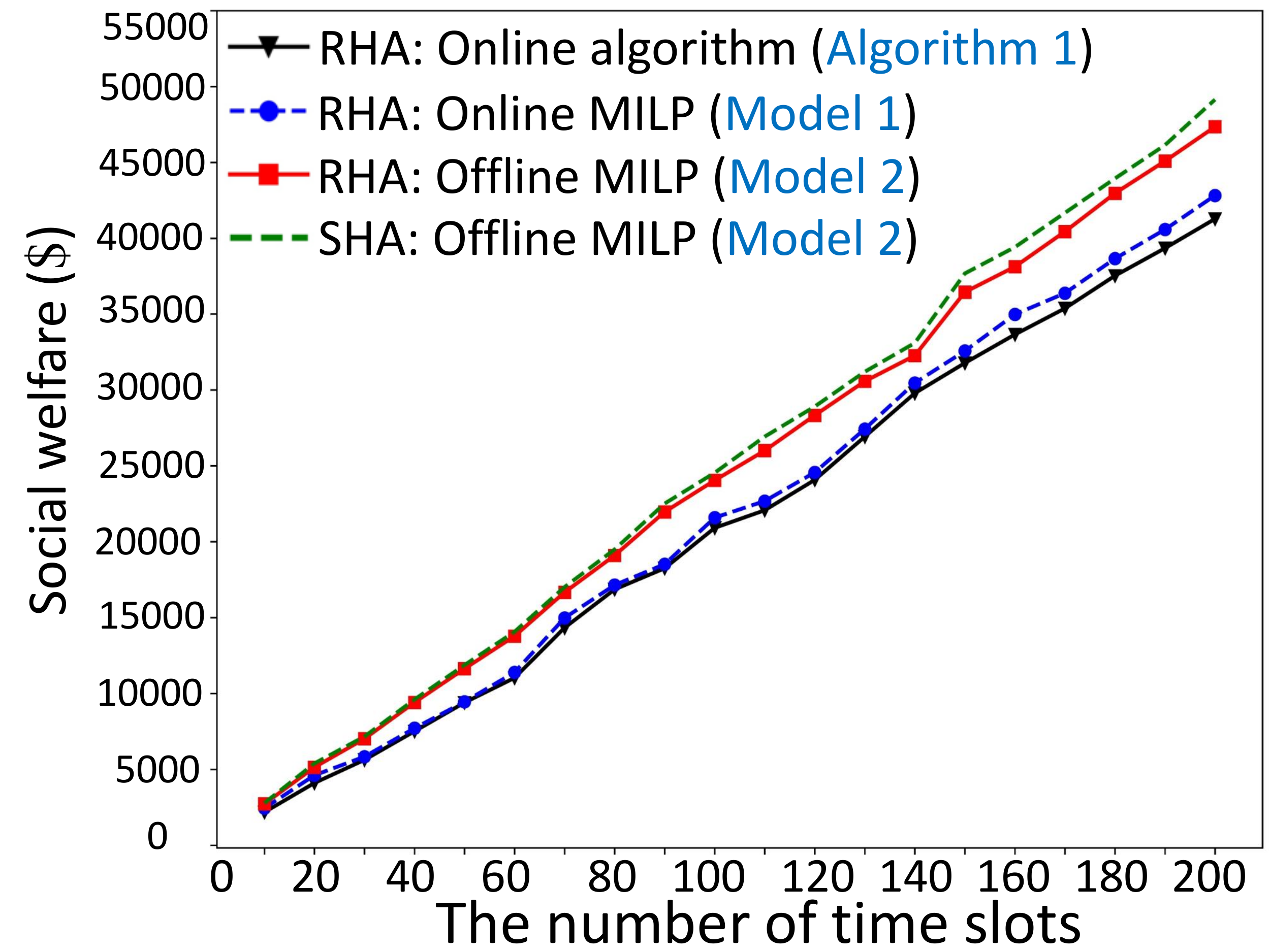}\label{fig18c}}
	\caption{Comparing the computation time and social welfare obtained by different configures in PAYG mechanism}\label{fig18}
\end{figure*}

Analogously, the results for the PAAP mechanism are reported in \Fig\ref{fig19} based on the following four RHA (SHA) configurations:\\
\noindent \textbf{1})RHA: solve online algorithm (\Alg \ref{alg2}) based on RHA ($\Delta t=1$, $\mathcal{T}=100$);\\
\noindent \textbf{2})RHA: solve online MILP (\textcolor{Cerulean}{Model 1.2}) based on RHA ($\Delta t=1$, $\mathcal{T}=100$);\\
\noindent \textbf{3})RHA: solve offline MILP (\textcolor{Cerulean}{Model 2.2}) based on RHA ($\Delta t=10$, $\mathcal{T}=100$);\\
\noindent \textbf{4})SHA: solve offline MILP (\textcolor{Cerulean}{Model 2.2}) based on SHA ($\Delta t=100$, $\mathcal{T}=100$).
\begin{figure*}[ht!]
\centering
	\subfloat[Computation time]{\includegraphics[width=0.33\textwidth]{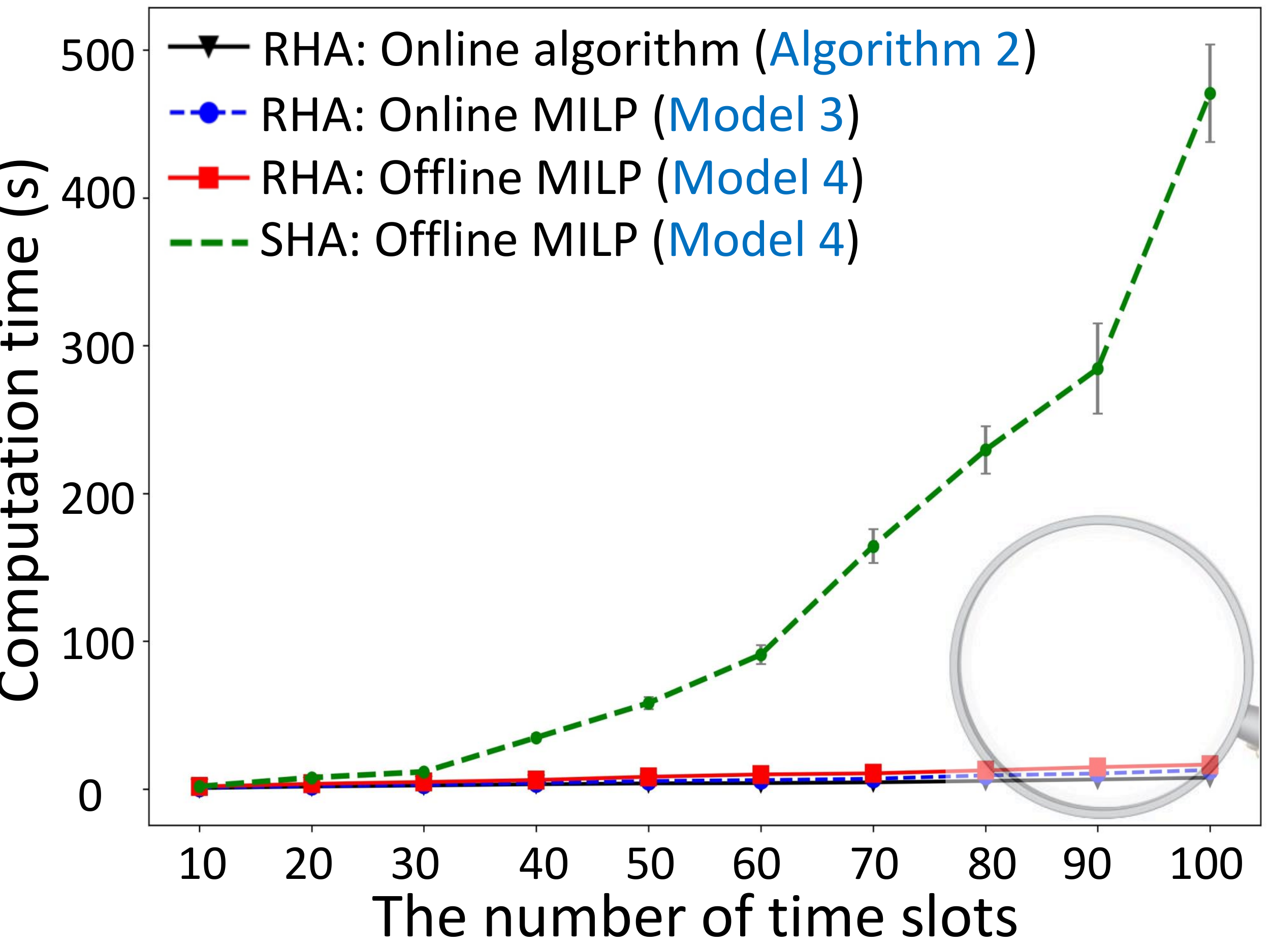}\label{fig19a}} 
    \subfloat[A closer look at  online algorithms]{\includegraphics[width=0.33\textwidth]{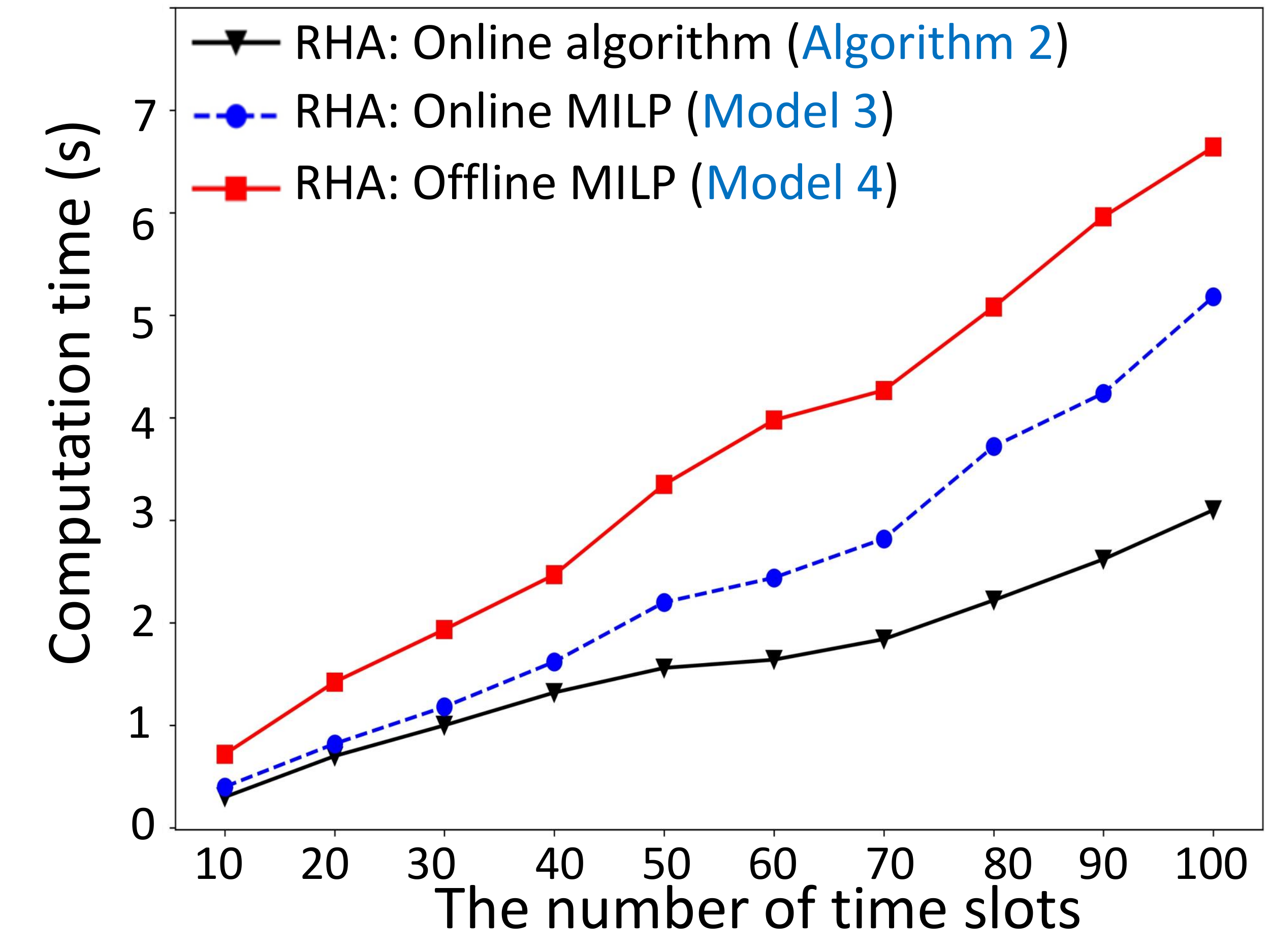}\label{fig19b}}
    \subfloat[ Social welfare]{\includegraphics[width=0.33\textwidth]{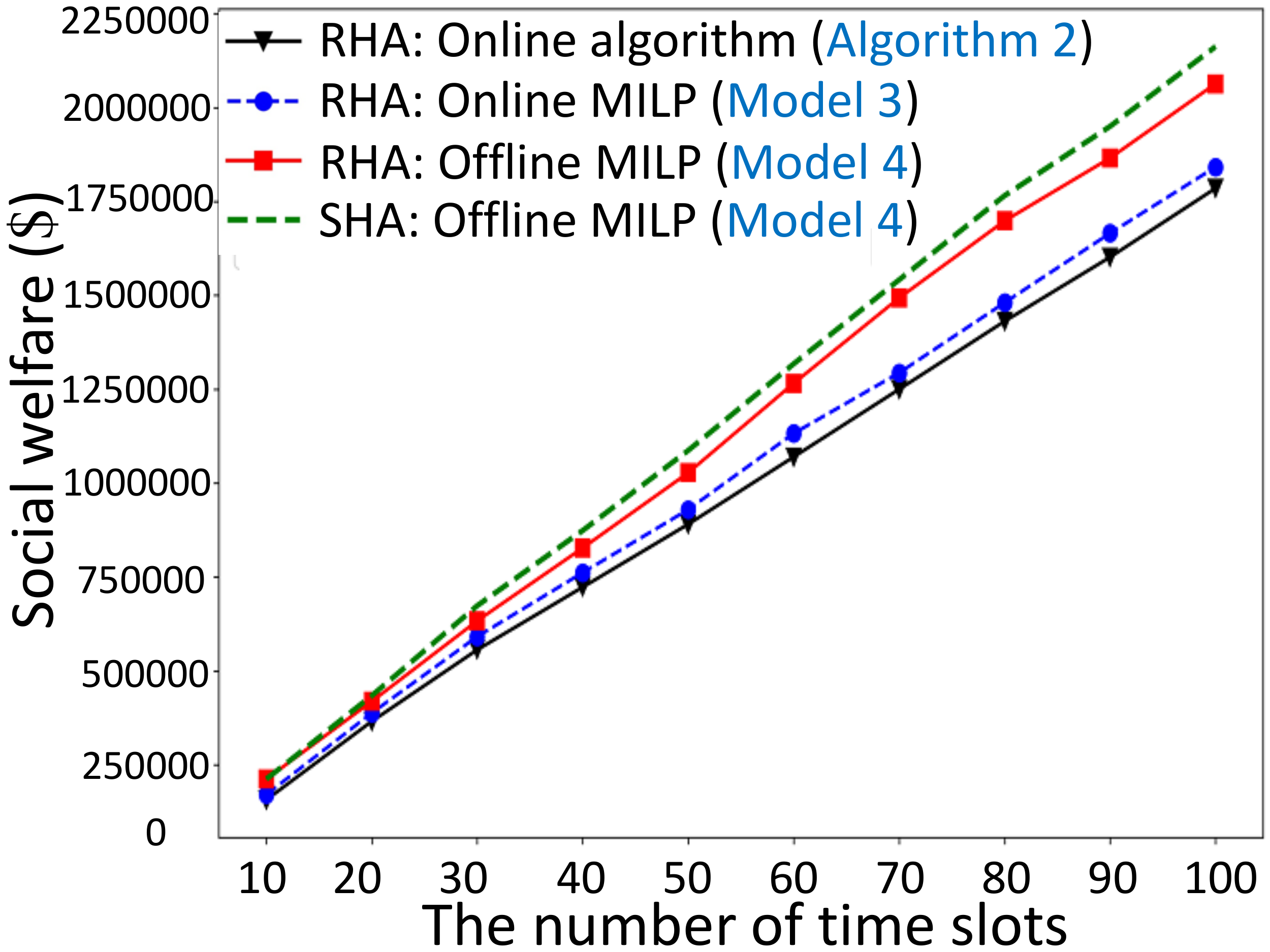}\label{fig19c}}
	\caption{Comparing the computation time and social welfare obtained by different configures in PAAP mechanism}
	\label{fig19}
\end{figure*}\\

\Fig\ref{fig18}  and \Fig\ref{fig19}  display the varying
trade-offs between runtime and solution quality of online and offline configurations. \Fig \ref{fig18a} and \Fig \ref{fig19a} show that the offline configuration runs in exponential time and takes significantly more time than online configurations. \Fig \ref{fig18c} and \Fig \ref{fig19c} show that the social welfare obtained by the offline configurations is considerably larger than that obtained by the online configurations, among which `RHA: offline MILP' can obtain the maximum social welfare. We observe that the social welfare obtained by \Alg \ref{alg1} (resp. \Alg \ref{alg2}) is only marginally lower than that obtained by the online MILP \textcolor{Cerulean}{Model 1.1} (resp. \textcolor{Cerulean}{Model 2.1}) when solved to optimality. In turn, \Alg \ref{alg1} and \Alg \ref{alg2} are faster compared to exact optimization methods. Overall, the rolling horizon configuration ($\Delta t=10, \mathcal{T}=240$) is shown to reduce the gap in social welfare observed between the online and offline algorithms, while retaining substantial computational efficiency. Compared with the optimal solution obtained by `SHA: offline MILP', we are willing to accept a suboptimal solution obtained by `RHA: offline MILP' to ensure that the solution remains feasible and near optimal when the computation time is required to significantly decrease.  

We conducted a sensitivity analysis on the time step ($\Delta t$) of the `RHA: offline MILP' configuration using the PAYG mechanism (the results in reported in \Fig\ref{H22} in the Appendix). Our findings show that when $\Delta t$ under RHA ($\Delta t>0$, $\mathcal{T}=240$) increases from 10 $\sim$ 200, in terms of social welfare, we have `SHA: offline MILP' $>$ `RHA: offline MILP' $>$ `RHA: online MILP' $>$ `RHA: online algorithm', which is consistent with the results shown in \Fig\ref{fig18}.

\section{Conclusion and remarks}
\label{S8}
We first summarize the main contributions of this work before discussing practical issues, and outlining future research directions.

\subsection{Conclusion}
This paper makes three broad methodological contributions. The first is to propose two types of online auction-based MaaS mechanisms (PAYG and PAAP), which can effectively and efficiently allocate mobility resources while accounting for users' preferences and WTP. We show that both mechanisms are incentive-compatible and maximize social welfare. The second is to propose novel mathematical programming formulations for the online mobility resource allocation problems that arise within the PAYG and the PAAP mechanisms. These formulations allow users to bid for any amount of mode-agnostic travel times in the MaaS system which are converted into mobility resources expressed in units of speed-weighted travel distance. Third, we show that these online resource allocation problems can be reformulated as multidimensional knapsack problems and propose polynomial-time, customized primal-dual algorithms with bounded competitive ratios for implementing the PAYG and PAAP mechanisms. 
Our numerical results uncover insights into how the proposed mechanisms behave in relation to the parameters, highlight the performance of the proposed online algorithms, as well as the benefits obtained through the proposed RHA configurations. Comparing the PAYG and the PAAP mechanisms, we find that the PAAP mechanism outperforms the PAYG in terms of improving social welfare and reducing the unit price of mobility resources. The competitive ratio ($\Theta$) of \Alg \ref{alg1} (resp. \Alg \ref{alg2}) can provide a lower bound for the social welfare ratio ($\mathcal{R}$) in the online PAYG (resp. PAAP) mechanism. Further, we find that RHA configurations with larger time horizons can improve users' booking flexibility. We also find that the `RHA: offline MILP' configuration provides a good compromise between solution quality and computational scalability.

\subsection{Remarks and future research}
The global economic transition and state-of-the-art technologies are driving the transformation of the transport sector from an infrastructure/manufacturing focused industry to a service/experience focused industry \citep{AndyHong}. In line with the transition from a focus on `products' to `service' to `user experience', we proposed an innovative MaaS paradigm emphasizing the nature of service nature and user experience. In this paradigm, users can bid for mobility resources in a continuous fashion and expect multi-modal mobility services tailored to their willingness to pay (WTP) and travel requirements. 

\cite{sochor2018topological} proposed a four-level taxonomy to divide different MaaS schemes, current MaaS schemes such as UbiGo have not reached Level 3 (integration of the service offer). In comparison, the proposed MaaS paradigm aims to reach the highest level (integration of societal goals). Although the proposed paradigm may provide several advantages as discussed above, there remains several practical issues to address before such MaaS systems can be deployed. At the micro level, users' habits and attitudes are recognised as essential factors. In the proposed MaaS paradigm, users need to quantify the abstract characters (e.g., inconvenience tolerance and travel delay budget) and report them to the MaaS regulator, which might be difficult to identify in practice. \cite{karlsson2020development} showed that it is difficult to change people's travel behaviour due to the  established habits and the individual's perceived `action space'; thus future research will investigate user adoption and attempt to quantify to which degree are users willing to change their travel habits. At the macro level, \cite{merkert2020collaboration} identified the importance of system integration and the elimination of the influence of boundary effects on different modes in a system, and indicated that the combined operation of MaaS systems across public and private travel options may increase the pressure upon TSPs to provide multimodal and seamless services. Thus the potential collaboration different stakeholders in the MaaS system need to be further investigated.

This paper has taken a first step towards designing mechanisms for the operation of MaaS systems. However, market dynamics and other complex factors in MaaS systems are not accounted for in the proposed approach motivating future research in this direction. In particular, two-sided economic deregulated MaaS markets, or hierarchical configurations such as Stackelberg competition between TSPs and users may provide more realistic configurations for the development of emerging MaaS ecosystems.

\section*{Acknowledgments}
The authors wish to express their thanks to Dr. Haris Aziz from the University of New South Wales and Prof. David Hensher from the University of Sydney for their useful suggestions. This research was partially supported by the Australian Government through the Australian Research Council's Discovery Projects funding scheme (DP190102873). Dr. Wei Liu acknowledges the support from Australian Research Council through the Discovery Early Career Researcher Award (DE200101793).

\appendix
\section{Mathematical notations}
\label{parameters}
\begin{center}
\setlength{\abovecaptionskip}{0pt}
\setlength{\belowcaptionskip}{0pt}
\centering
\footnotesize
\begin{longtable}{ll}
\toprule
\multicolumn{2}{l}{Variables}\\
\hline $l_{ij}^m$  & Real variable denoting the the travel time served by mode $m$ in the MaaS bundle for user $i$'s bid $j$\\
$p_{t}$ &  Parameter/real variable denoting the unit price at time slot $t$ in the online/offline resource allocation problem\\
$p_{ij}^{t}$ & Real variable denoting the actual payment of user $i$'s bid $j$ at time slot $t$\\
\multirow{2}{*} {$x_{ij }$} & Binary variable denoting whether user $i$'s bid $j$ is accepted (PAYG)\\
 & Continuous variable denoting the fraction of resources  allocated to user $i$'s bid $j$ (PAAP)\\
\multirow{2}{*} {$x_{ij}^{t}$} & Binary variable denoting whether user $i$'s bid $j$ is accepted at time slot $t$ (PAYG)\\
& Continuous variable denoting the fraction allocated to user $i$'s bid $j$ at time slot $t$ (PAAP)\\
\multirow{2}{*}{$\chi_{i,s}$} &Binary variable denoting whether MaaS bundle $s$ is allocated to user $i$ (PAYG)\\
& Real variable denoting the fraction  allocated to user $i$'s MaaS bundle $s$ (PAAP)\\
\multirow{2}{*}{$\chi_{i,s}^{t}$} &Binary variable denoting whether MaaS bundle  $s$ is allocated to user $i$ at time slot $t$ (PAYG)\\
 & Real variable denoting the fraction  allocated to user $i$'s MaaS bundle $s$ at time slot $t$ (PAAP)\\
\hline \multicolumn{2}{l}{Parameters}\\
\hline 
$A_{t}$ & The weighted quantity of available mobility resources at time slot $t$\\
$b_{ij}$ & Bidding price of user $i$'s bid $j$ \\
$b_{i,s}$ & Bidding price of user $i$'s MaaS bundle $s$ \\
$C$& The weighted capacity of the mobility resources in each time slot\\
$D_{i}$ & user $i$'s shortest travelling distance arranged by MaaS regulator based on Origin destination information.\\
$N_{i}$ & The total  travel time slots of user $i$'s  MaaS bundle (PAYG)\\
$L_{i}$ & user $i$'s pre-defined  time period (slots) for the MaaS package bundle (PAAP)\\
$O_{i}$ & user $i$'s requested departure time in  PAYG mechanism and  start date in the PAAP mechanism \\
$p_{ij}$ & The actual payment of user $i$'s bid $j$, which is a constant in each time slot\\
$Q_{ij}$ & speed-weighted travel distance: distance weighted by  user $i$'s requested speed for bid $j$\\
$Q_{i,s}$ & speed-weighted travel distance of user $i$'s MaaS bundle $s$\\
$T_{ij}$ & The requested travel time of user $i$'s bid $j$ \\
$v_{m}$ & Average commercial speed of transport mode $m$ \\
$\Gamma_{i}$ &The maximum inconvenience degrees that user $i$ can tolerate during a service\\
$\sigma_{m}$ & inconvenience cost per unit of time for travel mode $m$\\
$\Phi_{i}$ & user $i$'s travel delay budget: maximum delay that user $i$ can accept during a service\\
\bottomrule
\end{longtable}
\label{A}
\end{center}
\section{Sensitivity analysis on the time step in terms of social welfare}
\label{EE}
\begin{figure*}[ht!]
\subfloat[$\Delta t=10$]{\includegraphics[width=0.3\textwidth]{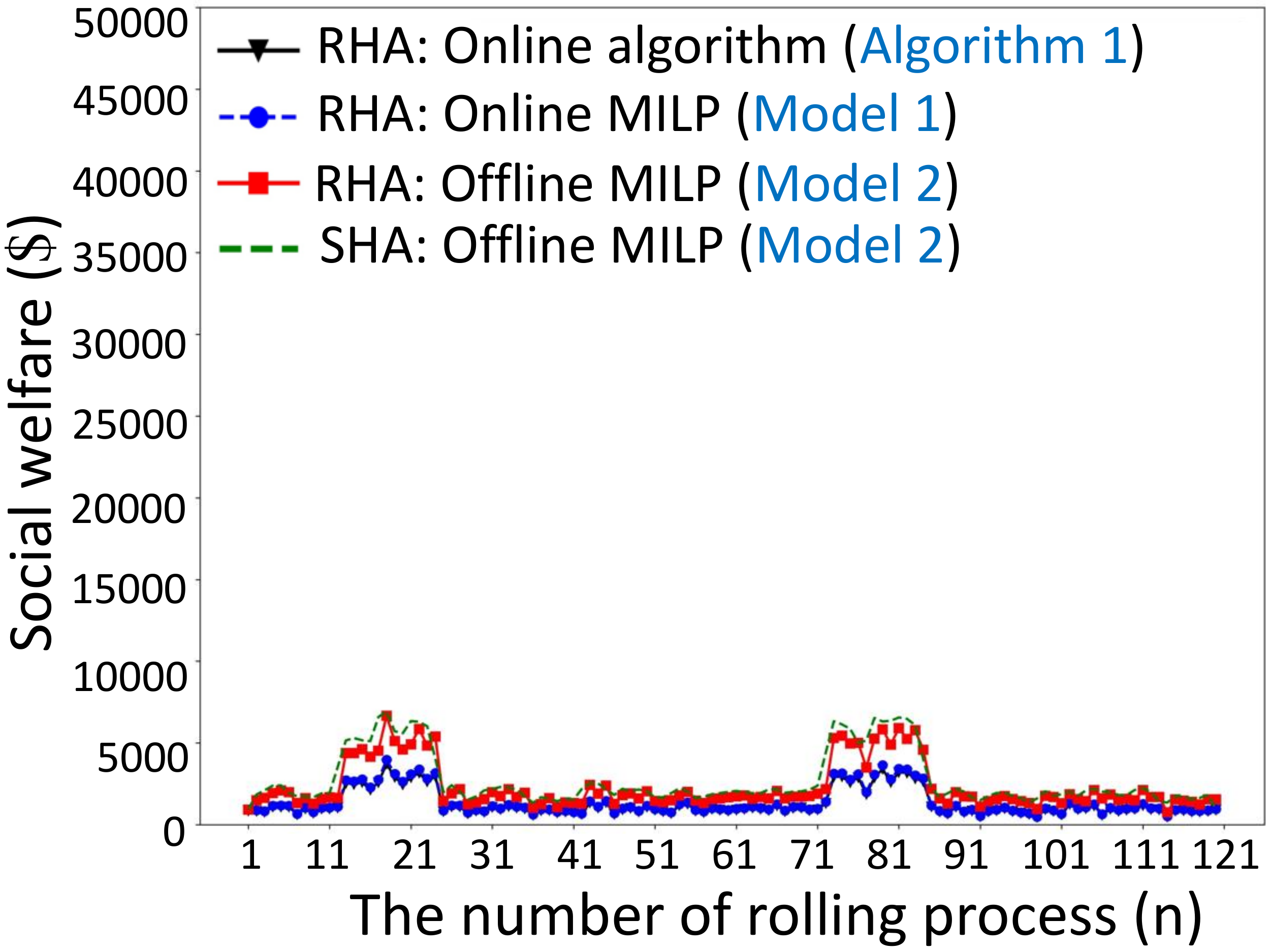}\label{Nt1}}
\subfloat[$\Delta t=30$]{\includegraphics[width=0.3\textwidth]{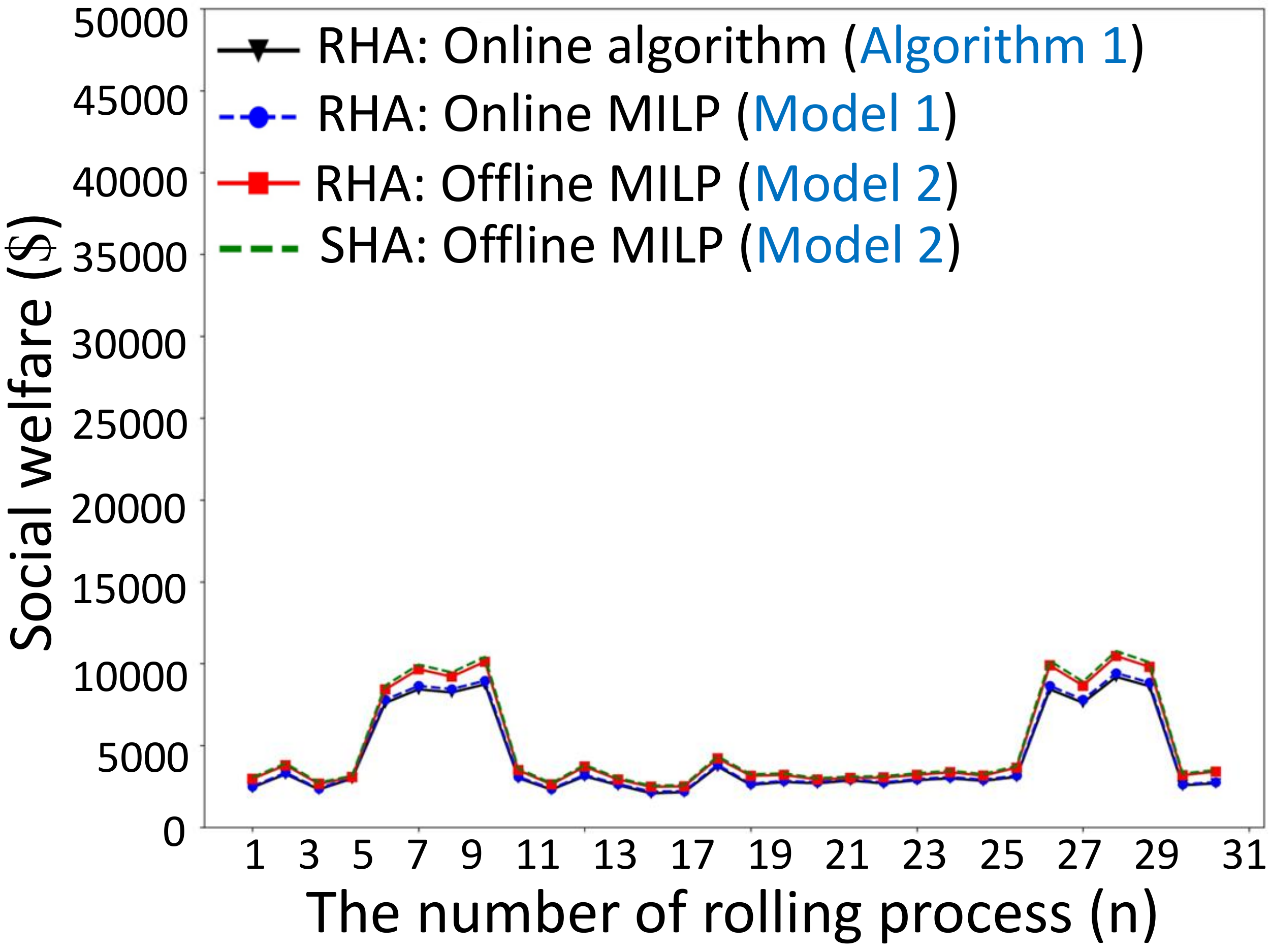}\label{Nt2}}
\subfloat[$\Delta t=50$]{\includegraphics[width=0.3\textwidth]{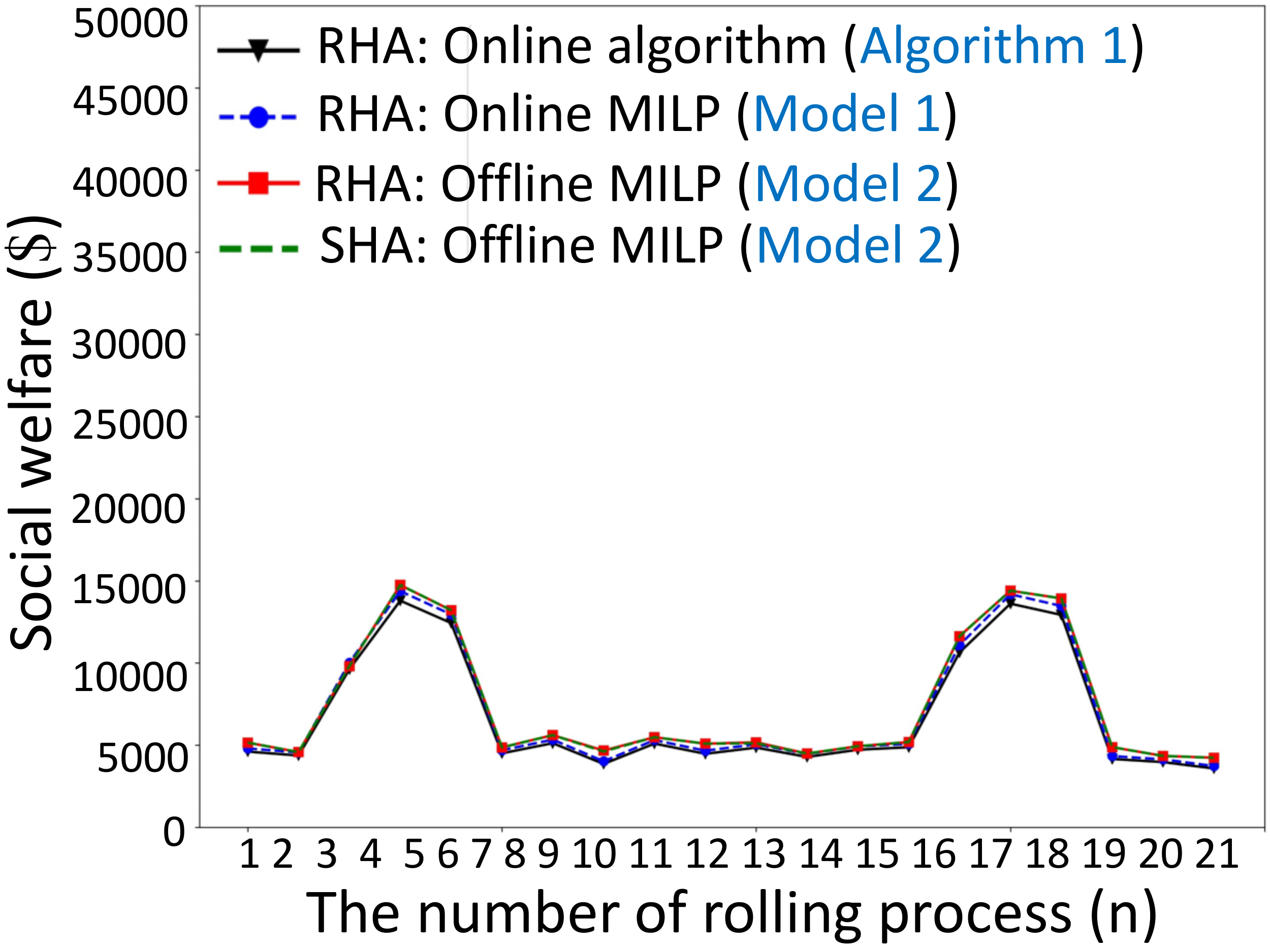}\label{Nt3}}\\
\subfloat[$\Delta t=70$]{\includegraphics[width=0.3\textwidth]{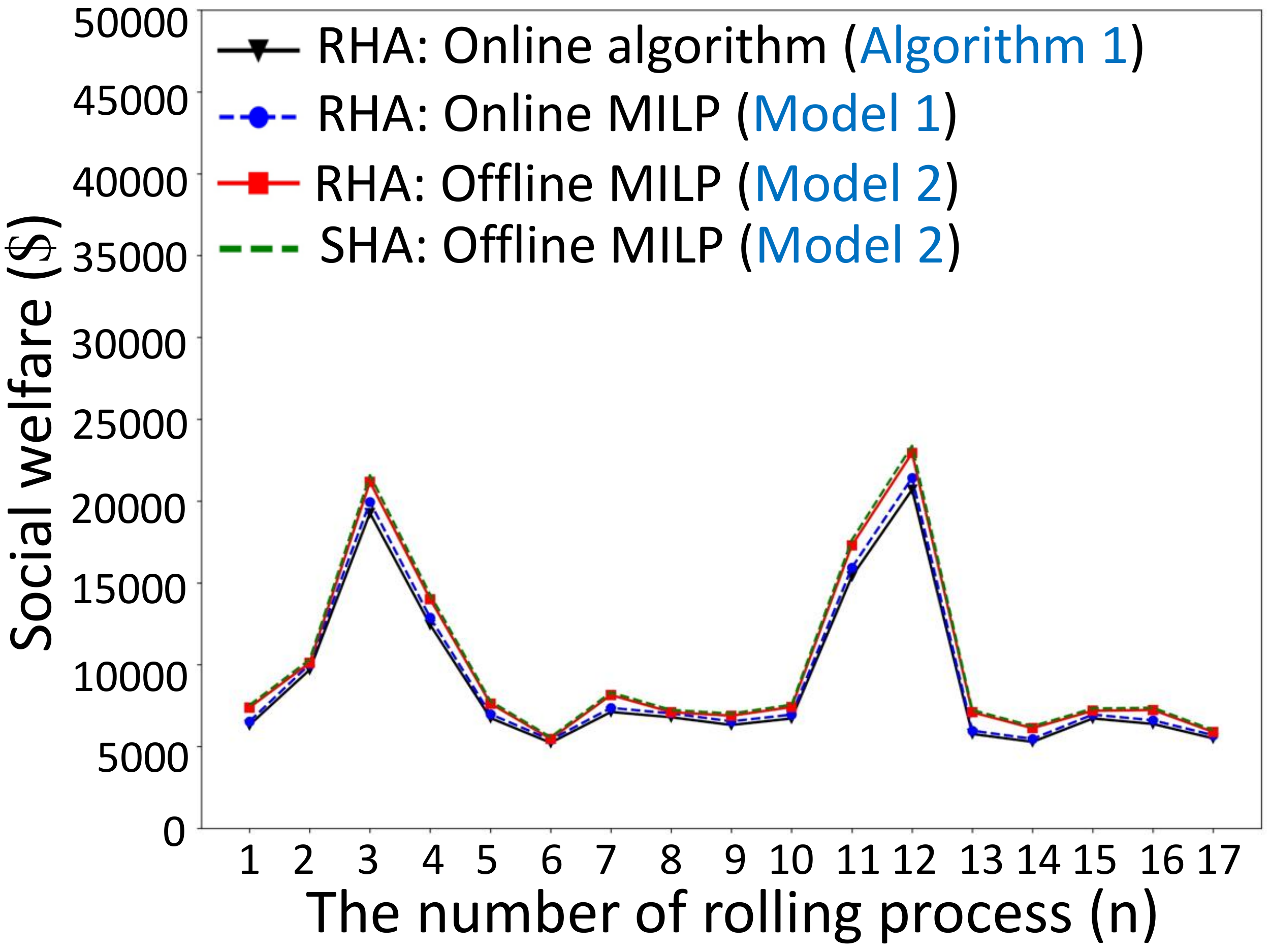}\label{Nt4}}
\subfloat[$\Delta t=90$]{\includegraphics[width=0.3\textwidth]{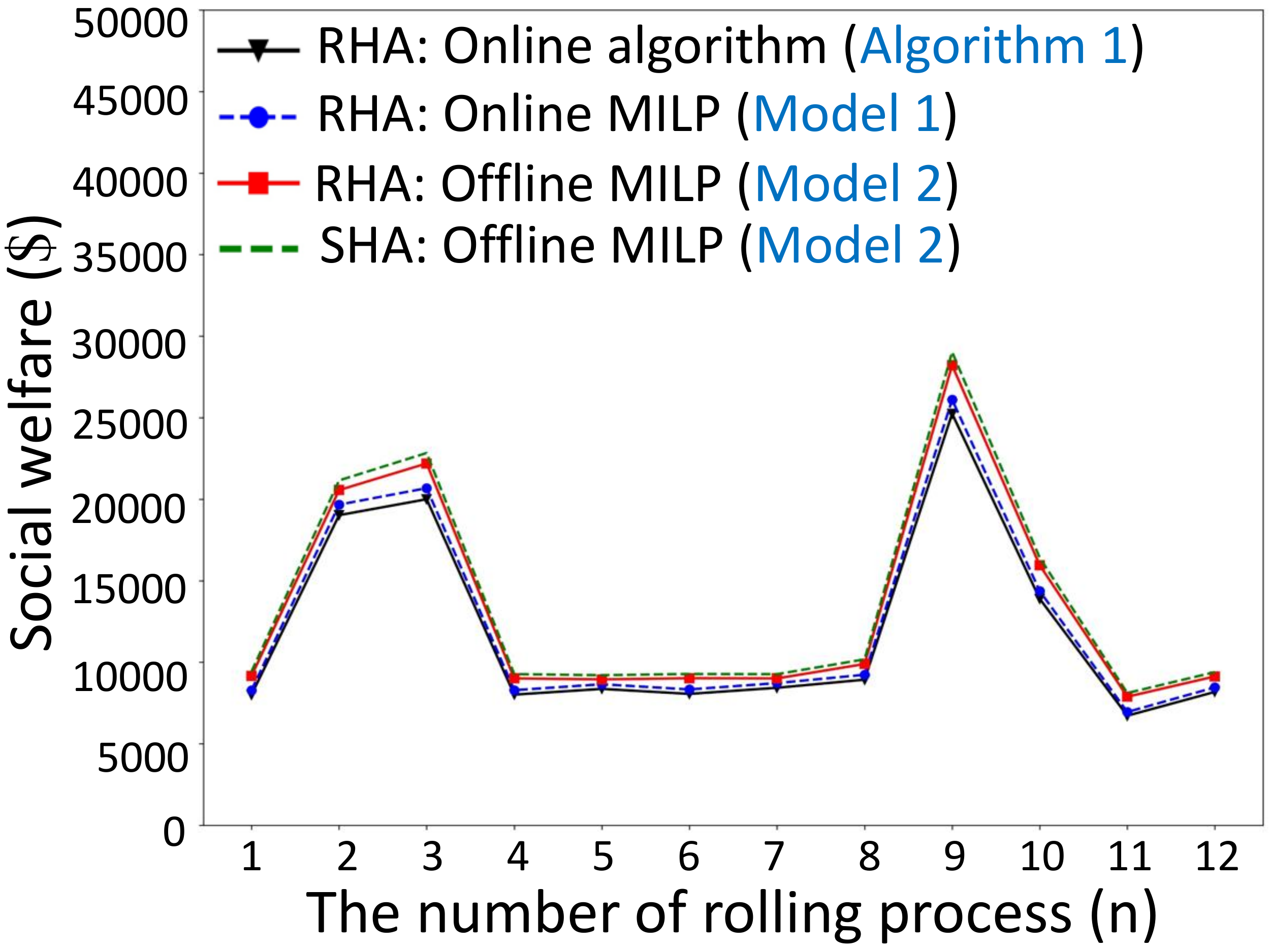}\label{Nt5}}
\subfloat[$\Delta t=110$]{\includegraphics[width=0.3\textwidth]{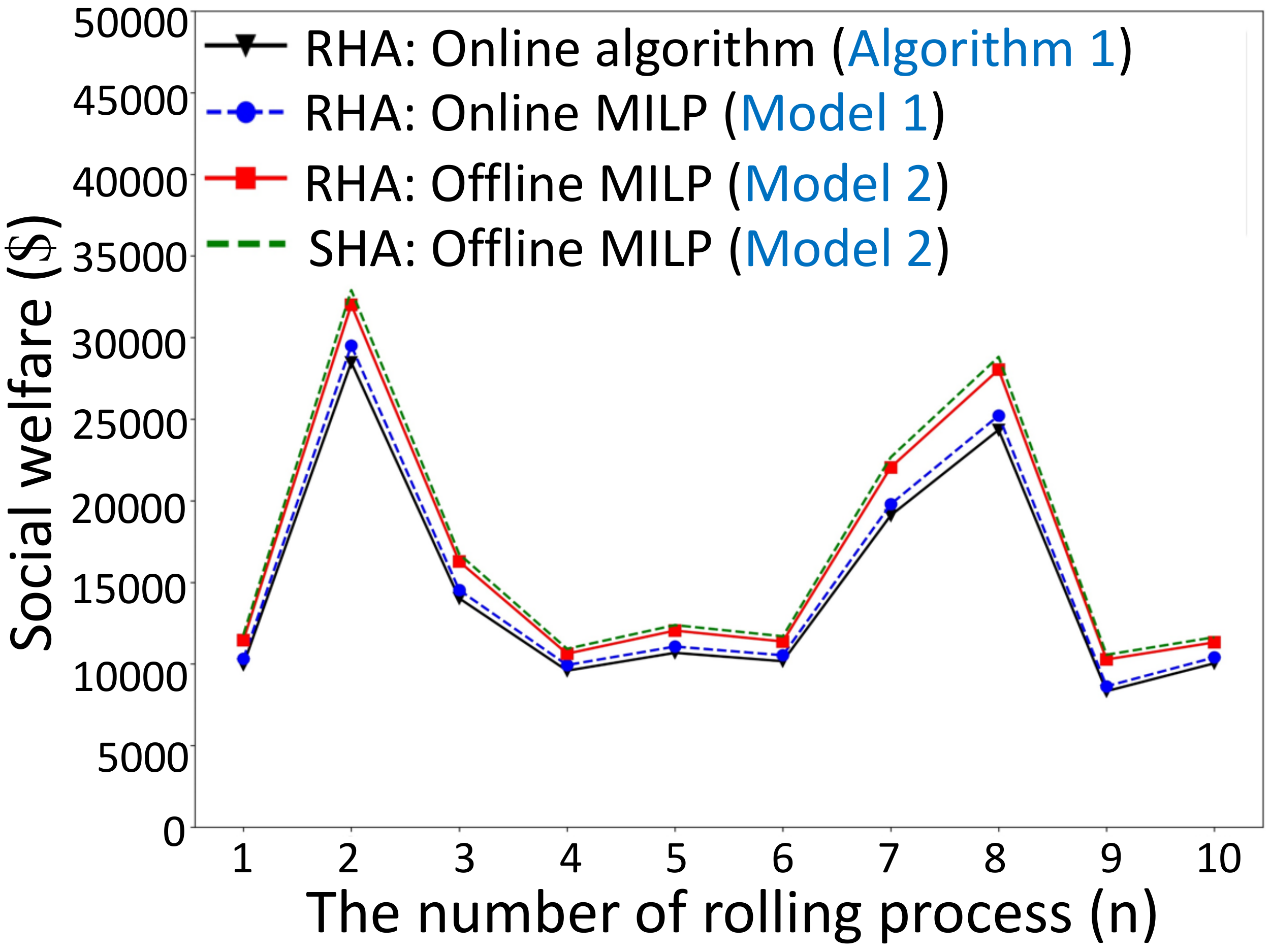}\label{Nt6}}\\
\end{figure*}
\addtocounter{figure}{-1}     
\begin{figure*}[ht!]
\addtocounter{figure}{1}
\addtocounter{subfigure}{6}
\subfloat[$\Delta t=130$]{\includegraphics[width=0.3\textwidth]{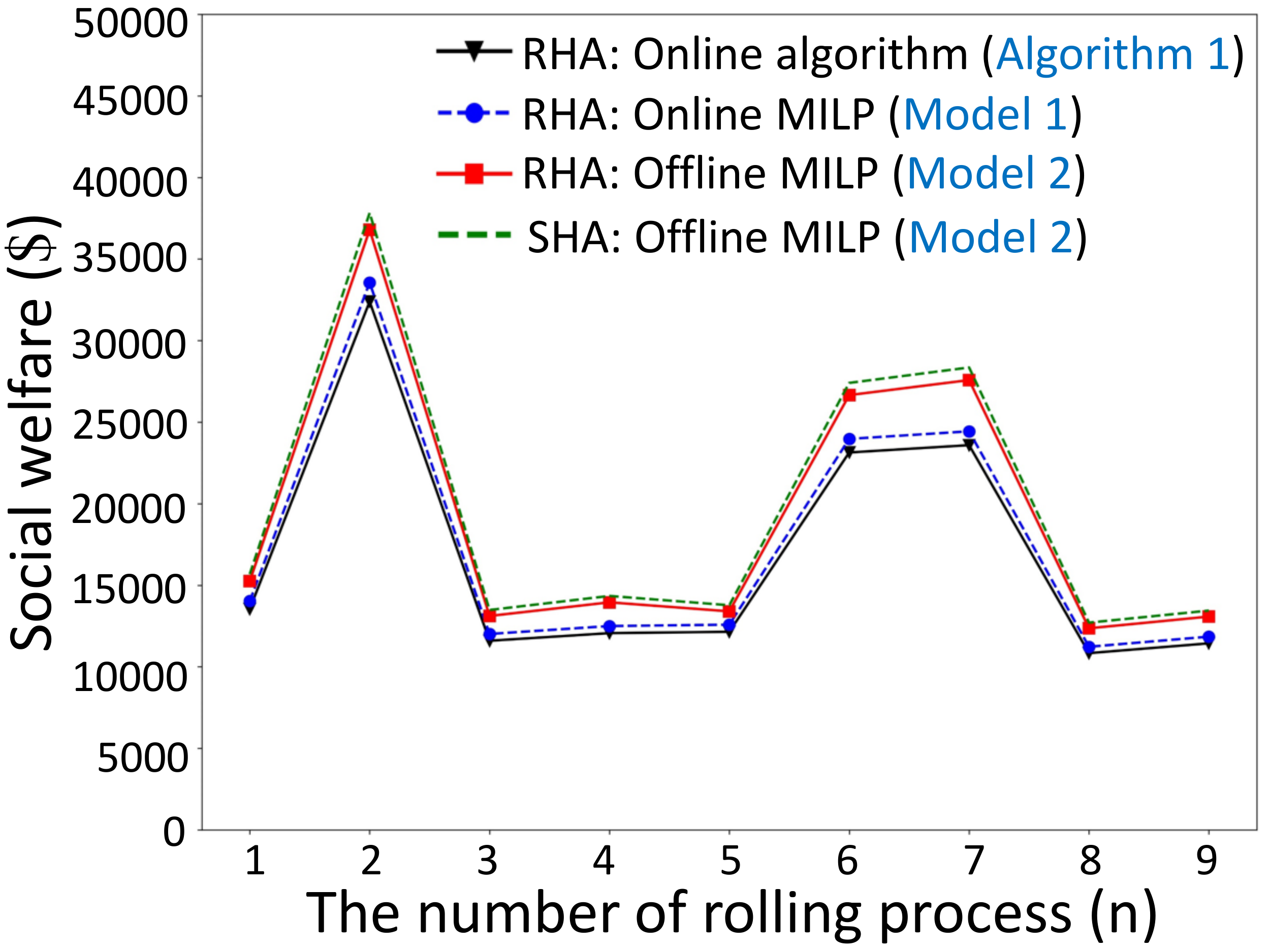}\label{Nt7}}
\subfloat[$\Delta t=170$]{\includegraphics[width=0.3\textwidth]{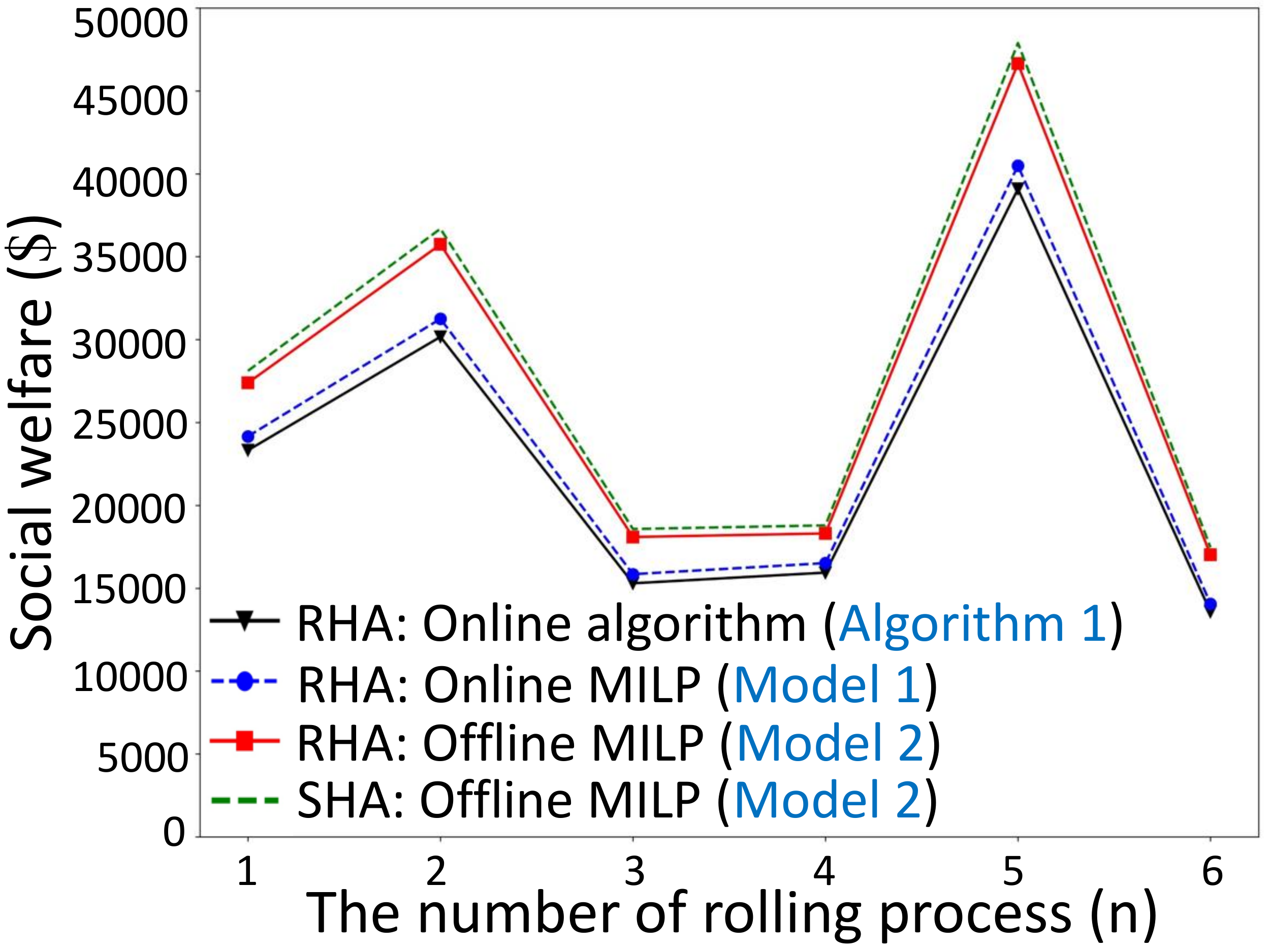}\label{Nt8}}
\subfloat[$\Delta t=200$]{\includegraphics[width=0.3\textwidth]{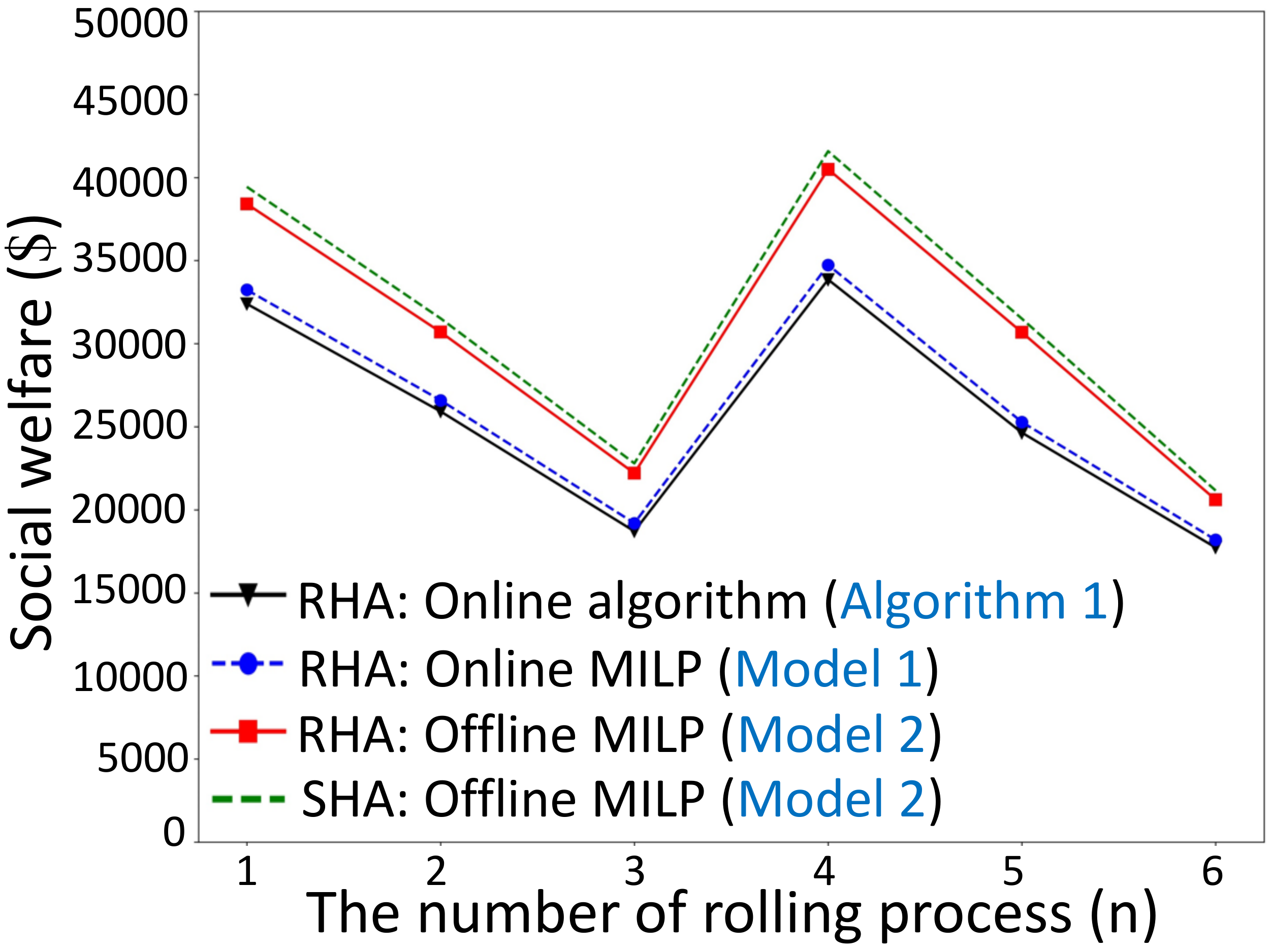}\label{Nt9}}\\
\caption{Sensitivity analysis on the time step ($\Delta t$) in terms of social welfare in the PAYG mechanism.}
\label{H22}
\end{figure*}

\newpage
\bibliographystyle{elsarticle-harv}
\bibliography{bibliography.bib}
\end{document}